\documentclass[a4paper, reqno, titlepage, twoside, 11pt]{article}
\usepackage[utf8]{inputenc}
\usepackage{fancyhdr}
\usepackage{amsthm}
\usepackage{amssymb}
\usepackage[american]{babel}
\usepackage{exscale}
\usepackage[mathscr]{euscript}   
\usepackage{url}
\usepackage[all,ps,tips,tpic]{xy}
\usepackage{color}
\usepackage{bbm}
\usepackage{bbold}   
\usepackage{a4}
\usepackage{fancyhdr}
\usepackage{mathdots}   
\usepackage{mathrsfs}
\usepackage{MnSymbol}
\fancypagestyle{plain}{
\fancyhf{}
\cfoot{\thepage}
 
  }

\newtheorem{lem}{Lemma}%[section]

\newtheorem{definition}[lem]{Definition}

\newtheorem{deflem}[lem]{Definition/Lemma}

\newtheorem{cor}[lem]{Corollary}

\newtheorem{prop}[lem]{Proposition}

\newtheorem*{defin}{Definition}

\newtheorem*{THM}{Theorem}
\newtheorem*{Lemma}{Lemma}

\theoremstyle{remark}

\DeclareMathOperator{\id}{id}
\DeclareMathOperator{\Name}{Name}
\DeclareMathOperator{\dom}{dom}

\DeclareMathOperator{\Ord}{Ord}
\DeclareMathOperator{\OR}{OR}
\DeclareMathOperator{\SUPP}{supp}
\DeclareMathOperator{\cf}{cf}

\DeclareMathOperator{\Card}{Card}
\DeclareMathOperator{\HT}{ht}
\DeclareMathOperator{\Succ}{Succ}

\DeclareMathOperator{\supp}{supp}
\DeclareMathOperator{\rg}{rg}
\DeclareMathOperator{\im}{im}

\newcommand{\uhr}{\hspace*{-0,5mm} \upharpoonright \hspace*{-0,5mm} }

\newcommand{\m}{\mathbb}
\newcommand{\tbl}{\textquotedblleft}
\newcommand{\tbr}{\textquotedblright}

\newcommand{\ol}{\overline}
\newcommand{\wt}{\widetilde}

\begin{document}

\title{An Easton-like theorem for Zermelo-Fraenkel Set Theory without Choice}
\begin{center} \Large  \bfseries An Easton-like theorem for \\ Zermelo-Fraenkel Set Theory without Choice \end{center}

\begin{center} { \scshape Anne Fernengel and Peter Koepke } \end{center}

\vspace*{3mm}

\begin{center} \bfseries Abstract \end{center}
\begin{quote} \small We show that in Zermelo-Fraenkel Set Theory without the Axiom of Choice a surjectively modified continuum function $\theta (\kappa)$ can take almost arbitrary values for all infinite cardinals.  
This choiceless version of Easton's Theorem is in sharp contrast to the situation in $ZFC$, where for singular cardinals $\kappa$, the value of $2^\kappa$ is strongly influenced by the behavior of the continuum function below. \\
Our construction can roughly be described as follows: 
In a ground model $V \vDash ZFC\, \plus\, GCH$ with a \tbl reasonable\tbr\, function $F: \Card \rightarrow \Card$ on the 
infinite cardinals, 
a class forcing $\m{P}$ is introduced, which blows up the power sets of all cardinals according to $F$. The eventual model $N \vDash ZF$ is a symmetric extension by $\m{P}$ such that $\theta^N (\kappa) = F(\kappa)$ holds for all $\kappa$.

\end{quote}

\vspace*{3mm}

\paragraph{Introduction.}

In 1970, William Easton proved that for regular cardinals $\kappa$, any reasonable behavior of the $2^\kappa$-function is consistent with $ZFC$ (\cite{Easton}). \\ For singular cardinals $\kappa$ though, 
the situation is a lot more involved, since then the value of $2^\kappa$ is strongly influenced by the behavior of the continuum function below. For instance, the {\textit{Singular Cardinals Hypothesis}} ($SCH$) implies that for any singular cardinal $\kappa$ with $2^\nu < \kappa$ for all $\nu < \kappa$, it already follows that $2^\kappa = \kappa^\plus$. Its negation's consistency strength was determined by Motik Gitik in \cite{G1} and \cite{G2} to be the existence of a measurable cardinal $\lambda$ with Mitchell order $\sigma (\lambda) = \lambda^{\plus \plus}$. \\[-4mm]

Silver's Theorem (\cite{Silver}) states that for any cardinal $\kappa$ of uncountable cofinality with $2^\nu = \nu^\plus$ for all $\nu < \kappa$, it already follows that $2^\kappa = \kappa^\plus$. Hence, the $SCH$ holds if it holds for all singular cardinals of countable cofinality. In particular, Easton's Theorem can not be generalized to singular cardinals. \\ Another prominent result concerning upper bounds on the continuum function for singular cardinals is the following theorem by Shelah (\cite{Shelah}): 
\[ \mbox{If } 2^{\aleph_n} < \aleph_\omega \mbox{ for all } n < \omega, \mbox{ then } 2^{\aleph_\omega} < \aleph_{\omega_4}. \]
Without the Axiom of Choice, however, there is a lot more possible.
One has to distinguish between {\textit{injective}} and \textit{surjective} failures of the $SCH$: From Theorem 2 in \cite{AK} it follows that the $GCH$ below $\aleph_\omega$ together with an injective map $f: \lambda \rightarrow \powerset(\aleph_\omega)$ for high $\lambda$ has rather mild consistency strength. On the other hand, Motik Gitik and Peter Koepke prove in \cite{GK} without any large cardinal assumptions that for any cardinal $\lambda$, there exists a model $N$ of the theory  \[ZF\, \plus \, \mbox{\tbl} GCH \mbox{\textit{ holds below }} \aleph_\omega \tbr\, \plus \, \mbox{\tbl} \mbox{\textit{ there is a surjective map }} f: \powerset(\aleph_\omega) \rightarrow \lambda\tbr.\] 

Generally, the \tbl size\tbr\,of $\powerset(\kappa)$ can be measured surjectively by the $\theta$-function \[\theta (\kappa) := \sup \{ \alpha \in \Ord\ | \ \exists\, f: \powerset(\kappa) \rightarrow \alpha \mbox{ surjective function}\},\] which provides a surjective substitute for the continuum function $2^\kappa$ in settings without the Axiom of Choice.

If $\theta(\kappa) = \lambda$, then there exists a surjective function $f: \powerset(\kappa) \rightarrow \alpha$ for all $\alpha < \lambda$, but there is not surjective function $f: \powerset(\kappa) \rightarrow \lambda$. \\[-3mm]

We show that in $ZF$, there is an analogue of Easton's Theorem 
for regular \textit{and} singular cardinals. Namely, the only constraints on the $\theta$-function
are the obvious ones: weak monotonicity, and $\theta (\kappa) \geq \kappa^{\plus \plus}$ for all $\kappa$. \\[-3mm]

We write $Card$ for the class of infinite cardinals.

\begin{THM} Let $V$ be a ground model of $ZFC \, + \, GCH$ with a function $F: \Card \rightarrow \Card$ such that the following properties hold: \begin{itemize} \item $\forall \kappa\ F(\kappa) \geq \kappa^{\plus \plus}$ \item $\forall\, \kappa, \lambda \ \big(\kappa \leq \lambda \rightarrow F(\kappa) \leq F(\lambda)\big)$. \end{itemize} Then there is a cardinal-preserving extension $N \supseteq V$ with $N \vDash ZF$ such that $\theta^N (\kappa) = F(\kappa)$ holds for all $\kappa$. 
\end{THM}

In other words, the $\theta$-function can take almost arbitrary values. \\[-4mm]

This paper is structured as follows. Section 1 contains a collection of basic definitions and properties about class forcing, that we will refer to later on. In section 2, a class forcing $\m{P}$ is introduced, which blows up the power sets of all cardinals $\kappa$ according to $F$. For a $V$-generic filter $G$ on $\m{P}$, the structure $\left\langle V[G], \in, V, G\right\rangle$ will not satisfy $ZFC$: For instance, $\m{P}$ adds a cofinal function $f: \omega \rightarrow \Ord$. However, we will see that certain set-sized subforcings of $\m{P}$ are fairly mild and preserve cardinals and the $GCH$. In section 3, a group $A$ of $\m{P}$-automorphisms is introduced, with a collection of $A$-subgroups that determine our symmetric submodel $N$.  
In section 4, we prove that $N \vDash ZF$, and any set of ordinals $X \subseteq \alpha$ with $X \in N$ is contained in a cardinal-preserving $V$-generic extension by  
a\tbl mild\tbr\,$\m{P}$-subforcing as mentioned before. Hence, cardinals are $N$-$V$-absolute. 
By construction of $N$, we will have surjections $s: \powerset^N (\kappa) \rightarrow \alpha$ for all $\kappa$ and $\alpha < F(\kappa)$.
Finally, in section 5 we prove
that $\theta^N (\kappa) \leq F(\kappa)$ for all $\kappa$, 
i.e. there is no surjective function $S: \powerset (\kappa) \rightarrow F(\kappa)$ in $N$. We conclude with a remark about the $\theta$-function and $DC$ in section 6.

\section{Basic Notations and Facts about Class Forcing.} \label{class forcing} 
%This chapter contains some basic definitions and results about class forcing. \\ 
%We start with 
We briefly review a couple of basic concepts and notations about class forcing and symmetric extensions.
% that we will refer to later on. 
For a detailed introduction to forcing and class forcing, see \cite{Jech} and \cite{Friedman}. Concerning the construction of symmetric extensions, the standard method for forcing with Boolean values as described in \cite{Jech2}, is translated to partial orders. A comprehensive presentation of this technique can be found in \cite{ID}. \\[-2mm]

We work with a ground model $V$ and a forcing $(\m{P}, \leq, \m{1})$, which will be a $V$-definable class. As usual, we denote by $Name(\m{P})^V$ the class of $\m{P}$-names in $V$. 

For a $V$-generic filter $G$ on $\m{P}$ and $V[G] := \{\dot{x}^G\ | \ \dot{x} \in \Name (\m{P})^V\}$, we will work with the structure $\left\langle V[G], \in, V, G\right\rangle$, where we have unary predicate symbols for the ground model and the generic filter, with
%with predicates for the ground model $V$ and the generic filter $G$, with 
%Sometimes, we will need 
the canonical names \[\check{V} := \{ (\check{x}, \m{1})\ | \ x \in V\}\ \ , \ \ \mbox{ and }\dot{G} := \{ (\check{p}, p)\ | \ p \in \m{P} \}. \] 

We extend our language of set theory $\mathcal{L}_\in$ by unary predicate symbols $A$ and $B$ for $V$ and $G$ respectively, and denote this extended language by $\mathcal{L}_\in^{A, B}$. \\ Generally, $\left\langle V[G], \in, V, G \right \rangle$ is not a model of $ZFC$.

\begin{defin}
For a formula $\varphi (v_0, \, \ldots \,, v_{n-1}) \in \mathcal{L}_{\in}^{A, B}$, a condition $p \in \m{P}$ and $\m{P}$-names $\dot{x}_0, \, \ldots \,,  \dot{x}_{n-1} \in \Name (\m{P})^V$, write \[p \Vdash^{V, \check{V}, \dot{G}}_{\m{P}} \varphi (\dot{x}_0, \, \ldots \,, \dot{x}_{n-1})\] if for any $G$ a $V$-generic filter on $\m{P}$ with $p \in G$, it follows that $\varphi (\dot{x}_0^G, \, \ldots \,, \dot{x}_{n-1}^G)$ holds in the structure $\left\langle V[G], \in, V, G\right\rangle$. \end{defin}

As for set forcing, we have the following \textit{symmetry lemma}: \begin{Lemma} If \[p \Vdash^{V, \check{V}, \dot{G}}_{\m{P}} \varphi (\dot{x}_0, \, \ldots \,, \dot{x}_{n-1})\] and $\pi$ is a $\m{P}$-automorphism extending to the name space $\Name(\m{P})^V$ as usual, then \[\pi p \Vdash^{V, \check{V}, \pi \dot{G}}_{\m{P}} \varphi (\pi \dot{x}_0, \, \ldots \,, \pi \dot{x}_{n-1}),\] i.e. for any $G$ a $V$-generic filter on $\m{P}$ with $\pi p \in G$, it follows that $\varphi ((\pi \dot{x}_0)^G, \, \ldots \,, (\pi \dot{x}_{n-1})^G)$ holds in the structure $\left\langle V[G], \in, V, \pi G\right\rangle$. \end{Lemma}
 
We define a $V$-generic \textit{symmetric extension} in the following setting:

We will have a group $A$ of $\m{P}$-automorphisms and finitely many formulas $\varphi_0 (x, y), $ $\ldots$, $\varphi_{l-1}(x, y)$ such that for any $i < l$ and $\alpha \in \Ord$, the class $A_i (\alpha) = \{\pi \in A\ | \ \varphi_i (\pi, \alpha)\} \subseteq A$ is a subgroup, such that the following \textit{normality property} holds: Whenever $\pi \in A$ and $i < l$, $\alpha \in \Ord$, then there are finitely many $i_0, \ldots, i_{n-1} < l$ and ordinals $\alpha_0, \ldots, \alpha_{n-1}$ with $\pi^{-1} A_i (\alpha) \pi \supseteq A_{i_0} (\alpha_0)\, \cap\, \cdots \, \cap \, A_{i_{n-1}} (\alpha_{n-1})$.
This corresponds to the property that in the set forcing case, the subgroups $A_i (\alpha)$ generate a normal filter. \\ A $\m{P}$-name $\dot{x}$ is \textit{symmetric} iff 
there are finitely many $(\alpha_0, i_0), \ldots, (\alpha_{n-1}, i_{n-1})$ with $i_0, \ldots, i_{n-1} < l$, $\alpha_0, \ldots, \alpha_{n-1} \in \Ord$, such that \[\{\pi \in A\ | \ \pi \dot{x} = \dot{x}\} \supseteq A_{i_0}(\alpha_0)\, \cap\, \cdots\, \cap\, A_{i_{n-1}} (\alpha_{n-1}).\]
Recursively, a name $\dot{x}$ is \textit{hereditarily symmetric}, $\dot{x} \in HS$, if $\dot{x}$ is symmetric, and $\dot{y} \in HS$ for all $(\dot{y}, p) \in \dot{x}$.

The according \textit{symmetric submodel} is $V(G) := \{\dot{x}^G\ | \ \dot{x} \in HS\}$.  \\ We will work with the structure $\left\langle V(G), \in, V \right\rangle$ with an additional predicate symbol for the ground model.\\[-2mm]

\begin{defin} For an $\mathcal{L}_{\in}^{A, B}$-formula $\varphi (v_0,\,  \ldots, \, v_{n-1})$, a condition $p \in \m{P}$ and $\dot{x}_0, \, \ldots \,, \dot{x}_{n-1} \in HS$, let \[ p \,(\Vdash_s)^{V, \check{V}}_{\m{P}} \varphi (\dot{x}_0, \, \ldots \,, \dot{x}_{n-1})\] if for any $G$ a $V$-generic filter on $\m{P}$ with $p \in G$, it follows that $\varphi (\dot{x}_0^G, \, \ldots \,, x_{n-1}^G)$ holds in the symmetric extension $\left\langle V(G), \in, V\right\rangle$. \end{defin} 

This symmetric forcing relation \tbl $\Vdash_s$\tbr\, satisfies most of the basic properties of \tbl $\Vdash$\tbr. \\[-4mm]

The symmetry lemma holds as well: \begin{Lemma} If $\dot{x}_0, \, \ldots\, \dot{x}_{n-1} \in HS$ and $p \in \m{P}$ with \[p \, (\Vdash_s)^{V, \check{V}}_{\m{P}} \varphi (\dot{x}_0, \, \ldots \,, \dot{x}_{n-1})\] and $\pi$ is a $\m{P}$-automorphism extending to the name space as usual, it follows that \[\pi p \, (\Vdash_s)^{V, \check{V}}_{\m{P}} \varphi (\pi \dot{x}_0, \, \ldots \,, \pi \dot{x}_{n-1}), \] i.e. for any $G$ a $V$-generic filter on $\m{P}$ with $\pi p \in G$, it follows that $\varphi ((\pi \dot{x}_0)^G, \, \ldots \,, (\pi \dot{x}_{n-1})^G)$ holds in $\left\langle V(G), \in, V\right\rangle$. \end{Lemma} 

Our class forcing $\m{P}$ will be an increasing union of set-sized complete subforcings: \\[-3mm]

\textit{There is a sequence $( (\m{P}_{\alpha}, \leq_{\alpha})\ | \ \alpha \in \Ord)$ of set forcing such that \begin{itemize} \vspace*{-1mm}\item for all $\alpha < \beta$, it follows that $\m{P}_{\alpha} \subseteq \m{P}_{\beta}$ with $\leq_{\alpha} \, = \, \leq_{\beta} \, \uhr \, \m{P}_{\beta}$, \vspace*{-1mm} \item if $A \subseteq \m{P}_{\alpha}$ is a maximal antichain, then $A$ is also maximal in $\m{P}$. \end{itemize}}

Hence, the forcing $\m{P}$ satisfies the forcing theorem for \tbl $\Vdash$\tbr\,and \tbl $\Vdash_s$\tbr\,(cf. \cite{PRPPA}).

\section{The forcing.}

We start from a ground model $V \vDash ZFC\, \plus \, GCH$ with a function $F: \Card \rightarrow \Card$ on the class of infinite cardinals such that the following properties hold: \begin{itemize} \item $\forall \kappa\ F(\kappa) \geq \kappa^{\plus \plus}$ \item $\forall\, \kappa, \lambda\ \big(\kappa \leq \lambda \rightarrow F(\kappa) \leq F(\lambda)\big)$. \end{itemize}

In this section, we define our class forcing $\m{P}$ and give some basic properties.

We will have to treat limit cardinals and successor cardinals separately: 
$\m{P}$ is a product $\m{P} := \m{P}_0 \times \m{P}_1$, where $\m{P}_0$ will blow up the power sets of all limit cardinals $\kappa$, and $\m{P}_1$ is in charge of the successor cardinals $\kappa^\plus$. \\[-3mm]

The conditions in $\m{P}_0$ will be functions on trees with finitely many maximal points. \\[-4mm]

For constructing $\m{P}_0$, our function $F$ has to be modified as follows: For all limit cardinals $\kappa$, let $F_{\lim} (\kappa) := F(\kappa)$, and for any successor cardinal $\kappa^{\plus} > \aleph_\omega$, let $F_{\lim} (\kappa^{\plus}) := F(\ol{\kappa})$, where $\ol{\kappa} := \sup \{ \lambda < \kappa^{\plus} \ | \ \lambda \mbox{ is a limit cardinal }\}$. For $n < \omega$, set $F_{\lim} (\aleph_n) := F(\aleph_0)$. Let $F_{\lim} (0) := \{ 0\}$. \\[-3mm]

Our trees' levels will be indexed by cardinals, and on any level $\kappa$, the trees contain finitely many vertices $(\kappa, i)$ with $i < F_{\lim} (\kappa)$.

\begin{definition} \label{tree} A partial order $(t, \leq_t)$ is an {\upshape{$F_{\lim}$-tree}}, if \[t \subseteq \bigcup_{\kappa \in \Card} \{\kappa\} \times F_{\lim} (\kappa) \ \cup \ \{ (0, 0) \}\] 
with the following properties: \begin{itemize} \item If $(\kappa, i ) \leq_t (\lambda, j)$, then $\kappa \leq \lambda$. \item For any $(\lambda,j) \in t$ and $\kappa < \lambda$, there exists exactly one $i < F_{\lim} (\kappa)$ with $(\kappa, i) \leq_t (\lambda, j)$.
\item The tree $t$ has finitely many {\normalfont maximal points}, i.e. there are finitely many $(\kappa_0, i_0), \, \ldots \,, (\kappa_{n-1}, i_{n-1})$ with $t = \{ (\kappa, i)\ | \ \exists\, m < n\ (\kappa, i ) \leq_t (\kappa_m, i_m)\}$. \item There is no {\normalfont splitting at limits}, i.e. for any limit level $\kappa$ and $(\kappa, i), (\kappa, i^\prime) \in t$ with $\{(\lambda, j) \in t\ | \ (\lambda, j) \leq_t (\kappa, i)\} = \{(\lambda, j) \in t\ | \ (\lambda, j) \leq_t (\kappa, i^\prime)\}$, it follows that $i = i^\prime$. \end{itemize} 

If $(t, \leq_t)$ is an $F_{\lim}$-tree with maximal points $(\kappa_0, i_0), \, \ldots \,, (\kappa_{n-1}, i_{n-1})$, we call $\HT t := \max \{\kappa_0, \, \ldots \,, \kappa_{n-1})$ the {\normalfont height} \hspace*{-1,5mm} of $t$. \\[-2mm]

\end{definition}
The first and second conditions make sure that for any $F_{\lim}$-tree $(t, \leq_t)$, the predecessors of any $(\kappa, i) \in t$ with $\kappa = \aleph_\alpha$ are linearly ordered by $\leq_t$ in order type $\alpha$ (or $\alpha + 1$ for $\alpha$ finite), and for any $(\kappa, i) \in t$, it follows that $(0, 0) \leq_t (\kappa, i)$.\\[-2mm] 

There is a canonical partial order on the class of $F_{\lim}$-trees as follows: Let $(s, \leq_s) \leq_{F_{\lim}-tree} (t, \leq_t)$ iff $s \supseteq t$ and $\leq_s\; \supseteq \;\leq_t$. \\[-2mm]

The conditions in our forcing $\m{P}_0$ will be functions $p: (t(p), \leq_{t(p)}) \rightarrow V$ whose domain $( t(p), \leq_{t(p)})$ is an $F_{\lim}$-tree. \\
The functional values of $p$ below any maximal point $(\kappa, i) \in t(p)$, will make up a partial $0$-$1$-function on $\kappa$. If $(\kappa, i)$ and $(\lambda, j)$ are the maximal points of two branches meeting at level $\nu$, then the according $0$-$1$-functions coincide up to $\nu$. \\ Hence, a $\m{P}_0$-generic filter will add a new $\kappa$-subset $G_{(\kappa, i)}$ below any vertex $(\kappa, i)$ with $i < F_{\lim}(\kappa)$. 
The fourth condition in Definition \ref{tree} makes sure that for any $i, i^\prime < F_{\lim}(\kappa)$ with $i \neq i^\prime$, the according $\kappa$-subsets $G_{(\kappa, i)}$ and $G_{(\kappa, i^\prime)}$ given by the branches below $(\kappa, i)$ and $(\kappa, i^\prime)$ are not equal. Hence, our forcing adds $F_{\lim}(\kappa)$-many pairwise distinct $\kappa$-subsets for any cardinal $\kappa$. \\[-3mm]

Denote by $Fn(x, 2, \kappa)$ the collection of all functions $f: \dom f \rightarrow 2$ with $\dom f \subseteq x$ and $|\dom f| < \kappa$.

\begin{definition} The class forcing $(\m{P}_0, \leq_0)$ consists of all functions $p: (t (p), \leq_{t (p)})$ $\rightarrow V$ such that $(t (p), \leq_{t (p)})$ is an $F_{\lim}$-tree, and  \begin{itemize} \item $p (\kappa^{\plus}, i) \in Fn( \, [\kappa, \kappa^{\plus}), 2, \kappa^{\plus})$ for all $(\kappa^\plus, i) \in t(p)$ with $\kappa^\plus$ a successor cardinal, \item $p (\aleph_0, i) \in Fn(\aleph_0, 2, \aleph_0)$ for all $(\aleph_0, i) \in t(p)$,\, {\normalfont and}  \item $p (\kappa, i) = \emptyset$ for all $(\kappa, i) \in t(p)$ with $\kappa$ a limit cardinal or $\kappa = 0$. \item For $(\kappa, i) \in t (p)$, let \[p_{(\kappa, i)} := \bigcup \, \{ p(\nu^\plus, j)\ | \ (\nu^\plus, j) \leq_{t (p)} (\kappa, i) \}. \]
We require that $|p_{(\kappa, i)}| < \kappa$ for all $i < F_{\lim}(\kappa)$ whenever $\kappa$ is a regular limit cardinal. \end{itemize}

For $p: (t (p), \leq_{t (p)}) \rightarrow V$, $q : (t (q), \leq_{t (q)}) \rightarrow V$  
conditions in $\m{P}_0$, let $q \leq_0 p$ iff \begin{itemize} \item $(t (q), \leq_{t (q)}) \leq_{F_{\lim}-tree} (t(p), \leq_{t(p)})$,
\item $q(\kappa, i) \supseteq p (\kappa, i)$ for all $(\kappa, i) \in t (p)$.

\end{itemize} 

Let $\m{1}_0 := \emptyset$. \end{definition}

For $p \in \m{P}_0$, $p: (t (p), \leq_{t(p)}) \rightarrow V$ we call $\HT p := \HT t(p)$ the \textit{height} of $p$. \\[-3mm]

For a cardinal $\lambda$, we denote by $p \uhr (\lambda + 1): t (p) \uhr (\lambda + 1) \rightarrow V$ the restriction of $p$ to the subtree $t (p) \uhr (\lambda + 1) := \{ (\kappa, i) \in t (p)\ | \ \kappa \leq \lambda\}$,
$\leq_{t (p) \, \uhr \, (\lambda + 1)} \: := \ \leq_{t (p)} \, \cap \: (t (p) \uhr (\lambda + 1))$. Then $p \uhr (\lambda + 1) \in \m{P}_0$ with $p \leq_0 p \uhr (\lambda + 1)$. Let $\m{P}_0 \uhr (\lambda + 1) := \{p \uhr (\lambda + 1)\ | \ p \in \m{P}_0\}$. \\ 
Similarly, we define $p \uhr [\lambda, \infty): t (p) \uhr [\lambda, \infty) \rightarrow V$ (which is not a condition in $\m{P}_0$), with $t (p) \, \uhr \, [\lambda, \infty) := \{ (\kappa, i) \in t (p)\ | \ \kappa \geq \lambda\}$. 
Let $\big( p \uhr [\lambda, \infty) \big) (\kappa, i) := p (\kappa, i)$ for all $(\kappa, i) \in t(p)$ with $\kappa > \lambda$, and $\big( p \uhr [\lambda, \infty) \big) (\lambda, i) := \emptyset$ for any $(\lambda, i) \in t(p)$. Set $\m{P}_0 \uhr [\lambda, \infty) := \{p \uhr [\lambda, \infty)\ | \ p \in \m{P}_0\}$. \\
The forcing $\m{P}_0$ is dense in the product $\m{P}_0 \uhr (\lambda + 1)\, \times\, \m{P}_0 \uhr [\lambda, \infty)$. \\
Similarly, for cardinals $\mu$, $\lambda$ with $\mu < \lambda$, we define $p \, \uhr \, [\mu, \lambda + 1): t(p) \, \uhr \, [\mu, \lambda + 1) \rightarrow V$ with $t(p)\, \uhr \, [\mu, \lambda + 1) := \{ (\kappa, i) \in t(p)\ | \ \mu \leq \kappa \leq \lambda\}$, and set $\m{P}_0 \uhr [\mu, \lambda + 1) := \{p \uhr [\mu, \lambda + 1)\ | \ p \in \m{P}_0\}$. \\[-2mm]

For conditions $p, q \in \m{P}_0$ with $p\, \| \, q$, it follow that $t(p)\, \cup \, t(q)$ with the order relation $\leq_{t(p)}\, \cup \, \leq_{t(q)}$ is an $F_{\lim}$-tree as well, and we can define a \textit{maximal common extension} $p\, \cup\, q$ of $p$ and $q$ as follows: Let $t(p\, \cup\, q) \,:= \,t(p)\, \cup\, t(q)$, $\leq_{t(p\, \cup \, q)} \, := \,\leq_{t(p)}\, \cup \, \leq_{t(q)}$ with $(p\, \cup\, q) (\kappa, i) := p(\kappa, i)\, \cup\, q(\kappa, i)$ for any $(\kappa, i) \in t(p)\, \cap \, t(q)$, $(p\, \cup\, q) (\kappa, i) := p(\kappa, i)$ whenever $(\kappa, i) \in t(p) \setminus t(q)$ and $(p\, \cup \, q) (\kappa, i) := q(\kappa, i)$ for all $(\kappa, i) \in t(q) \setminus t(p)$. \\[-2mm]

Surely, the class forcing $\m{P}_0$ does not preserve $ZFC$: For example, with the generic filter $G$ as a parameter, one can use the finiteness of the trees and construct in $V[G]$ a cofinal function $f: \omega \rightarrow \Ord$. \\[-3mm]

However, 

\begin{lem} The forcing $(\m{P}_0, \leq_0)$ is an increasing union of set-sized complete subforcings. \end{lem}

\begin{proof} For $\alpha \in Ord$, let $(\m{P}_0)_\alpha := \m{P}_0 \uhr (\aleph_\alpha + 1) = \{p \in \m{P}_0\ | \ \HT p \leq \aleph_\alpha\}$ with $(\leq_0)_\alpha$ the ordering on $(\m{P}_0)_\alpha$ induced by $\leq_0$. Then $\m{P}_0 = \bigcup \, \{ (\m{P}_0)_\alpha\ | \ \alpha \in \Ord\}$ is an increasing union of set-sized forcings such that for any $\alpha < \beta$, it follows that $(\m{P}_0)_\alpha \subseteq (\m{P}_0)_\beta$ with $(\leq_0)_\alpha = (\leq_0)_\beta\, \uhr \, (\m{P}_0)_\alpha$. \\ Let $A$ be a maximal antichain in some $(\m{P}_0)_\alpha$. It remains to verify that $A$ is also maximal in $\m{P}_0$. Assume towards a contradiction, there was a condition $p \in \m{P}_0$ with $p \, \bot \, q$ for all $q \in A$. Take $\ol{q} \in A$ with $\ol{q}\; \| \; p \uhr (\aleph_\alpha + 1)$, and denote by $r$ a common extension of $p \uhr (\aleph_\alpha + 1)$ and $\ol{q}$ in $(\m{P}_0)_\alpha$. Then $\ol{r} := r\, \cup \, p$ with $\ol{r} \uhr (\aleph_\alpha + 1) = r$, $\ol{r} \uhr [\aleph_\alpha, \infty) = p \uhr [\aleph_\alpha, \infty)$ is a common extension of $p$ and $\ol{q}$. Contradiction.  

\end{proof}

Hence, it follows that $\m{P}_0$ satisfies the forcing theorem for every $\mathcal{L}_\in^{A, B}$-formula $\varphi$; in particular, the forcing relation $\Vdash^{V, \check{V}, \dot{G}}_{\m{P}_0}$ is definable. 
Thus, the forcing theorem also holds for the symmetric forcing relation $(\Vdash_s)_{\m{P}}^{V, \check{V}}$ and $\varphi$ an $\mathcal{L}_\in^A$-formula. \\[-3mm]

For any $G$ a $V$-generic filter on $\m{P}_0$ and $\alpha$ an ordinal, it follows that the filter $(G_0)_{\alpha}:= \{p \in G_0\ | \ \HT p \leq \aleph_\alpha\}$ is $V$-generic on $(\m{P}_0)_{\alpha}$. \\[-2mm]

For successor cardinals $\kappa^{\plus}$, the forcing $\m{P}_0$ only adds $F_{\lim} (\kappa^{\plus})$-many $\kappa^{\plus}$-subsets, which might be less than the desired $F(\kappa^{\plus})$, so we need a second forcing $\m{P}_1$ to blow up the power sets $\powerset (\kappa^\plus)$.
  
The reason why we use for $\m{P}_0$ the function $F_{\lim}$ instead of $F$ is that for singular limit cardinals $\kappa$, we will have to use the forcing $\m{P}_0 \uhr (\kappa^\plus + 1)$ instead of $\m{P}_0 \uhr (\kappa + 1)$ for capturing $\kappa$-subsets in $N$ in our proof of $\theta^N (\kappa) \leq F(\kappa)$, and we will need that $F_{\lim} (\kappa^{\plus}) =  F(\kappa)$ to make sure that $\m{P}_0 \uhr (\kappa^\plus + 1)$ only has size $F(\kappa)$. \\[-1mm]

The forcing $\m{P}_1$ will be a variant of Easton forcing with finite support: We will have a finite support-product of forcings $Fn (\, [\kappa, \kappa^{\plus}) \times F (\kappa^{\plus}), 2, \kappa^{\plus})$, where a successor cardinal $\kappa^{\plus}$ shall only be included into the forcing if $F(\kappa^{\plus})$ is strictly greater than $F(\nu^\plus)$ for all $\nu < \kappa$.

\begin{definition} Let $\Succ^\prime$ denote the class of all successor cardinals $\kappa^{\plus}$ with the property that $F(\kappa^{\plus}) > F(\nu^\plus)$ for all $\nu^\plus < \kappa^\plus$. The forcing $(\m{P}_1, \leq_1, \m{1}_1)$ consists of all conditions $p: \supp p \rightarrow V$ with $\supp p \subseteq Succ^\prime$ finite and
\[p(\kappa^{\plus}) \in Fn (\, [\kappa, \kappa^{\plus}) \times F(\kappa^{\plus}), 2, \kappa^{\plus})\] for all $\kappa^\plus \in \supp p$ such that $\dom\, p (\kappa^\plus)$ is rectangular, i.e. $\dom\, p(\kappa^{\plus}) = \dom_x\, p(\kappa^{\plus}) \times \dom_y \, p (\kappa^{\plus})$ for some $\dom_x \, p(\kappa^{\plus}) \subseteq [\kappa, \kappa^{\plus})$ and $\dom_y \, p(\kappa^{\plus}) \subseteq F(\kappa^{\plus})$. \\ The conditions in $\m{P}_1$ are ordered by reverse inclusion: Let $q \leq_1 p$ iff $\supp q \supseteq \supp p$ with $q (\kappa^\plus) \supseteq p(\kappa^\plus)$ for all $\kappa^\plus \in \supp p$. \\ The maximal condition is $\m{1}_1 := \emptyset$. \end{definition} 

For a cardinal $\lambda$ and $p \in \m{P}_1$, we denote by $p \uhr (\lambda + 1)$ the restriction of $p$ to the domain $\{\kappa^{\plus} \in \supp p\ | \ \kappa^{\plus} \leq \lambda\}$. Similarly, we write $p \uhr [\lambda, \infty)$ for the restriction of $p$ to $\{\kappa^{\plus} \in \supp p \ | \ \kappa^{\plus} > \lambda\}$. 

Let $\m{P}_1 \uhr (\lambda + 1):= \{p_1 \uhr (\lambda + 1)\ | \ p_1 \in \m{P}_1\}$, and $\m{P}_1 \uhr [\lambda, \infty) := \{p_1 \uhr [\lambda, \infty)\ | \ p_1 \in \m{P}_1\}$. Then $\m{P}_1 \, \cong \, \m{P}_1 \uhr (\lambda + 1)\, \times\, \m{P}_1 \uhr [\lambda, \infty)$. \\[-3mm] 

For a successor cardinal $\kappa^{\plus} \in Succ^\prime$, we set $\m{P}_1 (\kappa^{\plus}) := \{p (\kappa^{\plus})\ | \ p \in \m{P}_1\} = Fn ( \, [\kappa, \kappa^{\plus}) \times F(\kappa^{\plus}), 2, \kappa^{\plus})$. If $G_1$ is a $V$-generic filter on $\m{P}_1$, it follows that $G_1 (\kappa^\plus) := \{p (\kappa^\plus)\ | \ p \in G_1\}$ is $V$-generic on $\m{P}_1 (\kappa^\plus)$. \\[-2mm]

For $\alpha \in \Ord$, let $(\m{P}_1)_\alpha := \{p \uhr (\aleph_\alpha + 1)\ | \ p \in \m{P}_1\} = \{p \in \m{P}_1\ | \ \supp p \subseteq \aleph_\alpha + 1\}$. Then $\m{P}_1 = \bigcup \{ (\m{P}_1)_\alpha\ | \ \alpha \in \Ord\}$ is an increasing union of set-sized complete subforcings. 

Hence, $\m{P}_1$ satisfies the forcing theorem for every $\mathcal{L}_\in^{A, B}$-formula $\varphi$, and the forcing relation $\Vdash_{\m{P}_1}^{V, \check{V}, \dot{G}}$ is definable. Thus, the forcing theorem also holds for $(\Vdash_s)_{\m{P}}^{V, \check{V}}$ and $\varphi$ a $\mathcal{L}_\in^A$-formula. \\[-3mm]

If $G_1$ is a $V$-generic filter on $\m{P}_1$ and $\alpha \in \Ord$, it follows that $(G_1)_{\alpha}:= \{p \in G_1\ | \ \supp p \subseteq \aleph_\alpha + 1\}$ is a $V$-generic filter on $(\m{P}_1)_{\alpha}$.

\begin{definition} \[( \m{P}, \leq) := (\m{P}_0 \times \m{P}_1, \leq_{\,\m{P}_0 \times \m{P}_1}).\] \end{definition}

If $G_0$ is $V$-generic on $\m{P}_0$ and $G_1$ is $V[G_0]$-generic on $\m{P}_0$, then it follows as in the set forcing case that $G := G_0 \times G_1$ is a $V$-generic filter on $\m{P}$. From the definability of $\Vdash_{\m{P}_0}^{V, \check{V}, \dot{G}}$, it follows that the converse is true, as well. \\[-2mm]

If $p = (p_0, p_1)$ is a condition in $p$ and $\lambda \in \Card$, let $p \uhr (\lambda + 1) := (p_0 \uhr (\lambda + 1), p_1 \uhr (\lambda + 1))$. Set $\eta (p) := \min \{\lambda\ | \ p \uhr (\lambda + 1) = p\}$. \\ Similarly as before, we let $\m{P} \uhr (\lambda + 1) := \{p \uhr (\lambda + 1)\ | \ p \in \m{P}\}$. \\[-2mm]

Our eventual symmetric submodel $N \subseteq V[G]$ will have the crucial property that sets of ordinals $X \subseteq \alpha$ with $X \in N$ can be captured in \tbl mild\tbr\, $V$-generic extensions of the following form:

\begin{deflem} \label{gen} For conditions $p, q \in \m{P}_0$ with $(t(q), \leq_{t(q)}) \leq_{F_{\lim}-tree}$ $(t(p), \leq_{t(p)})$, we denote by $q \uhr t(p)$ the restriction of $q$ to the domain $t(p)$. Let \[\m{P}_0 \uhr t (p) := \{q \uhr t(p) \ | \ q \in \m{P}_0, t(q) \leq t(p) \},\] with the partial order induced by $\leq_0$, and the maximal element $\m{1}_{\m{P}_0 \, \uhr \, t(p)}: t(p) \rightarrow V$ with $\m{1}_{\m{P}_0 \, \uhr \, t(p)} (\kappa, i) = \emptyset$ for all $(\kappa, i) \in t(p)$. \\[-3mm]

For $G_0$ a $V$-generic filter on $\m{P}_0$ and $p \in G_0$, it follows that \[G_0 \uhr t (p) := \{ q \uhr t (p) \ | \ q \in G_0, t(q) \leq_{F_{\lim}-tree} t(p)\} = \{q \in G_0\ | \ t(q) = t(p)\} \]is a $V$-generic filter on $\m{P}_0 \uhr t (p)$. \end{deflem}

\begin{proof} Consider a dense set $D \subseteq \m{P}_0 \uhr t (p)$. It suffices to show that $\ol{D} := \{q \in \m{P}_0\ | \ q \uhr t (p) \in D\}$ is dense in $\m{P}_0$ below $p$. \\
Take $q \in \m{P}_0$ with $q \leq_0 p$. There exists $r \in \m{P}_0 \uhr t (p)$, $r \in D$, with $r \leq_0 q \uhr t (p)$. We define a condition $\ol{q} \in \m{P}_0$ as follows: $(t (\ol{q}), \leq_{t (\ol{q})}) := (t (q), \leq_{t (q)})$ with $\ol{q} (\kappa, i) := r (\kappa, i)$ for $(\kappa, i) \in t (p)$, and $\ol{q} (\kappa, i) := q(\kappa, i)$, else. Then $\ol{q} \leq_0 q$ with $\ol{q} \uhr t (p) = r \in D$ as desired. \end{proof} 

For finitely many $(\kappa_0, i_0), \, \ldots \,, $ $(\kappa_{n-1}, i_{n-1}) \in t (p)$, we denote by $t (p) \uhr \{ (\kappa_0, i_0), \, \ldots \,, $ $(\kappa_{n-1}, i_{n-1})\}$ the subtree $ \{(\kappa, i) \in t(p)\ | \ \exists\, m < n \ (\kappa, i) \leq_{t (p)} (\kappa_m, i_m)\}$ with the ordering induced by $\leq_{t(p)}$. We write $p \uhr \{ (\kappa_0, i_0), \, \ldots \,, (\kappa_{n-1}, i_{n-1})\}$ for the restriction of $p$ to the subtree $t(p) \uhr \{ (\kappa_0, i_0), \ldots, (\kappa_{n-1}, i_{n-1})\}$. \\ If the set $\{(\kappa_0, i_0), \ldots, (\kappa_{n-1}, i_{n-1})\}$ {\normalfont contains all maximal points of $p$}, i.e. for any $(\kappa, i) \in t(p)$ there is $l < n$ with $(\kappa, i) \leq_{t(p)} (\kappa_l, i_l)$, then we sometimes use the notation $G_0 \uhr \{(\kappa_0, i_0), \, \ldots \,, (\kappa_{n-1}, i_{n-1})\}$ instead of $G_0 \uhr t(p)$. \\[-3mm]

We have similar restrictions for $\m{P}_1$: 

\begin{deflem}
Consider finitely many cardinals $\ol{\kappa}_0, \, \ldots \,, \ol{\kappa}_{\ol{n}-1} \in \Succ^\prime$, and $\ol{\imath}_0 < F(\ol{\kappa}_0), \, \ldots \,, \ol{\imath}_{\ol{n}-1} < F(\ol{\kappa}_{\ol{n}-1})$. For a condition $p_1 \in \m{P}_1$, we define $p_1 \uhr \{ (\ol{\kappa}_0, \ol{\imath}_0), \, \ldots \,, $ $(\ol{\kappa}_{\ol{n}-1}, \ol{\imath}_{\ol{n}-1}) \}$ as follows: \[\dom \ p_1 \uhr \{ (\ol{\kappa}_0, \ol{\imath}_0), \, \ldots \,, (\ol{\kappa}_{\ol{n}-1}, \ol{\imath}_{\ol{n}-1}) \} := \]\[ \big( \dom_x p (\ol{\kappa}_0) \times \{\ol{\imath}_0 \}\big) \, \cup \, \cdots \, \cup \big( \dom_x p (\ol{\kappa}_{\ol{n}-1}) \times \{\ol{\imath}_{\ol{n}-1}\} \big) = \]\[ = \{(\alpha, i) \in \dom p\ | \ \exists\, l < \ol{n}\ \;i = \ol{\imath}_l\},\] 
and for any $ (\alpha, \ol{\imath}_l ) \in \dom \, p_1 (\ol{\kappa}_l)$, \[ \big(p_1 \uhr \{ (\ol{\kappa}_0, \ol{\imath}_0), \, \ldots \,, (\ol{\kappa}_{\ol{n}-1}, \ol{\imath}_{\ol{n}-1}) \}\big) (\alpha, \ol{\imath}_l) := p_1 (\ol{\kappa}_l) (\alpha, \ol{\imath}_l).\]

Let
\[\m{P}_1 \uhr \{ (\ol{\kappa}_0, \ol{\imath}_0), \, \ldots \,, (\ol{\kappa}_{\ol{n}-1}, \ol{\imath}_{\ol{n}-1}) \} := \{p_1 \uhr \{ (\ol{\kappa}_0, \ol{\imath}_0), \, \ldots \,, (\ol{\kappa}_{\ol{n}-1}, \ol{\imath}_{\ol{n}-1})\}\ | \ p_1 \in \m{P}_1\}.\]

For $G_1$ a $V$-generic filter on $\m{P}_1$,
it follows that \[G_1 \uhr \{(\ol{\kappa}_0, \ol{\imath}_0), \, \ldots \,, (\ol{\kappa}_{\ol{n}-1}, \ol{\imath}_{\ol{n}-1})\} := \{\, p_1 \uhr \{ (\ol{\kappa}_0, \ol{\imath}_0), \, \ldots \,, (\ol{\kappa}_{\ol{n}-1}, \ol{\imath}_{\ol{n}-1})\}\ | \ p_1 \in G_1\, \}\] is a $V$-generic filter on $\m{P}_1 \uhr \{ (\ol{\kappa}_0, \ol{\imath}_0), \, \ldots \,, (\ol{\kappa}_{\ol{n}-1}, \ol{\imath}_{\ol{n}-1}) \}$.  \\[-3mm]

\end{deflem}

In other words, for any $l < \ol{n}$ with $\ol{\kappa}_l = \wt{\kappa}_l^\plus$, it follows that $\m{P}_1 \uhr \{ (\ol{\kappa}_0, \ol{\imath}_0), \, \ldots \,, (\ol{\kappa}_{\ol{n}-1}, \ol{\imath}_{\ol{n}-1}) \}$ adds a new Cohen-subset to the interval $[\wt{\kappa}_{l}, \wt{\kappa}_{l}^\plus)$. 

\begin{proof} If $D$ is a dense 
subset of $\m{P}_1 \uhr \{ (\ol{\kappa}_0, \ol{\imath}_0), \, \ldots \,, (\ol{\kappa}_{\ol{n}-1}, \ol{\imath}_{\ol{n}-1}) \}$, it follows that \[\ol{D} := \{p_1 \in \m{P}_1\ | \ p_1 \uhr \{ (\ol{\kappa}_0, \ol{\imath}_0), \, \ldots \,, (\ol{\kappa}_{\ol{n}-1}, \ol{\imath}_{\ol{n}-1}) \} \in D \}\] is dense in $\m{P}_1$.
\end{proof}

Hence, if $G = G_0 \times G_1$ is a $V$-generic filter on $\m{P}$ with $(\kappa_0, i_0), \, \ldots \,, (\kappa_{n-1}, i_{n-1}), $ $(\ol{\kappa}_0, \ol{\imath}_0), \, \ldots \,, (\ol{\kappa}_{\ol{n}-1}, \ol{\imath}_{\ol{n}-1})$ as before and $p \in G_0$, $p: t (p) \rightarrow V$ such that $\{(\kappa_0, i_0)\,,$ $\ldots$\,, \,$(\kappa_{n-1}, i_{n-1})\} \subseteq t(p)$ contains all maximal points of $t(p)$, it follows that \[G_0 \uhr t(p) \times G_1 \uhr \{ (\ol{\kappa}_0, \ol{\imath}_0), \, \ldots \,, (\ol{\kappa}_{\ol{n}-1}, \ol{\imath}_{\ol{n}-1})\} \] is a $V$-generic filter on $\m{P}_0 \uhr t (p) \times \m{P}_1 \uhr \{ (\ol{\kappa}_0, \ol{\imath}_0), \, \ldots \,, (\ol{\kappa}_{\ol{n}-1}, \ol{\imath}_{\ol{n}-1})\}$. \\[-3mm]

We will now see that these forcings preserves all cardinals. 

\begin{prop} \label{prescard} Consider a condition $p \in \m{P}_0$ such that $\{ (\kappa_0, i_0), \, \ldots \,, (\kappa_{n-1}, i_{n-1})\} \subseteq t(p)$ contains all maximal points of $t(p)$; moreover, finitely many $(\ol{\kappa}_0, \ol{\imath}_0)$,\,$\ldots$\,,$(\ol{\kappa}_{\ol{n}-1}, \ol{\imath}_{\ol{n}-1})$ with $\ol{\kappa}_0, \, \ldots \,, \ol{\kappa}_{\ol{n}-1} \in \Succ^\prime$, $\ol{\imath}_0 < F (\ol{\kappa}_0), \, \ldots \,, \ol{\imath}_{\ol{n}-1} < F(\ol{\kappa}_{\ol{n}-1})$. \\The forcing \[\m{P}_0 \uhr t(p) \times \m{P}_1 \uhr \{(\ol{\kappa}_0, \ol{\imath}_0), \, \ldots \,, (\ol{\kappa}_{\ol{n}-1}, \ol{\imath}_{\ol{n}-1}) \} \] preserves cardinals and the $GCH$.  \end{prop}

\begin{proof}
Similarly as in \cite[Lemma 1]{GK}, we show that for all cardinals $\lambda$, \[\big(2^\lambda \big)^{V[G_0 \; \uhr \; t (p) \times G_1 \, \uhr \, \{(\ol{\kappa}_0, \ol{\imath}_0), \, \ldots \,\}]} = (\lambda^\plus)^V.\]
First, consider the case that $\lambda = \ol{\lambda}^\plus$ is a successor cardinal. Let $(\m{P}_0 \uhr t (p)) \uhr (\lambda + 1):= \{q \uhr (\lambda + 1)\ | \ q \in \m{P}_0 \uhr t (p)\}$ and $(\m{P}_0 \uhr t (p)) \uhr [\lambda, \infty) := \{q \uhr [\lambda, \infty)\ | \ q \in \m{P}_0 \uhr t (p)\}$. \\
Similarly, let $(\m{P}_1 \uhr \{(\ol{\kappa}_0, \ol{\imath}_0), \, \ldots \,\}) \uhr (\lambda + 1) := \{ (p \uhr (\lambda + 1) ) \uhr \{ (\ol{\kappa}_0, \ol{\imath}_0), \, \ldots \,\} \ | \ p \in \m{P}_1 \}$ denote the lower part, and $(\m{P}_1 \uhr \{(\ol{\kappa}_0, \ol{\imath}_0), \, \ldots \,\}) \uhr [\lambda, \infty) := \{ (p \uhr [\lambda, \infty) )\uhr \{(\ol{\kappa}_0, \ol{\imath}_0), \, \ldots \,\} \ | \ p \in \m{P}_1 \}$ the upper part of the forcing $\m{P}_1 \uhr \{(\ol{\kappa}_0, \ol{\imath}_0), \, \ldots \,\}$. \\[-3mm]

Then $\m{P}_0 \uhr t (p)  \times \m{P}_1 \uhr \{(\ol{\kappa}_0, \ol{\imath}_0), \, \ldots \,, (\ol{\kappa}_{\ol{n}-1}, \ol{\imath}_{\ol{n}-1}) \} $ can be factored as 

\[\Big((\m{P}_0 \uhr t (p))  \uhr (\lambda + 1)\times ( \m{P}_1 \uhr \{(\ol{\kappa}_0, \ol{\imath}_0), \, \ldots \,, (\ol{\kappa}_{\ol{n}-1}, \ol{\imath}_{\ol{n}-1})\}) \uhr (\lambda + 1) \big) \times \] \[\Big((\m{P}_0 \uhr t (p))  \uhr [\lambda, \infty) \times ( \m{P}_1 \uhr \{(\ol{\kappa}_0, \ol{\imath}_0), \, \ldots \,, (\ol{\kappa}_{\ol{n}-1}, \ol{\imath}_{\ol{n}-1})\} ) \uhr [\lambda, \infty)\Big),\] 

where the first factor has cardinality $\leq \lambda$, since $\lambda = \ol{\lambda}^\plus$ is a successor cardinal, and the second factor is $\leq\lambda$ - closed. Thus, it follows that

\[\big(2^\lambda \big)^{V[G_0 \, \uhr \, t (p) \times G_1 \, \uhr \, \{(\ol{\kappa}_0, \ol{\imath}_0), \, \ldots \,\}]} \leq 
|\powerset(\lambda)|^V = (\lambda^\plus)^V\] as desired.\\[-3mm]

If $\lambda$ is a regular limit cardinal, the same argument applies.\\[-3mm]

It remains to show that \[\big(2^\lambda \big)^{V[G_0 \, \uhr \, t (p) \times G_1 \, \uhr \, \{(\ol{\kappa}_0, \ol{\imath}_0), \, \ldots \,\}]} = (\lambda^\plus)^V\] in the case that $\lambda$ is a singular limit cardinal. Assume the contrary and take $\lambda$ least such that $\eta := \cf \lambda < \lambda$ and \[\big(2^\lambda \big)^{V[G_0 \, \uhr \, t (p) \times G_1 \, \uhr \, \{(\ol{\kappa}_0, \ol{\imath}_0), \, \ldots \,\}]} > (\lambda^\plus)^V.\] Let $(\lambda_i\ | \ i < \eta)$ denote a cofinal sequence in $\lambda$. By assumption, it follows that \[\big(2^{\ol{\lambda}}\big)^{V[G_0 \, \uhr \, t (p) \times G_1 \, \uhr \, \{(\ol{\kappa}_0, \ol{\imath}_0), \, \ldots \,\}]} = (\ol{\lambda}^\plus)^V\] for all $\ol{\lambda} < \lambda$. Thus, \[2^\lambda \leq \prod_{i < \eta} 2^{\lambda_i} \leq \big(2^{< \lambda} \big)^{\eta} = \lambda^\eta \leq \lambda^\lambda = 2^\lambda\] holds true in $V$ and $V[G_0 \uhr t (p) \times G_1 \uhr \{(\ol{\kappa}_0, \ol{\imath}_0), \, \ldots \,\}]$. Since $\eta$ is regular, we have \[\big| \, (\m{P}_0 \uhr t (p)) \uhr (\eta + 1)\times (\m{P}_1 \uhr \{(\ol{\kappa}_0, \ol{\imath}_0), \, \ldots \,\}) \uhr (\eta + 1) \, \big| \leq \eta,\] and \[(\m{P}_0 \uhr t (p)) \uhr [\eta, \infty) \times (\m{P}_1 \uhr \{(\ol{\kappa}_0, \ol{\imath}_0), \, \ldots \,\}) \uhr [\eta, \infty)\] is $\leq\eta$ - closed. Thus, \[(2^\lambda)^{V[G_0 \, \uhr \, t (p) \times G_1 \, \uhr \, \{ (\ol{\kappa}_0, \ol{\imath}_0), \, \ldots \, \} ]} = (\lambda^\eta)^{V[G_0 \, \uhr \, t (p) \times G_1 \, \uhr \, \{ (\ol{\kappa}_0, \ol{\imath}_0), \, \ldots \, \} ]} \leq \]\[ \leq (\lambda^\eta)^{V[(G_0 \, \uhr \, t (p)) \, \uhr \, (\eta + 1) \times (G_1 \, \uhr \, \{ (\ol{\kappa}_0, \ol{\imath}_0), \, \ldots \,\}) \, \uhr \, (\eta + 1)]} \leq (2^\lambda)^{V[(G_0 \, \uhr \, t (p)) \, \uhr \, (\eta + 1) \times (G_1 \, \uhr \, \{ (\ol{\kappa}_0, \ol{\imath}_0), \, \ldots \,\}) \, \uhr \, (\eta + 1)]} \leq \]\[ \leq |\powerset( \lambda \times \eta)|^V \leq (2^\lambda)^V = (\lambda^\plus)^V,\] which gives the desired contradiction. 

\end{proof}

We will see that any set of ordinals in our eventual symmetric submodel $N$ can be captured in a generic extension by one of these forcings $\m{P}_0 \uhr t (p)  \times \m{P}_1 \uhr \{(\ol{\kappa}_0, \ol{\imath}_0), \, \ldots \,, (\ol{\kappa}_{\ol{n}-1}, \ol{\imath}_{\ol{n}-1})\} $.
Hence, $N$ preserves all cardinals.

\section{Symmetric names.} \label{symmetric names}
For defining our symmetric submodel $N$, we first we need a group $A$ of $\m{P}$-automorphisms.
We will have $A = A_0 \times A_1$ with $A_0$ a group of $\m{P}_0$-automorphisms, and $A_1$ a group of $\m{P}_1$-automorphisms.
\begin{definition} Denote by $A_0 ({\normalfont levels})$ the collection of all $\pi = (\pi (\kappa)\ | \ \kappa \in \Card, \kappa < \HT \pi)$ with $\HT \pi$, the {\normalfont height} of $\pi$, a cardinal, such that any $\pi (\kappa): \{ (\kappa, i) \ | \ i < F_{\lim} (\kappa)\} \rightarrow \{ (\kappa, i)\ | \ i < F_{\lim} (\kappa) \}$ is a bijection with finite support $\supp \pi (\kappa) := \{(\kappa, i) \ | \ \pi (\kappa) (\kappa, i) \neq (\kappa, i)\}$. \\ A map $\pi \in A_0 ({\normalfont levels})$ induces an automorphism $\pi_{tree}$ on the class of $F_{\lim}$-trees as follows: Set $\pi_{tree} (t, \leq_t) := (s, \leq_s)$ with $s := \pi[t] := \{ \pi (\kappa) (\kappa, i)\ | \ (\kappa, i) \in t\}$, where for $\kappa \geq \HT \pi$, we take for $\pi (\kappa)$ the identity on $\{(\kappa, i)\ | \ i < F_{\lim} (\kappa)\}$. Let $\leq_s \, := \, \pi[\leq_t] := \{(\pi (\kappa) (\kappa, i), \pi(\lambda) (\lambda, k))\ | \ (\kappa, i) \leq_t (\lambda, k)\}$. \\ Also, $\pi$ induces an automorphism $\ol{\pi}: \m{P}_0 \rightarrow \m{P}_0$: For $p \in \m{P}_0$, $p: t (p) \rightarrow V$, let $\ol{\pi} (p): \pi_{tree} (t (p), \leq_{t (p)} ) \rightarrow V$ with $\ol{\pi}(p) \big(\pi (\kappa) (\kappa, i)\big) = p(\kappa, i)$ for all $(\kappa, i) \in t (p)$. \\ Let \[A_0 := \{\ol{\pi} \ | \ \pi \in A_0 ({\normalfont levels})\}.\]  
\end{definition}

We will often confuse an automorphism $\pi$ with its extensions $\pi_{tree}$ and $\ol{\pi}$. \\[-3mm] 

Note that for an $F_{\lim}$-tree $t (p)$, it follows that $\pi (t (p))$ is essentially the same tree, where only the vertices $(\kappa, i)$ have now different \tbl names\tbr\, $\pi (\kappa) (\kappa, i)$.\\[-3mm]

Any $\pi \in A_0$ can be extended to an automorphism on $\Name (\m{P}_0)$ as usual, which will be denoted by the same letter. \\[-2mm]

Let $\kappa$ be a cardinal and $G_0$ a $V$-generic filter on $\m{P}_0$. For every $i < F_{\lim} (\kappa)$, the forcing $\m{P}_0$ adjoins a new $\kappa$-subset $(G_0)_{(\kappa, i)}$ given by the branch through $(\kappa, i)$: 
\[ (G_0)_{(\kappa, i)} = \{\zeta < \kappa\ | \ \exists\, p \in G_0 \ \exists\, (\lambda, j) \leq_{t (p)} (\kappa, i):\ p(\lambda, j) (\zeta) = 1 \}.\] 

Then $ (G_0)_{(\kappa, i)}$ has  a canonical name \[(\dot{G}_0)_{(\kappa, i)} := \{ (\zeta, p)\ | \ \zeta < \kappa\, , \, p \in \m{P}_0 \uhr (\kappa + 1)\, , \, \exists\, (\lambda, j) \leq_{t (p)} (\kappa, i): \ p(\lambda, j) (\zeta) = 1\}, \]

For any $\pi \in A_0$, it follows that $\pi \, \big((\dot{G}_0)_{(\kappa, i)} \big) = (\dot{G}_0)_{\pi (\kappa) (\kappa, i)}$.
Thus, our automorphisms in $A_0$ allow for swapping the generic subsets. \\[-3mm]

We call an automorphism $\pi \in A_0$ \textit{small} if it satisfies the following property: \\[-3mm]

\textit{For all $(\kappa, i)$, it follows that $\pi(\kappa)(\kappa, i) = (\kappa, j)$ such that there is a limit ordinal $\gamma (i)$ with $i, j \in [\gamma (i), \gamma (i) + \omega)$. } \\[-3mm]

It is not difficult to see that for any pair of conditions $p, q \in \m{P}_0$, there is a small automorphism $\pi \in A_0$ with $\pi p \, \| \, q$. Indeed, by the finiteness of the trees, it is possible to arrange that for any $(\kappa, i) \in t(p)$, we have $\pi (\kappa) (\kappa, i) \notin t(p) \, \cup\, t(q)$. \\[-2mm]

Now, we turn to $\m{P}_1$. This time, we do not quite have a group of automorphism on $\m{P}_1$, but a collection $A_1$ of bijections $\pi: D_\pi \rightarrow D_\pi$ such that any $D_\pi$ is dense in $\m{P}_1$ with the following properties: \begin{itemize} \item Whenever $p, q \in \m{P}_1$ with $p \in D_\pi$ and $q \leq p$ such that $\supp q = \supp p$, it follows that also $q \in D$. \item For any $p \in D_\pi$ and $\ol{\supp} \subseteq \supp p$, it follows that also $p \, \uhr \, \ol{\supp} \in D$. \end{itemize} Indeed, for any map $\pi \in A_1$, we will have a finite \tbl support\tbr\,  $supp \, \pi \subseteq Succ^\prime$, and for any $\kappa^\plus \in \supp \pi$ a \tbl domain\tbr\, $dom \, \pi (\kappa^\plus) = dom_x \,\pi (\kappa^\plus)\, \times\, dom_y \,\pi (\kappa^\plus)$ with $|dom \, \pi (\kappa^\plus)| < \kappa^\plus$ such that \[D_\pi = \{p \in \m{P}_1\ | \ \forall\, \kappa^\plus \in \supp \pi\, \cap \, \supp p\ \ \dom p(\kappa^\plus) \supseteq \dom \pi (\kappa^\plus)\}\hspace*{1,5cm} (\ast).\]

For any $\pi, \sigma \in A_1$, $\pi: D_\pi \rightarrow D_\pi$, $\sigma: D_\sigma \rightarrow D_\sigma$, there will be a map $\vartheta: D_\vartheta \rightarrow D_\vartheta$ in $A_1$ with $\vartheta = \sigma \circ \pi$ on $D_\vartheta = D_\pi \cap D_\sigma$. \\
Also, for any $\pi \in A_1$, there will be a map $\pi^{-1}$ in $A_1$ with $D_{\pi^{-1}} = D_\pi$ and $\pi \circ \pi^{-1} = \pi^{-1} \circ \pi = \id_{D_\pi}$. \\ We call this structure a \textit{group of partial $\m{P}_1$-automorphisms}.\\[-2mm]

For any $D = D_\pi$ with $(\ast)$ as above, we define a hierarchy $\ol{\Name (\m{P}_1)}^D_\alpha$ recursively as follows: \begin{itemize} \item $\ol{\Name (\m{P}_1)}^D_0 := \emptyset$, \item $\ol{\Name (\m{P}_1)}^D_{\alpha + 1}:= \{\dot{x} \in \Name (\m{P}_1) \ | \ x \subseteq \ol{\Name (\m{P})}_\alpha^D \, \times \, D\}$, and \item $\ol{\Name (\m{P}_1)}^D_\lambda := \bigcup_{\alpha  < \lambda} \ol{\Name (\m{P}_1)}^D_\alpha$ for $\lambda$ a limit ordinal. \end{itemize}

Let \[\ol{\Name (\m{P}_1)}^D := \bigcup_{\alpha \in \Ord} \ol{\Name (\m{P}_1)}^D_{\alpha}.\]  

Whenever $\pi: D_\pi \rightarrow D_\pi$ and $\dot{x} \in \ol{\Name (\m{P}_1)}^{D_\pi}$, we can define $\pi \dot{x}$ as usual. \\ However, for a $\m{P}_1$-name $\dot{x} \notin \ol{\Name(\m{P}_1)}^{D_\pi}$, it is not clear how to apply $\pi$, so we need an extension $\ol{\dot{x}}^{D_\pi}$: 

Let $D \subseteq \m{P}_1$ with property $(\ast)$ as above. Define recursively: \[ \ol{\dot{x}}^D := \{(\ol{\dot{y}}^D, \ol{p})\ | \ \exists\,(\dot{y}, p) \in \dot{x}\, :\ \, \ol{p} \leq p, \; \ol{p} \in D, \; \supp \ol{p} = \supp p \}.\]
Then for any $V$-generic filter $G_1$ on $\m{P}_1$ it follows that $(\ol{\dot{x}}^D)^{G_1} = \dot{x}^{G_1}$.\\[-2mm]

Let $\pi: D_\pi \rightarrow D_\pi$, $\sigma: D_\sigma \rightarrow D_\sigma$ as above.  
It is not difficult to verify the following properties:

\begin{itemize} \item For any $\dot{x} \in \Name (\m{P}_1)$, $\ol{\ol{\dot{x}}^{D_\pi}}^{D_\sigma} = \ol{\dot{x}}^{D_\pi \cap D_\sigma}$. \item Whenever $\dot{x} \in \ol{\Name(\m{P}_1)}^{D_\pi}$, then also $\ol{\dot{x}}^{D_\sigma} \in D_\pi$ with $\pi \ol{\dot{x}}^{D_\sigma} = \ol{\pi \dot{x}}^{D_\sigma}$. 
\end{itemize}

If $G_1$ is $V$-generic on $\m{P}_1$, then for any $\kappa^\plus \in Succ^\prime$, $i < F(\kappa^\plus)$, the generic $\kappa^\plus$-subset \[(G_1)_{(\kappa^\plus, i)} := \{ \zeta \in [\kappa, \kappa^\plus)\ | \ \exists\, p \in G_1\ p (\kappa^\plus) (\zeta, i) = 1\}\] has the canonical name \[(\dot{G}_1)_{(\kappa^\plus, i)} := \{ (\zeta, p)\ | \ \zeta \in [\kappa, \kappa^\plus), p \in \m{P}_1 \uhr (\kappa^\plus + 1), p (\kappa^\plus) (\zeta, i) = 1\}.\]

Firstly, we want that for any two of these generic $\kappa^{\plus}$-subsets $(G_1)_{(\kappa^\plus, i)}$ and $(G_1)_{(\kappa^\plus, i^\prime)}$, there is an automorphism in $A_1$ interchanging them. In other words: We want to include into $A_1$ the collection of all $\pi = (\pi (\kappa^\plus) \ | \ \kappa^\plus \in \supp \pi)$  
with $D_\pi = \m{P}_1$ such that for any $\kappa^\plus \in \SUPP \pi$, there is a bijection $f_\pi (\kappa^\plus)$ on a finite set $supp\, \pi (\kappa^\plus) \subseteq F(\kappa^\plus)$ with $\pi (\dot{G}_1)_{(\kappa^\plus, i)} =  (\dot{G}_1)_{(\kappa^\plus, f_\pi (\kappa^\plus) (i))}$ for all $i \in \supp \pi (\kappa^\plus)$. \\ For these automorphisms $\pi$, we will have $\pi p (\kappa^\plus) (\zeta, i) = p (\kappa^\plus) \big(\zeta, f_\pi (\kappa^\plus) (i)\big)$ whenever $p \in \m{P}_1$ and $\zeta \in [\kappa, \kappa^\plus)$, $i \in \SUPP \pi (\kappa^\plus)$. For all the remaining $\kappa^\plus$ and $(\zeta, i)$, we will have $\pi p (\kappa^\plus) (\zeta, i) = p(\kappa^\plus) (\zeta, i)$. \\[-3mm]

Also, we want that for any $p, q \in \m{P}_1$, there is an automorphism $\pi \in A_1$ with $\pi p \, \| \, q$. These $\pi$ will be of the following form:
For any $\kappa^\plus \in \supp \pi$,    there is  
$\dom \pi (\kappa^\plus) = \dom_x \pi (\kappa^\plus) \, \times \, \dom_y \pi (\kappa^\plus) \subseteq [\kappa, \kappa^\plus) \times F(\kappa^\plus)$ with $|\dom \pi (\kappa^\plus)| < \kappa^\plus$, and a collection \[ \big(\pi (\kappa^\plus) (\zeta, i)\ | \ (\zeta, i) \in \dom \pi (\kappa^\plus)\big) \in 2^{\dom \pi (\kappa^\plus)}, \] such that $\pi$ changes the values $p (\kappa^\plus) (\zeta, i)$ if and only if $\pi (\kappa^\plus) (\zeta, i) = 1$; i.e., for any condition $p \in D_\pi$, we will have $\pi p (\kappa^\plus) (\zeta, i) \neq p (\kappa^\plus) (\zeta, i)$ whenever $\pi (\kappa^\plus) (\zeta, i) = 1$, and $\pi p (\kappa^\plus) (\zeta, i) = p (\kappa^\plus) (\zeta, i)$ in the case that $\pi (\kappa^\plus) (\zeta, i) = 0$ or $(\zeta, i) \notin \dom \pi (\kappa^\plus)$. \\[-3mm]

$A_1$ will be generated by those two types of automorphisms. \\[-3mm]

All the $(\zeta, i)$ with $(\zeta, i) \in \dom \pi(\kappa^\plus)$ \textit{and} $i \in \supp \pi (\kappa^\plus)$ will have to be treated seperately: Namely, for any $\zeta \in \dom_x \pi (\kappa^\plus)$, we will have a bijection $\pi (\kappa^\plus) (\zeta)$ which maps the sequence $(p (\zeta, i)\ | \ i \in \supp \pi(\kappa^\plus))$ to $((\pi p) (\zeta, i)\ | \ i \in \supp \pi (\kappa^\plus))$. \\ These bijections $\pi (\kappa^\plus) (\zeta)$ will be necessary to retain a group structure. \\[-3mm]

This results in the following definition:

\begin{definition} \label{a1} $A_1$ consists of all automorphisms $\pi: D_\pi \rightarrow D_\pi$, $\pi = (\pi (\kappa^\plus)\ | \ \kappa^\plus \in \supp \pi)$ with finite support $\supp \pi \subseteq Succ^\prime$ such that for all $\kappa^\plus \in \supp \pi$, there are 
\begin{itemize} \item a finite set $\supp \pi (\kappa^\plus) \subseteq F(\kappa^\plus)$ with a bijection $f_\pi (\kappa^\plus): \supp \pi (\kappa^\plus) \rightarrow \supp \pi (\kappa^\plus)$, \item a domain $\dom \pi (\kappa^\plus) = \dom_x \pi (\kappa^\plus) \times \dom_y \pi (\kappa^\plus) \subseteq [\kappa, \kappa^\plus)\, \times\, F(\kappa^\plus)$ with $|\dom \pi (\kappa^\plus)| < \kappa^\plus$ such that $\supp \pi (\kappa^\plus) \subseteq \dom_y \pi (\kappa^\plus)$, and a collection $\big( \pi (\kappa^\plus) (\zeta, i)\ | $ $\ (\zeta, i) \in [\kappa, \kappa^\plus) \times F(\kappa^\plus) \big)$ 
%in 2^{[\kappa, \kappa^\plus) \times F(\kappa^\plus)}
with $\pi (\kappa^\plus) (\zeta, i) \in 2$ for all $(\zeta, i)$, and $\pi (\kappa^\plus) (\zeta, i) = 0$ whenever $(\zeta, i) \notin \dom \pi (\kappa^\plus)$, \mbox{\normalfont and }\item for any $\zeta \in \dom_x \pi (\kappa^\plus)$, a bijection $\pi (\kappa^\plus) (\zeta): 2^{\supp \pi (\kappa^\plus)} \rightarrow 2^{\supp \pi (\kappa^\plus)}$
\end{itemize}

such that $D_\pi := \{p \in \m{P}_1\ | \ \forall \kappa^\plus \in \supp p\, \cap \, \supp \pi\ \,\dom p (\kappa^\plus) \supseteq \dom \pi (\kappa^\plus)\}$, and
for any $p \in D_\pi$, the condition $\pi p$ is defined as follows: \\[-3mm]

We will have $\supp (\pi p) = \supp p$ with $\pi p (\kappa^\plus) = p(\kappa^\plus)$ whenever $\kappa^\plus \in \supp p \, \setminus \, \supp \pi$. \\ Let now $\kappa^\plus \in \supp p\, \cap\, \supp \pi$.

\begin{itemize} \item For any $i \in \supp \pi (\kappa^\plus)$ and $\zeta \notin \dom_x \pi (\kappa^\plus)$,
we have \[\pi p (\kappa^\plus) (\zeta, i) = p (\kappa^\plus) \big(\zeta, f_\pi (\kappa^\plus) (i)\big).\] \item For $\zeta \in \dom_x \pi (\kappa^\plus)$, \[\big( \pi p(\kappa^\plus)(\zeta, i)\ | \ i \in \supp \pi (\kappa^\plus)\big) = \pi (\kappa^\plus) (\zeta) \big(p (\kappa^\plus) (\zeta, i)\ | \ i \in \supp \pi (\kappa^\plus) \big).\]\item Whenever $i \notin \supp \pi (\kappa^\plus)$, then $\pi p(\kappa^\plus) (\zeta, i) := p (\kappa^\plus) (\zeta, i)$ if $\pi (\kappa^\plus) (\zeta, i) = 0$, and $\pi p(\kappa^\plus) (\zeta, i) \neq p (\kappa^\plus) (\zeta, i)$ in the case that $\pi (\kappa^\plus) (\zeta, i) = 1$. \end{itemize} \end{definition}

In other words: Outside the domain $\dom \pi (\kappa^\plus)$, we have a swap of the horizontal lines $p (\kappa^\plus) (\cdot\, , i)$ for $i \in \supp \pi (\kappa^\plus)$ according to $f_\pi (\kappa^\plus)$. \\Inside $\dom \pi (\kappa^\plus)$, the values $\pi p (\kappa^\plus) (\zeta, i)$ for $i \in \supp \pi (\kappa^\plus)$ are determined by the maps $\pi (\kappa^\plus) (\zeta)$, while any of the remaining values $\pi p (\kappa^\plus) (\zeta, i)$ with $i \notin \supp \pi (\kappa^\plus)$ is equal to $p(\kappa^\plus) (\zeta, i)$ if and only if $\pi (\kappa^\plus) (\zeta, i) = 1$. \\[-3mm]

We use the dense sets $D_\pi$ to make sure that $\dom p (\kappa^\plus)$ is not mixed up by $\pi$: Since $\dom p (\kappa^\plus) \supseteq \dom \pi(\kappa^\plus)$ for all $\kappa^\plus \in \supp p\, \cap \, \supp \pi$, it follows that $\dom \pi p (\kappa^\plus) = \dom p(\kappa^\plus)$. \\[-3mm]

It is not difficult to verify that any $\pi \in A_1$ is indeed a $\leq$-preserving bijection on $D_\pi$. The inverse maps $\pi^{-1}$ and concatenations $\pi \circ \sigma$ can be written down explicitly, using the definition above. Hence, $A_1$ is indeed a \textit{group of partial $\m{P}_1$-automorphisms}. \\[-2mm]

For a map $\pi \in A_1$, $\pi: D_\pi \rightarrow D_\pi$ and $\kappa^\plus \in \supp \pi$, the domain $dom\, \pi (\kappa^\plus) = \dom_x \pi (\kappa^\plus) \, \times\, \dom_y \pi (\kappa^\plus) $ is determined by $D_\pi = \{p \in \m{P}_1\ | \ \forall\, \kappa^\plus \in \supp \pi\, \cap\, \supp p\ \dom p(\kappa^\plus) \supseteq \dom \pi (\kappa^\plus)\}$, and the maps $\big(\pi (\kappa^\plus) (\zeta, i)\ | \ (\zeta, i) \in [\kappa, \kappa^\plus)\, \times\, F(\kappa^\plus)\big)$ for $i \notin \supp \pi (\kappa^\plus)$ are given by $\pi$, as well. \\ However,  
if $\supp  \pi (\kappa^\plus) \subseteq \dom_y \pi (\kappa^\plus)$ with maps $f_\pi (\kappa^\plus) : \supp \pi (\kappa^\plus) \rightarrow \supp \pi (\kappa^\plus)$ and $\pi (\zeta): 2^{\supp \pi (\kappa^\plus)} \rightarrow 2^{\supp \pi (\kappa^\plus)}$ satifies the properties from Definition \ref{a1}, 
then any finite set $\ol{\supp}\, \pi (\kappa^\plus)$ with $\supp \, \pi (\kappa^\plus) \subseteq \ol{\supp}\, \pi (\kappa^\plus) \subseteq \dom_y (\kappa^\plus)$ is suitable for describing $\pi$ as well, after the maps $f_\pi (\kappa^\plus)$ and $\pi (\kappa^\plus) (\zeta)$ have been modified accordingly.

Thus, when talking about $supp \, \pi (\kappa^\plus)$ for some map $\pi \in A_1$, we will always think of $supp \, \pi (\kappa^\plus) \subseteq \dom_y \pi (\kappa^\plus)$ least with the property that there are maps $f_\pi (\kappa^\plus)$ and $\pi (\kappa^\plus) (\zeta)$ describing $\pi$ as in Definition \ref{a1}. Then the attributes $f_\pi (\kappa^\plus)$ and $(\pi (\kappa^\plus) (\zeta)\ | \ \zeta \in \dom_x \pi (\kappa^\plus))$ are unique, as well. \\[-2mm]

Let $A := A_0 \times A_1$ denote the collection of all automorphisms $\pi = (\pi_0, \pi_1)$, $\pi: D_\pi \rightarrow D_\pi$ with $D_\pi := \m{P}_0 \times D_{\pi_1}$. \\[-3mm] 

We will now define the $A$-subgroups that generate our symmetric model $N$, where by \textit{$A$-subgroup}, we mean in this context a subclass $B \subseteq A$ that is a group of partial automorphisms. \\ We will have $A$-subgroups of the form $B_0 \, \times\, A_1$ for an $A_0$-subgroup $B_0 \subseteq A_0$, and $A_0 \, \times\, B_1$ for an $A_1$-subgroup $B_1 \subseteq A_1$. \\[-3mm]

We start with $A_0$. \\[-4mm]

Firstly, for any $\kappa \in \Card, i < F_{\lim} (\kappa)$, we want to include the subgroup \[Fix_0(\kappa, i) := \{\pi_0 \in A_0\ | \ \pi_0(\kappa) (\kappa, i) = (\kappa, i)\},\] 
which makes sure that any canonical name $(\dot{G}_0)_{(\kappa, i)}$  is hereditarily symmetric, since $\pi_0 (\dot{G}_0)_{(\kappa, i)} = (\dot{G}_0)_{(\kappa, i)}$ for all $\pi_0 \in Fix_0 (\kappa, i)$. Thus, our model $N$ will contain any of the adjoined $\kappa$-subsets $(G_0)_{(\kappa, i)}$ given by the branches through the generic tree. \\[-3mm]

For any cardinal $\kappa$ and $\alpha < F_{\lim} (\kappa)$, we want in $N$ 
a surjection $s: \powerset(\kappa) \rightarrow \alpha$; which gives $\theta^N (\kappa) \geq F (\kappa)$ for all limit cardinals $\kappa$. 
However, we have to make sure that $\theta^N (\kappa) < F(\kappa)^\plus$; so the sequence $\big( (G_0)_{(\kappa, i)}\ | \ i < F (\kappa)\big)$ must not be contained in $N$. \\[-3mm]
     
Therefore, for cardinals $\kappa$ and $\alpha < F_{\lim} (\kappa)$ a limit ordinal, 
we consider the subgroup $Small_0(\kappa, [0, \alpha))$ containing all the $\m{P}_0$-automorphisms $\pi_0$ with the property that $\pi_0 (\kappa)$ is \textit{small below} $\alpha$, i.e. for any $i < \alpha$, it follows that $\pi_0 (\kappa) (\kappa, i) = (\kappa, j)$ for some $j$ such that $i, j \in [\gamma (i), \gamma (i) + \omega)$ for a limit ordinal $\gamma (i)$:
\[ Small_0 (\kappa, [0, \alpha)) := \{\pi_0 \in A_0\ | \ \forall\, i < \alpha, \; i \in [\gamma (i), \gamma (i) + \omega) \mbox{ with } \gamma (i) \mbox{ a limit ordinal: } \] \[ \pi_0 (\kappa) (\kappa, i) = (\kappa, j) \mbox{ for some } j \in [\gamma (i), \gamma (i) + \omega) \}.\]

Now, for any limit ordinal $i < \alpha$, we can define a \tbl cloud\tbr\, around $(\dot{G}_0)_{(\kappa, i)}$ as follows: \[(\dot{\wt{G_0}})_{(\kappa, i)}^\alpha := \big\{ \, \big(\pi (\dot{G_0})_{(\kappa, i)}, \m{1}\big)\ | \ \pi \in Small_0 (\kappa, [0, \alpha))\, \big\} = \] \[ = \big\{ \,\big( (\dot{G_0})_{(\kappa, i + n)}, \m{1} \big)\ | \ n < \omega\, \big\}.\] 

Then $(\wt{G_0})_{(\kappa, i)}^\alpha := \big ( (\dot{\wt{G_0}})_{(\kappa, i)}^\alpha \big)^G$ is the set of all $(G_0)_{(\kappa, i + n)}$ for $n < \omega$; so any two distinct clouds $(\wt{G_0})_{(\kappa, i)}^\alpha$ and $(\wt{G_0})_{(\kappa, j)}^\alpha$ for limit ordinals $i, j < \alpha$ are indeed disjoint. Hence, the sequence $\big( (\wt{G_1})_{(\kappa, i)}^\alpha\ | \ i < \alpha \mbox{ a limit ordinal }\big)$, which has a canonical symmetric name stabilized by all $\pi \in Small_0 (\kappa, [0, \alpha))$, gives a surjection $s: \powerset (\kappa) \rightarrow \alpha$ in $N$. \\[-2mm]

Note that the subgroups $Fix_0 (\kappa, i) \subseteq A_0$ and $Small_0 (\kappa, [0, \alpha)) \subseteq A_0$ are not normal. However, one can check that the collection of the $Fix_0 (\kappa, i)$ and $Small_0 (\kappa, [0, \alpha))$ satisfies the \textit{normality property} mentioned in Chapter \ref{class forcing}. \\[-3mm]

Now, we turn to $A_1$. 
For any $\kappa \in Succ^\prime$, $\kappa = \ol{\kappa}^\plus$ and $i < F(\kappa)$, we want to include the $A_1$-subgroup
%into $\mathcal{F}_1$ 
\[Fix_1 (\kappa, i) := \{\pi \in A_1\ | \ \forall\, p \in D_\pi\ (\pi p) \, \uhr \, \{(\kappa, i)\} = p\, \uhr \, \{(\kappa, i)\}\}.\]This makes sure that any generic $\kappa$-subset $(G_1)_{(\kappa, i)}$  
is contained in our eventual symmetric submodel $N$, since 
$\pi (\dot{G}_1)_{(\kappa, i)} = (\dot{G}_1)_{(\kappa, i)}$ for all $\pi \in Fix_1 (\kappa, i)$. \\Again, we have to make sure that the sequence $\big( (G_1)_{(\kappa, i)}\ | \ i < F(\kappa)\big)$ is not included into $N$, in order to achieve $\theta^N (\kappa) \leq F (\kappa)$. 
On the other hand, we need surjections $s: \powerset (\kappa) \rightarrow \alpha$ for all $\alpha < F(\kappa)$. Thus, similarly as before, for $\kappa \in Succ^\prime$, $\alpha < F(\kappa)$, let \[Small_1 (\kappa, [0, \alpha)) := \{\pi \in A_1\ | \ \forall\, i < \alpha\ \; i \notin \supp \pi (\kappa)\}.\]  

Then $Small_1 (\kappa, [0, \alpha))$ does not contain any of those automorphisms that interchange some $(\dot{G}_1)_{(\kappa, i)}$ and $(\dot{G}_1)_{(\kappa, j)}$ for $i, j < \alpha$. Thus, for any $i < \alpha$, we can define a \tbl cloud\tbr \,$(\wt{G_1})_{(\kappa, i)}^\alpha$ around $ (G_1)_{(\kappa, i)}$ with the symmetric name \[(\dot{\wt{G_1}})_{(\kappa, i)}^\alpha := \big\{ \, \big( \pi (\dot{G_1})_{(\kappa, i)}, \m{1}\big)\ | \ \pi \in Small_1 (\kappa, [0, \alpha))\, \big\}\] such that for $(\wt{G_1})_{(\kappa, i)}^\alpha := \big( (\dot{\wt{G_1}})_{(\kappa, i)}^\alpha \big)^G$, it follows that any two distinct clouds $(\wt{G_1})_{(\kappa, i)}$ and $(\wt{G_1})_{(\kappa, j)}$ are indeed disjoint. Hence, the sequence $\big( (\wt{G_1})_{(\kappa, i)}\ | \ i < \alpha\big)$, which has a symmetric name stabilized by all $\pi \in Small_1 (\kappa, [0, \alpha))$, gives a surjection $s: \powerset (\kappa) \rightarrow \alpha$ in $N$.\\[-3mm]   

Again, the normality property holds. \\[-3mm]       

\begin{definition} A name $\dot{x}$ is {\normalfont symmetric}, whenever there are $n, m, \ol{n}, \ol{m} < \omega$ and \begin{itemize} \item $\kappa_0, \, \ldots \,, \kappa_{n-1} \in \Card$, $i_0 < F_{\lim} (\kappa_0), \, \ldots \,, i_{n-1} < F_{\lim} (\kappa_{n-1})$ \item $\lambda_0, \, \ldots \,, \lambda_{m-1} \in \Card$, $\alpha_0 < F_{\lim} (\lambda_0), \, \ldots \,, \alpha_{m-1} < F_{\lim} (\lambda_{m-1})$ \item $\ol{\kappa}_0, \, \ldots \,, \ol{\kappa}_{\ol{n}-1} \in \Succ^\prime$, $\ol{\imath}_0 < F (\ol{\kappa}_0), \, \ldots \,, \ol{\imath}_{\ol{n}-1} < F (\ol{\kappa}_{\ol{n}-1})$ \item $\ol{\lambda}_0, \, \ldots \,, \ol{\lambda}_{\ol{m}-1} \in \Succ^\prime,  \ol{\alpha}_0 < F (\ol{\lambda}_0), \, \ldots \,, \ol{\alpha}_{\ol{m}-1} < F (\ol{\lambda}_{\ol{m}-1})$ \end{itemize} such that $\{\pi \in A\ | \ \pi \ol{\dot{x}}^{D_\pi} = \ol{\dot{x}}^{D_\pi}\}$ is a superset of the following intersection: \[Fix_0 (\kappa_0, i_0)\, \cap\, \cdots\, \cap\, Fix_0 (\kappa_{n-1}, i_{n-1})\, \cap\, Small_0 (\lambda_0, [0, \alpha_0))\, \cap\, \cdots\, \cap \, Small_0 (\lambda_{m-1}, [0, \alpha_{m-1}))\, \cap\] \[ \cap\, Fix_1 (\ol{\kappa}_0, \ol{\imath}_0)\, \cap\, \cdots\, \cap\, Fix_1 (\ol{\kappa}_{\ol{n}-1}, \ol{\imath}_{\ol{n}-1})\, \cap\, Small_1 (\ol{\lambda}_0, [0, \ol{\alpha}_0))\, \cap\, \cdots\, \cap \, Small_1 (\ol{\lambda}_{\ol{m}-1}, [0, \ol{\alpha}_{\ol{m}-1})).\]

Recursively, a name $\dot{x}$ is {\normalfont hereditarily symmetric}, $\dot{x} \in HS$, if $\dot{x}$ is symmetric, and $\dot{y} \in HS$ for all $(\dot{y}, p) \in \dot{x}$. \end{definition}  

\vspace*{2mm}
The following properties are not difficult to verify:
\begin{itemize} \item If $\dot{x} \in HS$ and $\pi \in A$, it follows that also $\ol{\dot{x}}^{D_\pi} \in HS$. 
\item If $\dot{x} \in HS$ and $\pi \in A$ with $\dot{x} \in \ol{\Name(\m{P})}^{D_\pi}$, then also $\pi \dot{x} \in HS$. \end{itemize}

\vspace*{2mm}

For $G$ a $V$-generic filter on $\m{P}$, our symmetric submodel is defined as follows: \[V(G) := \{\dot{x}^G\ | \ \dot{x} \in HS\}.\]
\section{The symmetric submodel} 

Fix a $V$-generic filter $G$ on $\m{P}$, and let $N := V(G)$ as defined in section \ref{symmetric names}. \\We claim that $N$ satifies the statement from our theorem, i.e. $N \vDash ZF$, $N$ preserves all $V$-cardinals, and $\theta^N (\kappa) = F(\kappa)$ for all $\kappa$. \\[-3mm]

In this section, we will verify that $N$ is indeed a model of $ZF$, although the class forcing $\m{P}$ does not preserve $ZFC$. Later on, we will see that any set of ordinals located in $N$ can be captured in a \tbl mild\tbr\,$V$-generic extension that preserves cardinals and the $GCH$. \\[-3mm]

For $\alpha \in \Ord$, let $\m{P}_\alpha := (\m{P}_0)_\alpha\, \times\, (\m{P}_1)_\alpha$. Then \[\m{P} = \bigcup \{ \m{P}_{\alpha} \ | \ \alpha \in \Ord\}\] is an increasing union of set-sized complete subforcings, so the forcing theorem holds for $\Vdash_{\m{P}}^{V, \check{V}, \dot{G}}$ and the symmetric forcing relation $(\Vdash_s)_{\m{P}}^{V, \check{V}}$. \\ Also, it follows that any $G_{\alpha} := \{p \in G\ | \ p \in \m{P}_{\alpha}\}$ is a $V$-generic filter on $\m{P}_{\alpha}$. \\[-3mm] 

This implies that both the generic extension $V[G]$ and the symmetric submodel $N$ satisfy the axioms of Extensionality, Foundation, Pairing, Union and Infinity. 

\begin{prop} \label{separation} The Axiom of Separation holds in $\left\langle V[G], \in, V\right\rangle$ and $\left\langle N, \in, V\right\rangle $ for every $\mathcal{L}_\in^{A}$-formula $\varphi (v_0, \ldots, v_{n-1})$. 
\end{prop} 
\begin{proof} We first consider $V[G]$. Let $a \in V[G]$ and $\varphi(v_0, \ldots, v_{n-1}) \in \mathcal{L}_\in^A$. 
W.l.o.g. assume $n = 1$ and take a parameter $z := z_0$ in $V[G]$. We have to show that there is $b \in V[G]$ with \[b = \big\{x \in a \  \big| \ \left\langle V[G], \in, V\right\rangle \, \vDash \varphi (x, z)\big\}.\] Take a cardinal $\lambda$ large enough such that there are names $\dot{a}, \dot{z} \in \Name (\m{P} \uhr (\lambda + 1))$  
with $a = \dot{a}^{G \, \uhr \, (\lambda + 1)}$, $z = \dot{z}^{G \, \uhr \, (\lambda + 1)}$. \\ Let \[\dot{b} := \{ (\dot{x}, p)\ | \ \dot{x} \in \dom\, \, \dot{a}, \, p \in \m{P} \uhr (\lambda + 1), p \Vdash^{V, \check{V}}_{\m{P}} (\dot{x} \in \dot{a}\, \wedge\, \varphi (\dot{x}, \dot{z}) )\, \}.\]  
We claim that $\dot{b}^G = b$. The direction \tbl $\subseteq$\tbr\, is clear. Concerning \tbl $\supseteq$\tbr\,, consider $x \in b$.
Let $\dot{x} \in \dom \dot{a}$ with $x  = \dot{x}^G$ and $\ol{p} \in G$ with \[\ol{p} \Vdash^{V, \check{V}}_{\m{P}} (\dot{x} \in \dot{a}\, \wedge\, \varphi (\dot{x}, \dot{z})).\] Let $p := \ol{p} \uhr (\lambda + 1)$. It suffices to verify also $p \Vdash^{V, \check{V}}_{\m{P}} (\dot{x} \in \dot{a}\, \wedge\, \varphi (\dot{x}, \dot{z}))$. If not, there would be $q \in \m{P}$, $q \leq p$ with \[q \Vdash^{V, \check{V}}_{\m{P}} \neg \, \big (\dot{x} \in \dot{a}\, \wedge\, \varphi(\dot{x}, \dot{z})\big).\] We construct a $\m{P}$-automorphism $\pi$ with $\pi \ol{p} \, \| \, q$ such that $\pi$ is the identity on $\m{P} \uhr (\lambda + 1)$. Then $\pi \ol{\dot{x}}^{D_\pi} = \ol{\dot{x}}^{D_\pi}$, $\pi \ol{\dot{a}}^{D_\pi} = \ol{\dot{a}}^{D_\pi}$ and $\pi \ol{\dot{z}}^{D_\pi} = \ol{\dot{z}}^{D_\pi}$; hence, \[\pi \ol{p} \Vdash^{V, \check{V}}_{\m{P}} \big(\ol{\dot{x}}^{D_\pi} \in \ol{\dot{a}}^{D_\pi} \, \wedge\, \varphi (\ol{\dot{x}}^{D_\pi}, \ol{\dot{z}}^{D_\pi})\big), \] contradicting that $\pi \ol{p}\, \| \, q$. \\[-3mm]
 
We start with $\pi_0$. Let $\HT \pi_0 := \max \{ \eta (\ol{p}), \eta (q)\}$. 
For $\alpha \leq \lambda$, let $\pi_0 (\alpha)$ be the identity. For $\lambda^\plus \leq \alpha \leq \HT \pi_0$, take for $\pi_0 (\alpha)$ a bijection on $\{(\alpha, i)\ | \ i < F_{\lim} (\alpha) \}$ with finite support such that for any $(\alpha, i) \in t(\ol{p})$, it follows that $\pi_0(\alpha) (\alpha, i) = (\alpha, j)$ for some $(\alpha, j) \notin t(\ol{p})\, \cup \, t(q)$.
Then from $q \leq \ol{p} \uhr (\lambda + 1)$ it follows that $\pi_0 \ol{p}_0 \, \| \, q_0$. \\[-3mm]

Now, we turn to $\pi_1$. Let $\supp \pi_1 := \supp \ol{p}_1\, \cup\, \supp q_1$. For $\alpha^\plus \in \supp \pi_1$ with $\alpha^\plus \leq \lambda$, let $\pi_1 (\alpha^\plus)$ be the identity. For $\alpha^\plus \in \supp \pi_1$ with $\alpha^\plus > \lambda$, we define $\pi_1 (\alpha^\plus)$ as follows: Let $\dom \pi_1 (\alpha^\plus) := \dom \ol{p}_1 (\alpha^\plus)\, \cap \, \dom q_1 (\alpha^\plus)$ and $\SUPP \pi_1 (\alpha^\plus) = \emptyset$; then we only need to define $\pi_1 (\alpha^\plus)(\zeta, i)$ for $\zeta \in \dom_x \ol{p}_1 (\alpha^\plus)\, \cap \, \dom_x q_1 (\alpha^\plus)$, $i \in \dom_y \ol{p}_1 (\alpha^\plus)\, \cap \, \dom_y q_1 (\alpha^\plus)$. Let $\pi_1 (\alpha^\plus) (\zeta, i) = 0$ if $\ol{p}_1 (\alpha^\plus) (\zeta, i) = q_1 (\alpha^\plus) (\zeta, i)$, and $\pi_1 (\alpha^\plus) (\zeta, i) = 1$ in the case that $\ol{p}_1 (\alpha^\plus) (\zeta, i) \neq q_1 (\alpha^\plus) (\zeta, i)$. Then $\pi_1 \ol{p}_1 \, \| \, q_1$. \\[-3mm]

Hence, our automorphism $\pi = (\pi_0, \pi_1)$ is as desired. \\[-3mm]

This proves Separation in $V[G]$ for any $\mathcal{L}_\in^A$-formula $\varphi$. \\[-3mm]

The proof for $N$ is similar, using symmetric names and the symmetric forcing relation $(\Vdash_s)^{V, \check{V}}_{\m{P}}$.

\end{proof}   

Now, in order to show that Replacement holds in $N$, it is enough to verify the Axiom Scheme of Collection (and then use Separation):

\begin{prop} \label{Replacement} For any $\mathcal{L}_\in^{A}$-formula $\varphi (x, y, v_0, \, \ldots \,, v_{n-1})$ and $a$, $z_0, \, \ldots \,, z_{n-1} \in N$ such that \[\left \langle N, \in, V \right\rangle \, \vDash \forall  x\in a\ \exists y \ \varphi (x, y, z_0, \, \ldots \,, z_{n-1}),\] there exists $b \in  N$ with the property that\[ \left\langle N, \in, V \right \rangle \ \vDash \ \forall x \in a \ \exists y \in  b\ \varphi (x, y, z_0, \, \ldots \,, z_{n-1}).\] 
\end{prop}

\begin{proof} For an ordinal $\alpha$ and the set forcing $\m{P}_\alpha$ as above, the $\Name_\beta (\m{P}_\alpha)^V$-hierarchy is defined recursively (in $V$) as usual: $\dot{x} \in \Name_{\beta + 1} (\m{P}_\alpha)^V$ iff $\dot{x} \subseteq \Name_\beta (\m{P}_\alpha)^V\, \times\, \m{P}_\alpha$, and for $\lambda$ a limit ordinal, $\dot{x} \in \Name_\lambda (\m{P}_\alpha)^V$ iff $\dot{x} \in \Name_\beta (\m{P}_\alpha)^V$ for some $\beta < \lambda$. \\
We are going to use the following \tbl diagonal hierarchy\tbr: For $\alpha \in \Ord$, let \[N_\alpha := \{\dot{x}^{G_{\alpha}}\ | \ \dot{x} \in HS\, \cap\, \Name_{\alpha + 1} (\m{P}_{\alpha})^V\}.\]
One has to check that this hierarchy is indeed definable in the structure $\left\langle V[G], \in, V, G \right\rangle$, i.e. there is an $\mathcal{L}_{\in}^{A, B}$-formula $\tau$ such that $\left\langle V[G], \in, V, G\right\rangle\  \vDash \tau (x, \alpha)$ iff $\alpha = \min \{\beta\ | \ x \in N_\beta\}$. 
Therefore, one first has to make sure that the interpretation function $(\cdot)^G$ is definable within $\left \langle V[G], \in, V, G \right \rangle$, where some extra care is needed, since the recursion theorem can only be applied very carefully (we do not have replacement in $V[G]$).  

However, for evaluating $\m{P}_\alpha$-names, it will be sufficient that replacement holds inside the set-generic extension $V[G_\alpha]$ (cf. \cite[Lemma 4.4]{Gitik}):
One can mimic the proof of the recursion theorem in $ZFC$ to construct function $f(x, y)$ in $\left \langle V[G], \in, V, G \right\rangle$ with $f(\alpha, \dot{v}) = v$ if and only if $\dot{v} \in \Name_{\alpha + 1} (\m{P}_{\alpha})^V$ and $v = \dot{v}^{G_{\alpha}}$. At some points, where we would like to use replacement to make sure that certain terms are indeed a set, we realize (recursively) that everything constructed so far could also be constructed inside $V[G_\alpha]$, so we can apply replacement there. \\[-3mm]

This function $f(\alpha, \dot{v})$ can be used to define our $N_\alpha$-hierarchy: Let $\tau (x, \alpha)$ be the formula \[\alpha = \min \{ \beta\ | \ \exists\, \dot{x} \in HS\, \cap \, \Name_{\beta + 1} (\m{P}_{\beta})^V\ \; x = f(\beta, \dot{x})\}.\] Then $\left\langle V[G], \in, V, G \right\rangle \, \vDash \tau (x, \alpha)$ if and only if $\alpha = \min\{\beta\ | \ \;x \in N_\beta\}$. \\[-2mm]

Now, consider $a \in N$ and an $\mathcal{L}_\in^{A}$-formula $\varphi$ with \[\left\langle N, \in V \right\rangle \vDash \forall x \in a \ \exists y \ \varphi (x, y).\] (We suppress the parameters $z_0, \ldots, z_{n-1}$ for simplicity.)
We have to show that there exists $b \in  N$ with the property that\[ \forall x \in a \ \exists y \in  b\ \left\langle N, \in, V\right\rangle\  \vDash \varphi (x, y).\] 

First, we use structural induction over the formula $\varphi$ to construct an $\mathcal{L}_\in^{A, B}$-formula $\ol{\varphi}$ such that for all $x \in a$ and $y$,\[ \left\langle V[G], \in, V, G \right\rangle \vDash \ol{\varphi} (x, y)\] if and only if \[ \left\langle N, \in, V \right \rangle \, \vDash \varphi (x, y). \] 

Then we define in $\left\langle V[G], \in, V, G \right\rangle$:\[A := \big\{\, (x, \alpha)\ | \ x \in a\; \wedge\; \alpha = \min \{\, \beta\ | \ \exists\, y \ \exists\, \dot{y} \in HS\, \cap\, \Name_{\beta+1} (\m{P}_\beta)^V \; :\ y = f(\dot{y}, \beta)\, \wedge\, \ol{\varphi} (x, y)\, \}\, \big\}.\]

Then $A = \big\{\, (x, \alpha)\ | \ x \in a\, \wedge\, \alpha = \min \{\beta\ | \ \exists\, y \in N_\beta\ \left\langle N, \in, V \right\rangle \vDash \varphi (x, y)\}\, \big\}$. \\[-3mm]

It suffices to show that there exists $\delta$ with $\rg\, A \subseteq \delta$, since this would imply that for all $x \in a$, there exists $y \in N_\delta$ with $\left\langle N, \in, V \right\rangle\,  \vDash \varphi (x, y)$. \\[-3mm]

Take $\lambda$ large enough such that there is $\dot{a} \in HS\, \cap \, \Name (\m{P} \uhr (\lambda + 1))^V$ with $a  = \dot{a}^{G \, \uhr \, (\lambda + 1)}$. We claim that $A \in V[G \uhr (\lambda + 1)]$. \\[-3mm]

Let \[A^\prime := \big\{ (\dot{x}^{G \, \uhr \, (\lambda + 1)}, \alpha)\ | \ \dot{x} \in \dom \dot{a}, \alpha = \min\{\beta\ | \ \exists\, \dot{y} \in HS\, \cap\, \Name_{\beta+1} (\m{P}_\beta)^V \]\[\ \exists\, p\, :\ p \Vdash_{\m{P}}^{V, \check{V}, \dot{G}} \big(\dot{x} \in \dot{a}\, \wedge\, \ol{\varphi}(\dot{x}, \dot{y})\big)\, , \, p \uhr (\lambda + 1) \in G \uhr (\lambda + 1) \}\, \big\}.\]
\vspace*{1mm}

Then $A^\prime \in V[G \uhr (\lambda + 1)]$. It remains to prove that $A = A^\prime$. \\[-3mm] 

Therefore, it suffices to show that in $\left\langle V[G], \in, V, G \right\rangle$, for any $\dot{x} \in \dom \dot{a}$ and $\beta \in \Ord$ the following are equivalent: \begin{itemize} \item[(I)]$\dot{x}^{G \, \uhr \, (\lambda + 1)} \in a\, \wedge\, \exists\, \dot{y} \in HS\, \cap\, \Name_{\beta+1} (\m{P}_\beta)^V\ \exists\, y\ :\ y = f(\dot{y}, \beta) \, \wedge\, \ol{\varphi} (\dot{x}^{G\, \uhr\, (\lambda + 1)}, y)$ \item[(II)] $\exists\, \dot{y} \in HS\, \cap\, \Name_{\beta+1} (\m{P}_\beta)^V\ \exists\, p\, :\ p \Vdash_{\m{P}}^{V, \check{V}, \dot{G}} \big(\dot{x} \in \dot{a}\, \wedge\, \ol{\varphi} (\dot{x}, \dot{y})\big),$ \\ $p \uhr (\lambda + 1) \in G \uhr (\lambda + 1)$. 
\end{itemize}  

The direction \tbl (I) $\Rightarrow$ (II)\tbr\, is clear. Concerning \tbl (II) $\Rightarrow$ (I)\tbr\,, assume towards a contradiction that there was $\dot{x} \in \dom \dot{a}$, $\beta \in \Ord$ and $\dot{y} \in HS\, \cap\, \Name_{\beta+1} (\m{P}_\beta)^V$ with $p \Vdash_{\m{P}}^{V, \check{V} ,\dot{G}} \big(\dot{x} \in \dot{a}\, \wedge\, \ol{\varphi} (\dot{x}, \dot{y})\big)$ for some $p \in \m{P}$ with $p \uhr (\lambda + 1) \in G \uhr (\lambda + 1)$, but (I) fails. \\ From $p \Vdash_{\m{P}}^{V, \check{V}, \dot{G}} \dot{x} \in \dot{a}$ with $p \uhr (\lambda + 1) \in G \uhr (\lambda + 1)$ and $\dot{x}, \dot{a} \in \Name (\m{P} \uhr (\lambda + 1))^V$, it follows that $\dot{x}^{G \, \uhr \, (\lambda + 1)} \in \dot{a}^{G \, \uhr \, (\lambda + 1)} = a$; hence, \[ \left\langle V[G], V, \in, G \right\rangle \vDash \neg \big(\, \exists\, \dot{y} \in HS\, \cap\, \Name_{\beta+1} (\m{P}_\beta)^V\ \exists\, y\ :\ y = f(\dot{y}, \beta) \, \wedge\, \ol{\varphi} (\dot{x}^{G\, \uhr\, (\lambda + 1)}, y)\, \big). \]

Take $q \in G$ such that \[q \Vdash_{\m{P}}^{V, \check{V}, \dot{G}} \, \forall \dot{y} \in HS\, \cap\, \Name_{\beta + 1} (\m{P}_\beta)^V\ \forall\,y \ : \ y = f(\dot{y}, \beta) \, \longrightarrow \, \neg \ol{\varphi} (\dot{x}, y). \] As in Proposition \ref{separation}, we can construct an automorphism $\pi$ such that $\pi p \, \| \, q$, and $\pi$ is the identity on $\m{P} \uhr (\lambda + 1)$. Then $\pi \ol{\dot{x}}^{D_\pi} = \ol{\dot{x}}^{D_\pi}$; hence, \[\pi p \Vdash_{\m{P}}^{V, \check{V}, \pi \dot{G}} \, \ol{\varphi} (\dot{x}, \pi \ol{\dot{y}}^{D_\pi}).\] 

By structural induction over the formula $\ol{\varphi}$, one can use an isomorphism argument to show that for any condition $r \in \m{P}$, it follows that $r \Vdash_{\m{P}}^{V, \check{V}, \dot{G}} \ol{\varphi} (\dot{x}, \pi \ol{\dot{y}}^{D_\pi})$ if and only if $r \Vdash_{\m{P}}^{V, \check{V}, \pi \dot{G}} \ol{\varphi} (\dot{x}, \pi \ol{\dot{y}}^{D_\pi})$. The induction step regarding the existential quantifier follows from the fact that for any $\dot{v} \in HS\, \cap \, \Name_{\alpha + 1} (\m{P}_\alpha)^V$ and $\pi \in A$, also $\pi \ol{\dot{v}}^{D_\pi} \in HS\, \cap \, \Name_{\alpha+1} (\m{P}_\alpha)^V$, and $\dot{v}^H = (\pi \ol{\dot{v}}^{D_\pi})^{\pi H}$ for any $V$-generic filter $H$ on $\m{P}$. \\[-3mm]

Hence, it follows that also \[ \pi p \Vdash_{\m{P}}^{V, \check{V}, \dot{G}} \, \ol{\varphi} ( \dot{x}, \pi \ol{\dot{y}}^{D_\pi}).\] But $\pi \ol{\dot{y}}^{D_\pi} \in HS\, \cap \, \Name_{\beta + 1} (\m{P}_\beta)^V$, which contradicts $\pi p\, \| \, q$. \\[-3mm]

Thus, (I) and (II) are equivalent, which implies $A = A^\prime$ as desired. Now, since $A \in V[G \uhr (\lambda + 1)]$, we can apply Replacement in the $ZFC$-model $V[G \uhr (\lambda + 1)]$ and obtain that $\rg\, A \subseteq \delta$ for some ordinal $\delta$. Therefore, \[\forall x \in a\ \exists y \in N_\delta\ \left \langle N, \in, V \right \rangle \vDash \varphi (x, y). \] Since $N_\delta \in N$ (the canonical name $\dot{N}_\delta := \{ (\dot{x}, \m{1})\ | \ \dot{x} \in HS\, \cap \, \Name_{\alpha + 1} (\m{P}_\alpha)^V \}$ is symmetric), this finishes the proof.

\end{proof}

Similarly, one can show that the Axiom of Replacement holds true in $V[G]$ as long the formula $\varphi$ does not make use of the parameter $G$ for the generic filter:

\begin{prop} For any $\mathcal{L}_\in^{A}$-formula $\varphi (x, y, v_0, \, \ldots \,, v_{n-1})$ and $a, z_0, \, \ldots \,, z_{n-1} \in V[G]$ such that \[\left\langle V[G], \in, V, G\right\rangle\, \vDash \forall x \in a \ \exists y \ \varphi (x, y, z_0, \, \ldots \,, z_{n-1}),\] it follows that there exists $b \in  V[G]$ with the property that\[ \left\langle V[G], \in, V, G\right\rangle\, \vDash \forall x \in a \ \exists y \in  b\ \varphi (x, y, z_0, \, \ldots \,, z_{n-1}).\] 
\end{prop}

One can use basically the same proof, but with the hierarchy $\big( (V[G])_\alpha\ | \ \alpha \in \Ord \big)$ instead of $(N_\alpha\ | \ \alpha \in \Ord$), where $(V[G])_\alpha := \{\dot{x}^{G_{\alpha}}\ | \ \dot{x} \in \Name_{\alpha+1} (\m{P}_{\alpha})\}$. 

\begin{prop} \label{power set} The Axiom of Power Set holds in $N$.
\end{prop}

\begin{proof} Consider a set $Y \in N$. 
We first show: \[\exists\, \ol{\lambda} \in \Card\ \powerset^N (Y) \subseteq V[G \uhr (\ol{\lambda}+1)] \hspace*{2cm} (\ast).\]

Take a cardinal $\mu$ large enough such that $Y \in V[G \, \uhr \, (\mu+1)]$ and $|Y|^{V[G \, \uhr \, (\mu + 1)]} \leq \mu$, i.e. there exists an injection $\iota: Y \hookrightarrow \mu$ in $V[G \, \uhr \, (\mu + 1)]$. Take $\dot{Y} \in \Name (\m{P} \, \uhr \,(\mu + 1))^V$ with $Y = \dot{Y}^{G \, \uhr \, (\mu + 1)}$. Let $\lambda := F(\mu)^\plus$; then $|\m{P} \uhr (\mu + 1)| \leq \lambda$.

We claim that $\powerset^N (Y) \subseteq V[G \uhr (\lambda+1)]$. \\[-2mm]

Consider $Z \in \powerset^N (Y)$, $Z = \dot{Z}^G$ with $\dot{Z} \in HS$ such that $\pi \ol{\dot{Z}}^{D_\pi} = \ol{\dot{Z}}^{D_\pi}$ for all $\pi$ which are contained in the intersection\[ Fix_0 (\kappa_0, i_0) \, \cap\, \, \ldots \,\, \cap\, Fix_0(\kappa_{n-1}, i_{n-1})\, \cap \, Small_0 (\lambda_0, [0, \alpha_0))\, \cap\, \cdots\, \cap \, Small_0 (\lambda_{m-1}, [0, \alpha_{m-1}))\, \cap\, \]\[ \cap\, Fix_1 (\ol{\kappa}_0, \ol{\imath}_0)\, \cap\, \cdots\, \cap\, Fix_1 (\ol{\kappa}_{\ol{n}-1}, \ol{\imath}_{\ol{n}-1})\, \cap\, \cdots\, \cap \,Small_1 (\ol{\lambda}_0, [0, \ol{\alpha}_0))\, \cap \, Small_1 (\ol{\lambda}_{\ol{m}-1}, [0, \ol{\alpha}_{\ol{m}-1})).\]

Take a condition $r \in G$ such that $t(r)$ contains the vertices $(\kappa_0, i_0), \ldots, (\kappa_{n-1}, i_{n-1})$ and all $t(r)$-branches have height $\geq \mu$.

Then $G_0 \uhr (\mu + 1)\, \times\, (G_0 \uhr t(r)) \uhr [\mu, \infty)\, \times\, G_1 \uhr (\mu + 1)\, \times\, \big( G_1 \uhr \{ (\ol{\kappa}_0, \ol{\imath}_0), \ldots, (\ol{\kappa}_{\ol{m}-1}, \ol{\imath}_{\ol{m}-1})\}\big) \uhr [\mu, \infty)$ is a $V$-generic filter on $\m{P}_0 \uhr (\mu + 1)\, \times \, (\m{P}_0 \uhr t(r)) \uhr [\mu, \infty)\, \times\, \m{P}_1 \uhr (\mu + 1)\, \times \, \big( \m{P}_1 \uhr \{(\ol{\kappa}_0, \ol{\imath}_0), \ldots, (\ol{\kappa}_{\ol{m}-1}, \ol{\imath}_{\ol{m}-1}) \} \big) \uhr [\mu, \infty)$.  \\[-3mm]

We want to show that $Z$ is contained in the intermediate generic extension \[V\big [G_0 \uhr (\mu + 1)\, \times\, (G_0 \uhr t(r)) \uhr [\mu, \infty)\, \times\, G_1 \uhr (\mu + 1)\, \times\, \big( G_1 \uhr \{ (\ol{\kappa}_0, \ol{\imath}_0), \ldots, (\ol{\kappa}_{\ol{m}-1}, \ol{\imath}_{\ol{m}-1})\}\big) \uhr [\mu, \infty) \big]. \]

Let $Z^\prime$ be the set of all $\dot{y}^{G \, \uhr \, (\mu + 1)}$ with $\dot{y} \in \dom \dot{Y}$ such that there exists $p \in \m{P}$, $p_0 \leq r$, with $p \Vdash_{\m{P}}^{V, \check{V}, \dot{G}} \dot{y} \in \dot{Z}$ such that: \begin{itemize} \item $ p_0 \uhr (\mu + 1) \in G_0 \uhr (\mu + 1)$, \item $(p_0 \uhr t(r)) \uhr [\mu, \infty) \in (G_0 \uhr t(r)) \uhr [\mu, \infty)$, \item $p_1 \uhr (\mu + 1) \in G_1 \uhr (\mu + 1)$, \item $(p_1 \uhr \{ (\ol{\kappa}_0, \ol{\imath}_0),\, \ldots\, \}) \uhr [\mu, \infty) \in (G_1 \uhr \{ (\ol{\kappa}_0, \ol{\imath}_0), \, \ldots\, \}) \uhr [\mu, \infty)$. \end{itemize}

It suffices to show that $Z = Z^\prime$. The direction \tbl$\subseteq$\tbr\, follows from the forcing theorem. For \tbl $\supseteq$\tbr , we use an isomorphism argument similarly as before: Assume there was $\dot{y}^{G \, \uhr \, (\mu + 1)} \in Z^\prime \setminus Z$ with $\dot{y} \in \dom \dot{Y}$ and $p$ with $p \Vdash \dot{y} \in \dot{Z}$ as in the definition of $Z^\prime$. \\ Take $q \in G$ such that $q_0 \leq r$ and $q \Vdash \dot{y} \notin \dot{Z}$. We will construct an automorphism $\pi$ with $\pi p \, \| \, q$ such that $\pi$ restricted to $\m{P} \uhr (\mu + 1)$ is the identity, and additionally,
\[ \pi \in Fix_0 (\kappa_0, i_0)\, \cap\, \cdots\, \cap\, Small_0 (\lambda_0, [0, \alpha_0))\, \cap\, \cdots\, \cap\, Fix_1 (\ol{\kappa}_0, \ol{\imath}_0)\, \cap\, \cdots\]\[ \cdots\, \cap\, Small_1 (\ol{\lambda}_0, [0, \ol{\alpha}_0))\, \cap\, \cdots\;.\]But then, from $\pi p \Vdash \pi \ol{\dot{y}}^{D_\pi} \in \pi \ol{\dot{Z}}^{D_\pi}$ and $\pi \ol{\dot{Z}}^{D_\pi} = \ol{\dot{Z}}^{D_\pi}$, $\pi \ol{\dot{y}}^{D_\pi} = \ol{\dot{y}}^{D_\pi}$, it follows that $\pi p \Vdash  \ol{\dot{y}}^{D_\pi} \in \ol{\dot{Z}}^{D_\pi}$. Together with $\pi p\, \| \, q$ and $q \Vdash \dot{y} \notin \dot{Z}$, this gives the desired contradiction. \\[-2mm] 

We start with the construction of $\pi_0$. Let $\HT \pi := \max \{\eta (p), \eta(q)\}$. For $\alpha \leq \mu$, let $\pi_0 (\alpha)$ be the identity. 
In the case that $\alpha \in [\mu^\plus, \HT \pi]$, we take for $\pi_0 (\alpha)$ a bijection on $\{(\alpha, i)\ | \ i < F_{\lim} (\alpha)\}$ with finite support such that: \begin{itemize} \item for any $(\alpha, i) \in t(r)$, we have $\pi_0 (\alpha) (\alpha, i) = (\alpha, i)$, \item for any $(\alpha, i) \in t(p) \setminus t(r)$, we have $\pi_0 (\alpha) (\alpha, i) = (\alpha, j)$ for some $j < F_{\lim} (\alpha)$ with $(\alpha, j) \notin t(p)\, \cup \, t(q)$, \item for any $i < F_{\lim} (\alpha)$ with $i \in [\gamma (i), \gamma (i) + \omega)$ for $\gamma$ a limit ordinal, we have $\pi_0 (\alpha) (\alpha, i) = (\alpha, i^\prime)$ such that also $i^\prime \in [\gamma (i), \gamma (i) + \omega)$. \end{itemize}

Then $\pi_0$ is the identity on $\m{P}_0 \uhr (\mu + 1)$, and $\pi_0 \in Fix_0 (\kappa_0, i_0)\, \cap\, \cdots\, \cap\, Fix_0 (\kappa_{n-1}, i_{n-1})$, since $\pi_0 (\alpha) (\alpha, i) = (\alpha, i)$ for all $(\alpha, i) \in t(r)$. Moreover, $\pi_0 \in Small_0 [\lambda_0, [0, \alpha_0))\, \cap\, \cdots\, \cap\, Small_0 (\lambda_{m-1}, [0, \alpha_{m-1}))$, since we only use small permutations. 
By construction, it follows that $\pi_0 p_0 \, \| \, q_0$. \\[-3mm] 

The map $\pi_1$ 
can be constructed as in the proof of Proposition \ref{separation}. Then $\pi_1 p_1\, \| \, q_1$, $\pi_1$ restricted to $\m{P}_1 \uhr (\mu + 1)$ is the identity, $\pi_1 \in Fix_1 (\ol{\kappa}_0, \ol{\imath}_0)\, \cap \, \cdots\, \cap \, Fix_1 (\ol{\kappa}_{\ol{n}-1}, \ol{\imath}_{\ol{n}-1})$ since $p_1$ and $q_1$ agree on $\m{P}_1 \uhr \{(\ol{\kappa}_0, \ol{\imath}_0), \, \ldots \,\}$, and $\pi_1 \in Small_1 (\ol{\lambda}_0, [0, \ol{\alpha}_0))\, \cap\, \cdots\, \cap\, Small_1 (\lambda_{\ol{m}-1}, [0, \alpha_{\ol{m}-1}))$, since $\SUPP \pi_1 (\alpha^\plus) = \emptyset$ for all $\alpha^\plus \in \Succ^\prime$. \\[-3mm]

Hence, our automorphism $\pi$ has all the desired properties, which implies $Z = Z^\prime$; so 
\[Z \in V\big[\, G_0 \uhr (\mu + 1)\, \times\, (G_0 \uhr t(r)) \uhr [\mu, \infty)\, \times \, G_1 \uhr (\mu + 1)\, \times\, \big(G_1 \uhr \{(\ol{\kappa}_0, \ol{\imath}_0), \, \ldots \,\}\big) \uhr [\mu, \infty)\, \big]. \]

Recall that 
we have an injection $\iota: Y \hookrightarrow \mu$ in $V[G \uhr (\mu + 1)]$; so using the parameter $Z$, we can construct in $V[ G_0\, \uhr\, (\mu + 1)\, \times\, (G_0 \, \uhr \, t(r))\, \uhr \, [\mu, \infty)\, \times\, G_1 \, \uhr \, (\mu + 1)\, \times\, (G_1 \uhr \{ (\ol{\kappa}_0, \ol{\imath}_0), \, \ldots \,\})\, \uhr \, [\mu, \infty)]$ a function $\iota_Z: \mu \rightarrow 2$ with $\iota_Z(\alpha) = 1$ iff $\alpha \in \, \im (\iota)$ with $\iota^{-1} (\alpha) \in Z$, and $\iota_Z(\alpha) = 0$, else. \\[-3mm]

The forcing   \[\m{P}_0 \, \uhr \, (\mu + 1) \, \times\, (\m{P}_0 \uhr t(r)) \uhr [\mu, \infty)\, \times \, \m{P}_1\, \uhr \, (\mu + 1)\, \times \, (\m{P}_1 \uhr \{ (\ol{\kappa}_0, \ol{\imath}_0), \, \ldots \,\}) \uhr [\mu, \infty) \] can be factored as \[\Big( \, \m{P}_0 \, \uhr \, (\mu + 1) \, \times\, (\m{P}_0 \uhr t(r)) \uhr [\mu, \lambda + 1)\, \times \, \m{P}_1\, \uhr \, (\mu + 1)\, \times \, (\m{P}_1 \uhr \{ (\ol{\kappa}_0, \ol{\imath}_0), \, \ldots \,\}) \uhr [\mu, \lambda + 1)\, \Big)\ \times \] \[\times \, \Big(\, (\m{P}_0 \uhr t(r)) \uhr [\lambda, \infty)\, \times \, (\m{P}_1 \uhr \{ (\ol{\kappa}_0, \ol{\imath}_0), \, \ldots \,\}) \uhr [\lambda, \infty)\, \Big),\] \vspace*{1mm} where the \tbl lower part\tbr\,  \[\m{P}_0 \, \uhr \, (\mu + 1) \, \times\, (\m{P}_0 \uhr t(r)) \uhr [\mu, \lambda + 1)\, \times \, \m{P}_1\, \uhr \, (\mu + 1)\, \times \, (\m{P}_1 \uhr \{ (\ol{\kappa}_0, \ol{\imath}_0), \, \ldots \,\}) \uhr [\mu, \lambda + 1) \] has cardinality $\leq F_{\lim} (\mu) \, \cdot\, \lambda\, \cdot\, F(\mu)^\plus\, \cdot\, \lambda = F(\mu)^\plus = \lambda$, and the \tbl upper part\tbr\, \[(\m{P}_0 \uhr t(r)) \uhr [\lambda, \infty)\, \times \, (\m{P}_1 \uhr \{ (\ol{\kappa}_0, \ol{\imath}_0), \, \ldots \,\}) \uhr [\lambda, \infty)\] is $\leq \lambda$-closed. \\[-3mm]

Hence, \[\iota_Z \in V\big[G_0 \uhr (\mu + 1) \, \times\, (G_0 \uhr t(r)) \uhr [\mu, \lambda + 1) \, \times\, G_1 \uhr (\mu + 1) \, \times\, (G_1 \uhr \{ (\ol{\kappa}_0, \ol{\imath}_0), \, \ldots \,\}) \uhr [\mu, \lambda + 1) \big]; \]
so $\iota_Z \in V[G \uhr (\lambda + 1)]$, which implies that also $Z \in V[G \uhr (\lambda + 1)]$. \\[-3mm]

Since $Z \in \powerset^N (Y)$ was arbitrary, it follows that $\powerset^N (Y) \subseteq V[G \uhr (\lambda + 1)]$ as desired. This proves $(\ast)$. \\[-3mm]

Now, let $a := \powerset^{V[G \, \uhr \, (\lambda+1)]} (Y) \in V[G \uhr (\lambda+1)]$. Then $\powerset^N (Y) \subseteq a$. Take $\dot{a} \in \Name(\m{P} \uhr (\lambda+1))^V$ with $a = \dot{a}^G = \dot{a}^{G \, \uhr \, (\lambda + 1)}$. \\[-3mm]

Inside the structure $ \left\langle V[G], \in, V, G\right\rangle$, we define a function $F: a \rightarrow Ord$ as follows: \\[-3mm]

{\textit{For $z \in a$, let $F(z) = \alpha$ if $\alpha = \min \{\beta\ | \ z \in N_\beta\}$ if such an $\alpha$ exists. Let $F(z) = 0$, else. }} \\[-3mm]

Now, we will use the function $f(\dot{x}, \alpha)$ from Proposition \ref{Replacement}
with the property that 
with $\left\langle V[G], \in, V, G \right\rangle \vDash f(\dot{x}, \alpha) = x$ iff $\dot{x} \in \Name_{\alpha + 1} (\m{P}_{\alpha})^V$ with $x = \dot{x}^{G_{\alpha}}$. \\[-3mm]           

Let $\eta (z, \beta)$ denote the statement    
\[\exists\, \dot{x} \in HS\, \cap \, \Name_{\beta+1} (\m{P}_{\beta})\ \ z = f(\dot{x}, \beta).\]
Then \[F = \big\{ \, (\dot{z}^{G \, \uhr \, (\lambda+1)}, \alpha)\ | \ \dot{z} \in \dom \dot{a}\, \wedge\, \dot{z}^{G \, \uhr \, (\lambda + 1)} \in \dot{a}^{G \, \uhr \, (\lambda+1)}\, \wedge\, \alpha = \min \{\beta\ | \ \eta (\dot{z}^{G \, \uhr \, (\lambda+1)}, \beta)\}\, \big\}\, \cup \] \[\cup\, \big\{ \, (\dot{z}^{G \, \uhr \, (\lambda+1)}, 0)\ | \ \dot{z} \in \dom \dot{a} \, \wedge \, \dot{z}^{G \, \uhr \, (\lambda+1)} \in \dot{a}^{G \, \uhr \, (\lambda+1)}\, \wedge\, \neg \exists\, \beta\ \eta(\dot{z}^{G \, \uhr \, (\lambda+1)}, \beta) \, \big\}. \]

We claim that $F \in V[G \uhr (\lambda+1)]$. \\[-3mm]
 
Let \[\wt{F} := \big\{ \, (\dot{z}^{G \, \uhr\,  (\lambda+1)}, \alpha)\ | \ \dot{z} \in \dom \dot{a}, \dot{z}^{G \, \uhr \,(\lambda+1)} \in \dot{a}^{G \, \uhr \, (\lambda+1)}, \exists\, p: \ p \Vdash_{\m{P}}^{V, \check{V}, \dot{G}} \,\alpha = \min \{\beta\ | \ \eta (\dot{z}, \beta)\}, \]\[ p \uhr (\lambda + 1) \in G \uhr (\lambda + 1)\, \big\} \ \ \cup\] \[ \cup \ \big\{ \, (\dot{z}^{G \, \uhr \, (\lambda+1)}, 0)\ | \ \dot{z} \in \dom \dot{a}, \dot{z}^{G \, \uhr \, (\lambda+1)} \in \dot{a}^{G \, \uhr \, (\lambda+1)}, \exists\, p: \ p \Vdash_{\m{P}}^{V, \check{V}, \dot{G}}\,  \neg \exists\, \beta\ \eta (\dot{z}, \beta), \] \[ p \uhr (\lambda + 1) \in G \uhr (\lambda + 1)\, \big\}. \] It suffices to show that $F = \wt{F}$. The direction \tbl $\subseteq$\tbr\, follows from the forcing theorem. Concerning \tbl $\supseteq$\tbr\,, we proceed as in the proof of Proposition \ref{Replacement}:

Assume towards a contradiction, there was $(\dot{z}^{G \, \uhr \, (\lambda+1)}, \alpha) \in \wt{F} \setminus F$ with $\dot{z} \in \dom \dot{a}$, $\dot{z}^{G \, \uhr \, (\lambda+1)} \in \dot{a}^{G \, \uhr \, (\lambda+1)}$. W.l.o.g., let $\alpha> 0$. \\ Take $p \in \m{P}$ with \[ p \Vdash_{\m{P}}^{V, \check{V}, \dot{G}} \, \alpha = \min \{\beta\ | \ \eta (\dot{z}, \beta)\} \] and $p \uhr (\lambda + 1) \in G \uhr (\lambda + 1)$. Since $(\dot{z}^{G \, \uhr \, (\lambda+1)}, \alpha) \notin F$, there must be $q \in G$ with \[q \Vdash_{\m{P}}^{V, \check{V}, \dot{G}} \, \neg\, (\alpha = \min \{\beta\ | \ \eta(\dot{z}, \beta)\}).\] 
As in the proof of Proposition \ref{separation}, we construct an automorphism $\pi$ with $\pi p \, \| q$ such that $\pi$ restricted to $\m{P} \uhr (\lambda+1)$ is the identity. Then $\pi \ol{\dot{z}}^{D_\pi} = \ol{\dot{z}}^{D_\pi}$; so \[\pi p \Vdash_{\m{P}}^{V, \check{V}, \pi \dot{G}} \ \alpha = \min \{\beta\ | \ \eta(\dot{z}, \beta)\}. \]    

Now, for any condition $r \in \m{P}$ and $\beta$ an ordinal, we have $r \Vdash_{\m{P}}^{V, \check{V}, \dot{G}} \eta (\dot{z}, \beta)$ if and only if $r \Vdash_{\m{P}}^{V, \check{V}, \pi \dot{G}} \eta (\dot{z}, \beta)$, like in the proof of Proposition \ref{Replacement}. Hence, \[ \pi p \Vdash_{\m{P}}^{V, \check{V}, \dot{G}} \alpha = \min \{\beta\ | \ \eta(\dot{z}, \beta) \}, \] contradicting that $\pi p \, \| \, q$. \\[-3mm] 
 
The case $\alpha = 0$ is similar. Hence, $F  =\wt{F} \in V[G \uhr (\lambda+1)]$ as desired. \\[-3mm]

Now, by Replacement in $V[G \uhr (\lambda+1)]$, it follows that $rg\; F$ is bounded by some ordinal $\delta$. Then any $z \in \powerset^N (Y) \subseteq a$ is contained in some $N_\alpha$ for $\alpha < \delta$; hence, $\powerset^N (Y) \subseteq N_\delta$. By the Axiom of Separation, this implies $\powerset^N (Y) \in N$ as desired. \\[-3mm]

\end{proof}

Thus, we have shown that the symmetric extension $N$ is indeed a model of $ZF$. \\[-3mm]

We will now see that $N$ preserves all $V$-cardinals, which follows from the fact that any set of ordinals $X \in N$, $X \subseteq \alpha$ can be captured in a \tbl mild\tbr\,$V$-generic extension by a forcing as in Proposition \ref{prescard}: 

\begin{lem}[Approximation Lemma] \label{approx} Let $X \in N$, $X \subseteq \alpha$ with $X = \dot{X}^G$ such that $\pi \ol{\dot{X}}^{D_\pi} = \ol{\dot{X}}^{D_\pi}$ for all $\pi$ which are contained in the intersection \[Fix_0 (\kappa_0, i_0) \, \cap\, \cdots\, \cap\, Fix_0 (\kappa_{n-1}, i_{n-1})\, \cap\, Small_0 (\lambda_0, [0, \alpha_0))\, \cap\, \cdots\, \cap\, Small_0 (\lambda_{m-1}, [0, \alpha_{m-1}))\, \cap\, \]\[\cap\, Fix_1 (\ol{\kappa}_0, \ol{\imath}_0)\, \cap\, \cdots\, \cap\, Fix_1 (\ol{\kappa}_{\ol{n}-1}, \ol{\imath}_{\ol{n}-1})\, \cap\, Small_1 (\ol{\lambda}_0, [0, \ol{\alpha}_0))\, \cap\, \cdots\, \cap\, Small_1(\ol{\lambda}_{\ol{m}-1}, [0, \ol{\alpha}_{\ol{m}-1})).\] Let $r \in G_0$ such that  
$\{(\kappa_0, i_0), \, \ldots \,, (\kappa_{n-1}, i_{n-1})\} \subseteq t(r)$ contains all maximal points of $t(r)$. \\[-3mm]  

Then \[X \in V[G_0 \uhr t (r) \times G_1 \uhr \{(\ol{\kappa}_0, \ol{\imath}_0), \, \ldots \,, (\ol{\kappa}_{\ol{n}-1}, \ol{\imath}_{\ol{n}-1})\}] = \] \[ = V[G_0 \uhr \{ (\kappa_0, i_0), \, \ldots \,, (\kappa_{n-1}, i_{n-1})\} \times G_1 \uhr \{(\ol{\kappa}_0, \ol{\imath}_0), \, \ldots \,, (\ol{\kappa}_{\ol{n}-1}, \ol{\imath}_{\ol{n}-1})\}].\] \end{lem}

\begin{proof}  Define \[X^\prime := \{\beta < \alpha\ | \ \exists\, q = (q_0, q_1)\, :\ q_0 \leq_0 r\, , \, q \Vdash \beta \in \dot{X}\, , \, q_0 \uhr t(r) \in G_0 \uhr t(r),\] 
\[q_1 \uhr \{ (\ol{\kappa}_0, \ol{\imath}_0), \, \ldots \,\} \in G_1 \uhr \{ (\ol{\kappa}_0, \ol{\imath}_0), \, \ldots \,\}\, \}. \] 

Then $X = X^\prime$ follows by an isomorphism argument as before.

\end{proof}

From Lemma \ref{approx} and Proposition \ref{prescard} we obtain: 

\begin{cor} Cardinals are $N$-$V$-absolute. \end{cor}

A factoring argument shows that for $X \subseteq \kappa$ with $\kappa$ a cardinal, the according forcings in the statement of Lemma \ref{approx} can be cut off at level $\kappa^{\plus}$:

\begin{cor} \label{corapprox} Let $X \in N$, $X \subseteq \kappa$ with $\kappa$ a limit cardinal. Then there are $n, n^\prime < \omega$, $j_0, \, \ldots \,,\, j_{n-1} < F_{\lim} (\kappa^{\plus}) = F(\kappa)$, and $\ol{\kappa}_0, \, \ldots \,, \ol{\kappa}_{n^\prime-1} \in Succ^\prime$ with $\ol{\kappa}_0 < \kappa, \, \ldots \,, \ol{\kappa}_{n^\prime-1} < \kappa$; $\ol{\imath}_0 < F(\ol{\kappa}_0), \, \ldots \,, \ol{\imath}_{n^\prime-1} < F(\ol{\kappa}_{n^\prime-1})$ such that \[X \in V [G_0 \uhr \{(\kappa^{\plus}, j_0), \, \ldots \,, (\kappa^{\plus}, j_{n-1})\} \times G_1 \uhr \{ (\ol{\kappa}_0, \ol{\imath}_0), \, \ldots \,, (\ol{\kappa}_{n^\prime-1}, \ol{\imath}_{n^\prime-1})\} \times G_1 (\kappa^{\plus})].\]

For a successor cardinal $\kappa^\plus$ and $X \in N$, $X \subseteq \kappa^{\plus}$, there are $n, n^\prime < \omega$, $j_0, \, \ldots \,, j_{n-1} < F_{\lim} (\kappa^{\plus})$, and $\ol{\kappa}_0, \, \ldots \,, \ol{\kappa}_{n^\prime-1} \in Succ^\prime$ with $\ol{\kappa}_0 \leq \kappa^{\plus}, \, \ldots \,, \ol{\kappa}_{n^\prime-1} \leq \kappa^{\plus}$; $\ol{\imath}_0 < F(\ol{\kappa}_0), \, \ldots \,, \ol{\imath}_{n^\prime-1} < F(\ol{\kappa}_{n^\prime-1})$ such that \[X \in V [G_0 \uhr \{(\kappa^{\plus}, j_0), \, \ldots \,, (\kappa^{\plus}, j_{n-1})\} \times G_1 \uhr \{ (\ol{\kappa}_0, \ol{\imath}_0), \, \ldots \,, (\ol{\kappa}_{n^\prime-1}, \ol{\imath}_{n^\prime-1})\}].\]

\end{cor}

\begin{proof} First, we consider the case that $\kappa$ is a limit cardinal. From Lemma \ref{approx}, it follows that there are finitely many cardinals $\kappa_0, \, \ldots \,, \kappa_{n-1}$, and $i_0 < F_{\lim} (\kappa_0), \, \ldots \,, i_{n-1} < F_{\lim} (\kappa_{n-1})$; moreover, finitely many $\ol{\kappa}_0, \, \ldots \,, \ol{\kappa}_{\ol{n}-1} \in Succ^\prime$ and $\ol{\imath}_0 < F(\ol{\kappa}_0), \, \ldots \,, \ol{\imath}_{\ol{n}-1} < F(\ol{\kappa}_{\ol{n}-1})$ with \[X \in V[G_0 \uhr \{ (\kappa_0, i_0), \, \ldots \,, (\kappa_{n-1}, i_{n-1})\} \times G_1 \uhr \{(\ol{\kappa}_0, \ol{\imath}_0), \, \ldots \,, (\ol{\kappa}_{\ol{n}-1}, \ol{\imath}_{\ol{n}-1})\}].\] W.l.o.g. we can assume $\kappa_0, \, \ldots \,, \kappa_{n-1} \geq \kappa^{\plus}$. Take a condition $r \in G_0$ such that $\{(\kappa_0, i_0), \, \ldots \,, (\kappa_{n-1}, i_{n-1})\} \subseteq t(r)$ contains all maximal points of $t(r)$. Then \[G_0 \uhr \{ (\kappa_0, i_0), \, \ldots \,, (\kappa_{n-1}, i_{n-1})\} \times G_1 \uhr \{(\ol{\kappa}_0, \ol{\imath}_0), \, \ldots \,, (\ol{\kappa}_{\ol{n}-1}, \ol{\imath}_{\ol{n}-1})\}\] is a $V$-generic filter on the forcing \[\m{P}_0 \uhr t (r) \times \m{P}_1 \uhr \{ (\ol{\kappa}_0, \ol{\imath}_0), \, \ldots \,, (\ol{\kappa}_{\ol{n}-1}, \ol{\imath}_{\ol{n}-1}\}),\] which can be factored in a \tbl lower part\tbr\, \[\big((\m{P}_0 \uhr t (r)) \uhr (\kappa^{\plus} + 1)\big) \times \big((\m{P}_1 \uhr \{ (\ol{\kappa}_0, \ol{\imath}_0), \, \ldots \,, (\ol{\kappa}_{\ol{n}-1}, \ol{\imath}_{\ol{n}-1})\}) \uhr (\kappa^{\plus} + 1) \big),\] with cardinality $\leq \kappa^{\plus}$, and an \tbl upper part\tbr \[\big((\m{P}_0 \uhr t (r)) \uhr [\kappa^{\plus}, \infty)\big) \times \big((\m{P}_1 \uhr \{ (\ol{\kappa}_0, \ol{\imath}_0), \, \ldots \,, (\ol{\kappa}_{\ol{n}-1}, \ol{\imath}_{\ol{n}-1})\}) \uhr [\kappa^{\plus}, \infty)\big),\] which is $\leq \kappa^{\plus}$-closed. Thus, $X$ is contained in the generic extension by the lower part: 
Let $(\kappa^{\plus}, j_0), \, \ldots \,, (\kappa^{\plus}, j_{n-1})$ denote the $\leq_{t (r)}$-predecessors of $(\kappa_0, i_0), \, \ldots \,, $ $(\kappa_{n-1}, i_{n-1})$ respectively, on level $\kappa^\plus$. Moreover, assume w.l.o.g. that $0 \leq n^\prime \leq n^{\prime \prime} \leq \ol{n}$ with $\ol{\kappa}_0, \, \ldots \,, \ol{\kappa}_{n^{\prime}-1} < \kappa$; $\ol{\kappa}_{n^\prime}, \, \ldots \,, \ol{\kappa}_{n^{\prime \prime}-1} = \kappa^{\plus}$, and $\ol{\kappa}_{n^{\prime \prime}}, \, \ldots \,, \ol{\kappa}_{\ol{n}-1} > \kappa^{\plus}$. 
Then \[X \in V[G_0 \uhr \{(\kappa^{\plus}, j_0), \, \ldots \,, (\kappa^{\plus}, j_{n-1})\} \times G_1 \uhr \{(\ol{\kappa}_0, \ol{\imath}_0), \, \ldots \,, (\ol{\kappa}_{n^{\prime \prime}-1}, \ol{\imath}_{n^{\prime \prime} - 1})\}] \subseteq \] \[ \subseteq V[G_0 \uhr \{ (\kappa^{\plus}, j_0), \, \ldots \, (\kappa^\plus, j_{n-1})\} \times G_1 \uhr \{(\ol{\kappa}_0, \ol{\imath}_0), \, \ldots \,, (\ol{\kappa}_{n^\prime-1}, \ol{\imath}_{n^\prime-1})\} \times G_1 (\kappa^{\plus})]\] as desired. \\[-2mm]

The case $X \subseteq \kappa^{\plus}$ is similar.

\end{proof}

\section{\bfseries $ \mathbf{ \boldsymbol{\forall} \boldsymbol{\kappa \in } Card\ \boldsymbol{\theta^N} \boldsymbol{(\kappa)}  = F\boldsymbol{(\kappa)}}$}

Firstly, using the subgroups $Small_0 (\kappa, [0, \alpha))$ or $Small_1 (\kappa, [0, \alpha))$, it is not difficult to see that $\theta^N (\kappa) \geq F(\kappa)$ for all cardinals $\kappa$; i.e. for any $\alpha < F(\kappa)$, there exists in $N$ a surjection $s: \powerset (\kappa) \rightarrow \alpha$: 

\begin{prop} $\forall\, \kappa \in \Card\ \theta^N (\kappa) \geq F(\kappa)$. \end{prop}

\begin{proof} First, we consider the case that $\kappa$ is a limit cardinal. Fix some cardinal $\alpha < F_{\lim} (\kappa) = F(\kappa)$; we construct in $N$ a surjection $s: \powerset (\kappa) \rightarrow \alpha$. \\ As already mentioned in section \ref{symmetric names}, we define for any limit ordinal $i < \alpha$ a \tbl cloud\tbr\, around $(\dot{G_0})_{(\kappa, i)}$ as follows: \[(\wt{\dot{G_0}})_{(\kappa, i)}^{\,\alpha} := \{ \big(\pi (\dot{G_0})_{(\kappa, i)}, \m{1}\big)\ | \ \pi \in Small_0 (\kappa, [0, \alpha))\, \}  = \{ \big( (\dot{G_0})_{(\kappa, i + n)}, \m{1} \big)\ | \ n < \omega\}.\] 

Then \[ (\wt{G_0})_{(\kappa, i)}^{\,\alpha} := \big((\wt{\dot{G_0}})_{(\kappa, i)}^{\,\alpha}\big)^G = \big\{ \, (G_0)_{(\kappa, i + n)}\ | \ n < \omega\, \big\} \in N \] for any limit ordinal $i < \alpha$, since the name $(\wt{\dot{G_0}})_{(\kappa, i)}^{\,\alpha}$ is fixed by all 
$\pi \in Small_0 (\kappa, [0, \alpha))$.
 
Moreover, any two distinct clouds $(\wt{G_0})_{(\kappa, i)}^{\,\alpha}$ and $(\wt{G_0})_{(\kappa, j)}^{\, \alpha}$ for limit ordinals $i$ and $j$ are disjoint -- here, we have to use that splitting at limits is not allowed in our tree forcing; so for $j, j^\prime < F_{\lim} (\kappa)$ with $j \neq j^\prime$ it follows that indeed, $(G_0)_{(\kappa, j)} \neq (G_0)_{(\kappa, j^\prime)}$. \\[-3mm] 

For $\m{P}$-names $\dot{x}$, $\dot{y}$, we denote by $\OR_{\m{P}} (\dot{x}, \dot{y})$ the canonical name for the ordered pair $(\dot{x}^G, \dot{y}^G)$. The sequence $\big( (\wt{G_0})_{(\kappa, i)}^{\,\alpha} \ | \ i < \alpha \mbox{ a limit ordinal} \big)$ is contained in $N$ as well, since its name \[ \big\{\,  \big( \, OR_{\m{P}}(i, (\wt{\dot{G_0}})_{(\kappa, i)}^{\,\alpha}), \m{1}\, \big)\ | \ i < \alpha \mbox{ a limit ordinal} \, \big\}\] is fixed by all $\pi \in Small_0 (\kappa, [0, \alpha))$. \\This gives in $N$ a well-defined surjection $\ol{s}: \powerset (\kappa) \rightarrow \{i < \alpha\ | \ i \mbox{ is a limit ordinal}\}$, by setting $\ol{s} (X) := i$ whenever $X \in (\wt{G_0})_{(\kappa, i)}^{\,\alpha}$ for some $i < \alpha$, and $\ol{s}(X) := 0$, else. \\ Also without the Axiom of Choice, $\ol{s}$ can be turned into a surjection $s: \powerset (\kappa) \rightarrow \alpha$. \\[-2mm] 

Concerning successor cardinals, it suffices to show that $\theta^N (\kappa^{\plus}) \geq F(\kappa^{\plus})$ for all $\kappa^{\plus} \in Succ^\prime$. Let $\alpha < F(\kappa^{\plus})$. We proceed similarly as before, setting \[(\wt{\dot{G_1}})_{(\kappa^{\plus}, i)}^{\,\alpha} := \big\{ \, \big(\, \pi (\dot{G_1})_{(\kappa^{\plus}, i)}, \m{1}\, \big)\ | \ \pi \in Small_1 (\kappa^{\plus}, [0, \alpha))\, \big\},\] and $\pi \,(G_1)_{(\kappa^\plus, i)}^{\alpha} := \big( \pi (\dot{G}_1)_{(\kappa^\plus, i)}^{\,\alpha} \big)^G$, with

\[(\wt{G_1})_{(\kappa^{\plus}, i)}^{\, \alpha} := \big( (\wt{\dot{G_1}})_{(\kappa^{\plus}, i)}^{\,\alpha} \big)^G = \big \{\,\pi\,(G_1)_{(\kappa^{\plus}, i)} \ | \ \pi \in Small_1 (\kappa^{\plus}, [0, \alpha)) \, \big\}.\]  

As before, it follows that the sequence $\big( (\wt{G_1})_{(\kappa^{\plus}, i)}^{\,\alpha}\ | \ i < \alpha \big)$ is contained in $N$, so it suffices to check that two distinct \tbl clouds\tbr\, $(\wt{G_1})_{(\kappa^{\plus}, i)}^{\,\alpha}$ and $(\wt{G_1})_{(\kappa^{\plus}, j)}^{\,\alpha}$ are indeed disjoint. Assume towards a contradiction, there were $\pi$, $\sigma \in Small_1 (\kappa^{\plus}, [0, \alpha))$ with $\pi \, (G_1)_{(\kappa^{\plus}, i)} = \sigma \, (G_1)_{(\kappa^{\plus}, j)}$. By genericity, take $\zeta \in [\kappa, \kappa^{\plus}) \setminus (\dom_x \pi (\kappa^{\plus}) \, \cup \, \dom_x \sigma (\kappa^{\plus}))$ with $ (G_1)_{(\kappa^{\plus}, i)} (\zeta) \neq (G_1)_{(\kappa^{\plus}, j)} (\zeta)$. Since $i, j < \alpha$ and $\pi, \sigma \in $ $Small_1 (\kappa^\plus, $ $[0, \alpha))$, it follows that $\pi \, (G_1)_{(\kappa^\plus, i)} (\zeta) = (G_1)_{(\kappa^\plus, i)} (\zeta)$ and $\sigma \, (G_1)_{(\kappa^\plus, j)} (\zeta) =  (G_1)_{(\kappa^\plus, j)} (\zeta)$. Contradiction. \\ Hence, the sequence $\big( (\wt{G_1})_{(\kappa^{\plus}, i)}^{\,\alpha}\ | \ i < \alpha \big)$ gives in $N$ a surjective function $s: \powerset(\kappa^\plus) \rightarrow \alpha$ as desired.
 
\end{proof}

It remains to show that $\theta^N (\kappa) \leq F(\kappa)$ for all cardinals $\kappa$. 

First, we consider the case that $\kappa$ is a limit cardinal. Assume towards a contradiction that there was a surjection $S: \powerset (\kappa) \rightarrow F(\kappa)$ in $N$. For the rest of this section, fix such a surjection $S$. \\
Let $\dot{S} \in HS$ with $S = \dot{S}^G$ such that $\pi \ol{\dot{S}}^{D_\pi} = \ol{\dot{S}}^{D_\pi}$ for all $\pi$ that are contained in the intersection \[Fix_0 (\kappa_0, i_0)\, \cap\, \cdots\, \cap\, Fix_0 (\kappa_{n-1}, i_{n-1})\, \cap\, Small_0 (\lambda_0, [0, \alpha_0))\, \cap\, \cdots\, \cap\, Small_0 (\lambda_{m-1}, [0, \alpha_{m-1}))\, \cap \]\[ \cap \, Fix_1 (\ol{\kappa}_0, \ol{\imath}_0)\, \cap\, \cdots\, \cap \, Fix_1 (\ol{\kappa}_{\ol{n}-1}, \ol{\imath}_{\ol{n}-1})\, \cap\, Small_1 (\ol{\lambda}_0, [0, \ol{\alpha}_0))\, \cap\, \cdots \, \cap Small_1 (\ol{\lambda}_{\ol{m}-1}, [0, \ol{\alpha}_{\ol{m}-1})),\] which will be abbreviated by $(A_{\dot{S}})$. \\[-3mm]

We know from Corollary \ref{corapprox} that any $X \in N$, $X \subseteq \kappa$ is contained in a generic extension of the form \[V [G_0 \uhr \{(\kappa^{\plus}, j_0), \, \ldots \,, (\kappa^{\plus}, j_{k-1})\} \times G_1 \uhr \{ (\mu_0, \ol{\jmath}_0), \, \ldots \,, (\mu_{\ol{k}-1}, \ol{\jmath}_{\ol{k}-1})\} \times G_1 (\kappa^{\plus})],\] where $k, \ol{k} < \omega$, $j_0, \, \ldots \,, j_{k-1} < F_{\lim} (\kappa^{\plus}) = F(\kappa)$, and $\mu_0, \, \ldots \,, \mu_{\ol{k}-1} < \kappa$, $\ol{\jmath}_0 < F(\mu_0), \, \ldots \,, $ $\ol{\jmath}_{\ol{k}-1} < F(\mu_{\ol{k}-1})$. \\[-2mm]

For a limit ordinal $\beta < F(\kappa)$ \textit{large enough for $(A_{\dot{S}})$} (we give a precise definition of this term later), we want to consider a map $S^\beta \subseteq S$, which will be the restriction of $S$ to all $X$ that are contained in a generic extension \[V [G_0 \uhr \{(\kappa^{\plus}, j_0), \, \ldots, \, (\kappa^\plus, j_{k-1})\} \times G_1 \uhr \{ (\mu_0, \ol{\jmath}_0), \, \ldots, \, (\mu_{\ol{k}-1}, \ol{\jmath}_{\ol{k}-1})\} \times G_1 (\kappa^{\plus})],\] where $j_0, \, \ldots \,, j_{k-1} < \beta$ and $\ol{\jmath}_0, \, \ldots \,, \ol{\jmath}_{\ol{k}-1} < \beta$. \\[-3mm]

Let $M$ denote the collection of all tuples $(s, (\mu_0, \ol{\jmath}_0), \, \ldots \,, (\mu_{\ol{k}-1}, \ol{\jmath}_{\ol{k}-1}))$ such that $\ol{k} < \omega$, $\mu_0, \, \ldots \,, \mu_{\ol{k}-1} \in \kappa\, \cap\, Succ^\prime$, $\ol{\jmath}_0 < F(\mu_0), \, \ldots \,, \ol{\jmath}_{\ol{k}-1} < F(\mu_{\ol{k}-1})$, and $s$ is a condition in $\m{P}_0$ with finitely many maximal points $(\kappa^{\plus}, j_0), \, \ldots, \, (\kappa^{\plus}, j_{k-1})$ with $j_0, \, \ldots, \, j_{k-1} < F_{\lim} (\kappa^\plus) = F(\kappa)$. \\
For $\beta < F(\kappa)$, we denote by $M_\beta$ the collection of all tuples $(s, (\mu_0, \ol{\jmath}_0), \, \ldots \,, (\mu_{\ol{k}-1}, \ol{\jmath}_{\ol{k}-1})) \in M$ such that additionally, $\ol{\jmath}_0 < \beta, \, \ldots, \, \ol{\jmath}_{\ol{k}-1} < \beta$, and $s$ has maximal points $(\kappa^{\plus}, j_0), \, \ldots, $ $\, (\kappa^{\plus}, j_{k-1})$ with $j_0, \, \ldots, \, j_{k-1} < \beta$.

\begin{prop} \label{theta 1} 
There is a limit ordinal $\beta < F(\kappa)$ such that the restriction \[S^\beta := S \uhr \big\{\, X \subseteq \kappa\ | \ \exists\, (s, (\mu_0, \ol{\jmath}_0), \, \ldots \,, (\mu_{\ol{k}-1}, \ol{\jmath}_{\ol{k}-1})) \in M_\beta: \ s \in G_0 \uhr (\kappa^\plus + 1)\, , \]\[ X \in V[G_0 \uhr t(s)\, \times\, G_1 \uhr \{(\mu_0, \ol{\jmath}_0), \, \ldots \,, (\mu_{\ol{k}-1}, \ol{\jmath}_{\ol{k}-1})\}\, \times\, G_1 (\kappa^{\plus})] \, \big\}\] is surjective onto $F(\kappa)$, as well. \end{prop} 

Later on, we will lead this into a contradiction by showing that any such $S^\beta$ must be contained in an intermediate generic extension which preserves cardinals $\geq F(\kappa)$, but also contains an injection $\iota: \dom S^\beta \hookrightarrow \beta$. \\[-2mm] 

We now define what we mean by  \textit{large enough for $(A_{\dot{S}})$}:
Fix a condition $r \in G_0$ such that $\{(\kappa_0, i_0), \, \ldots \,, (\kappa_{n-1}, i_{n-1})\} \subseteq t(r)$ contains all maximal points of $t(r)$, and an extension $\ol{r} \leq_0 r$, $\ol{r} \in G_0$ such that all $t(\ol{r})$-branches have height $\geq \kappa^{\plus}$. For $l < n$ with $\kappa_l \geq \kappa^\plus$, let $(\kappa^{\plus}, i_l^\prime)$ be the $t (\ol{r})$-predecessor of $(\kappa_l, i_l)$ on level $\kappa^{\plus}$; in the case that $\kappa_l < \kappa^{\plus}$, let $(\kappa^{\plus}, i_l^\prime)$ denote some $t (\ol{r})$-successor of $(\kappa_l, i_l)$ on level $\kappa^{\plus}$. \\ We say that a limit ordinal $\wt{\beta} < F_{\lim} (\kappa^{\plus}) = F(\kappa)$ is  {\textit{large enough for $(A_{\dot{S}})$}} if the following hold: 
\begin{itemize}\item $\wt{\beta} > i_0^\prime, \, \ldots \,, i_{n-1}^\prime$, \item $\wt{\beta} > \alpha_l$ for all $l < m$ with $\lambda_l \leq \kappa^{\plus}$, \item $\wt{\beta} > \ol{\imath}_l$ for all $l < \ol{n}$ with $\ol{\kappa}_l < \kappa$, \item $\wt{\beta} > \ol{\alpha}_l$ for all $l < \ol{m}$ with $\ol{\lambda}_l < \kappa$.
\end{itemize}

We will refer to these conditions $r$, $\ol{r}$ later on. \\[-3mm]

We want to show that whenever a limit ordinal $\wt{\beta} < F(\kappa)$ is \textit{large enough for} $(A_{\dot{S}})$ and $\beta := \wt{\beta} + \kappa^{\plus}$ (addition of ordinals), then $S^\beta$ must be surjective onto $F(\kappa)$, as well. \\[-2mm]

For any tuple $(s, (\mu_0, \ol{\jmath}_0), \, \ldots \,, (\mu_{\ol{k}-1}, \ol{\jmath}_{\ol{k}-1})) \in M$ and $\dot{x} \in \Name (\m{P}_0 \uhr t(s)\, \times\, \m{P}_1 \uhr \{(\mu_0, \ol{\jmath}_0), \, \ldots, \, (\mu_{\ol{k}-1}, \ol{\jmath}_{\ol{k}-1})\}\, \times\, \m{P}_1 (\kappa^{\plus}))$, we define a canonical extension $\wt{\dot{x}} \in \Name (\m{P})$ as follows: \\[-3mm]

Recursively, set \[ \wt{\dot{x}} := \big\{ \, (\wt{\dot{y}}, \ol{v})\ \big| \ \exists\, \big(\dot{y}, (v_0 \uhr t(s), v_1 \uhr \{(\mu_0, \ol{\jmath}_0), \, \ldots \,\}, v_1 (\kappa^{\plus}))\big) \in \dot{x}\ : \ \ol{v}_0 = v_0 \uhr t(s), \]\[ \supp \ol{v}_1 \subseteq \kappa^{\plus} + 1, \ol{v}_1 \uhr \{ (\mu_0, \ol{\jmath}_0), \, \ldots \,\} = v_1 \uhr \{ (\mu_0, \ol{\jmath}_0), \, \ldots \,\}, \ol{v}_1 (\kappa^{\plus}) = v_1 (\kappa^{\plus}) \, \big\}.\]
If $s \in G_0 \uhr (\kappa^\plus + 1)$, it follows that
\[\wt{\dot{x}}^G =\dot{x}^{G_0 \, \uhr \, t(s) \times G_1 \, \uhr \, \{(\mu_0, \ol{\jmath}_0), \, \ldots \,\} \times G_1 (\kappa^{\plus})}.\]
Sometimes, this name $\wt{\dot{x}}$ will be extended further to a name ${\ol{\wt{\dot{x}}}}^{D_\pi} \in \ol{\Name (\m{P})}^{D_\pi}$. In order to simplify notation, this extension will be denoted by $\wt{\dot{x}}^{D_\pi}$.

We now give a proof of Proposition \ref{theta 1}.

\begin{proof} Assume towards a contradiction that a limit ordinal $\wt{\beta} < F(\kappa)$ is large enough and $\beta := \wt{\beta} + \kappa^{\plus}$ (addition of ordinals), but $S^\beta$ is not surjective. Let $\alpha < F(\kappa)$ with $\alpha \notin \rg S^\beta$. \\
Fix some cardinal $\lambda$ with $\lambda > \max\{ \kappa^{\plus}, \kappa_0, \, \ldots \,, \kappa_{n-1}, \lambda_0, \, \ldots \,, \lambda_{m-1}, \ol{\kappa}_0, \, \ldots \,, \ol{\kappa}_{\ol{n}-1},$ $\ol{\lambda}_0, \, \ldots \,, \ol{\lambda}_{\ol{m}-1}\}$ such that $\dot{S} \in \Name (\m{P} \uhr (\lambda + 1))^V$. 
Then $S \in V[G \uhr (\lambda + 1)]$, and 
we can define a canonical $\m{P}\, \uhr (\lambda + 1)$-name for $S^\beta$ as follows: \[\dot{S}^\beta := \{ \big(\, OR_{\m{P} \, \uhr \, (\lambda + 1)} (\wt{\dot{X}}, \alpha), \ol{p}\, \big)\ | \ \exists\, (s, (\mu_0, \ol{\jmath}_0), \, \ldots \,, (\mu_{\ol{k}-1}, \ol{\jmath}_{\ol{k}-1})) \in M_\beta: \ \]\[ \dot{X} \in \Name (\m{P}_0 \uhr t(s)\, \times\, \m{P}_1 \uhr \{ (\mu_0, \ol{\jmath}_0), \, \ldots \,, (\mu_{\ol{k}-1}, \ol{\jmath}_{\ol{k}-1})\}\, \times\, \m{P}_1 (\kappa^{\plus})), \, \ol{p} = (\ol{p}_0, \ol{p}_1) \in \m{P} \uhr (\lambda + 1), \]\[\ol{p}_0 \leq s, \, \ol{p} \Vdash_{\m{P} \, \uhr \, (\lambda + 1)} OR_{\m{P} \, \uhr \, (\lambda + 1)} (\wt{\dot{X}}, \alpha) \in \dot{S} \}.  \]

It is not difficult to check that indeed, $(\dot{S}^\beta)^{G \, \uhr \, (\lambda + 1)} = S^\beta$. \\[-2mm]

Since $S: \powerset^N (\kappa) \rightarrow F(\kappa)$ is surjective, there must be $X \in \powerset^N (\kappa)$ with $(X, \alpha) \in S$. By Corollary \ref{corapprox}, take $(s, (\mu_0, \ol{\jmath}_0), \, \ldots \,, (\mu_{\ol{k}-1}, \ol{\jmath}_{\ol{k}-1})) \in M$ such that $s \in G_0 \uhr (\kappa^\plus + 1)$ has maximal points $(\kappa^{\plus}, j_0) ,\, \ldots \,, (\kappa^{\plus}, j_{k-1})$, and \[X = \dot{X}^{G_0 \, \uhr \, t(s) \, \times\, G_1 \, \uhr \, \{ (\mu_0, \ol{\jmath}_0), \, \ldots \,, (\mu_{\ol{k}-1}, \ol{\jmath}_{\ol{k}-1})\, \times\, G_1 (\kappa^{\plus})}\] for some $\dot{X} \in \Name (\m{P}_0 \uhr t(s)\, \times\, \m{P}_1 \uhr \{ (\mu_0, \ol{\jmath}_0), \, \ldots \,, (\mu_{\ol{k}-1}, \ol{\jmath}_{\ol{k}-1})\}\, \times\, \m{P}_1 (\kappa^{\plus}))$. W.l.o.g. we can assume that the sequences $(j_l\ | \ l < k)$ and $(\ol{\jmath}_l\ | \ l < \ol{k})$ are both increasing. \\[-3mm] 

Now, $(X, \alpha) = (\wt{\dot{X}}^{G \, \uhr \, (\lambda + 1)}, \alpha) \in \dot{S}^{G \, \uhr \, (\lambda + 1)}$, but $\alpha \notin \rg (\dot{S}^\beta)^{G \, \uhr \, (\lambda + 1)}$, so we can take $p \in G \uhr (\lambda + 1)$ such that \[p \Vdash_{\m{P} \, \uhr \, (\lambda + 1)} OR_{\m{P} \, \uhr \, (\lambda + 1)} (\wt{\dot{X}}, \alpha) \in \dot{S}\ \ \ , \ \ \ p \Vdash_{\m{P} \, \uhr \, (\lambda + 1)} \alpha \notin \rg \, \dot{S}^\beta.\]

W.l.o.g., let $p_0 \leq \ol{r}$, $p_0 \leq s$ and $ht\, p \geq \kappa^{\plus}$. \\[-3mm]

Now, take $h \leq k$ such that $j_0, \, \ldots \,, j_{h-1} < \beta$, $j_h, \, \ldots \,, j_{k-1} \geq \beta$, and $\ol{h} \leq \ol{k}$ with $\ol{\jmath}_0, \, \ldots \,, \ol{\jmath}_{\ol{h}-1} < \beta$, $\ol{\jmath}_{\ol{h}}, \, \ldots \,, \ol{\jmath}_{\ol{k}-1} \geq \beta$. Then $(X, \alpha) \notin S^\beta$ implies that $h < k$ or $\ol{h} < \ol{k}$. \\[-3mm]

Pick pairwise distinct ordinals $\pi_0 j_{h}, \, \ldots \,, \pi_0 j_{k-1}$ in $(\wt{\beta}, \beta) \setminus \{j_0, \, \ldots \,, j_{h-1} \}$ such that $\{(\kappa^\plus, \pi_0 j_h), \, \ldots \,, (\kappa^\plus, \pi_0 j_{k-1})\} \, \cap\, t(p_0) = \emptyset$. \\ We want to construct a $\m{P}_0$-automorphism $\pi_0$ which is the identity below level $\kappa^\plus$, and swaps for any $l \in [h, k)$ the vertex $(\kappa^\plus, j_l)$ with the vertex $(\kappa^\plus, \pi_0 j_l)$;
i.e. for any $q_0 \in \m{P}_0$ and $l \in [h, k)$ with $(\kappa^\plus, j_l), (\kappa^\plus, \pi_0 j_l) \in t(q_0)$, we want that
$(\pi_0 q_0) \uhr \{(\kappa^{\plus}, \pi_0 j_l)\} = q_0 \uhr \{ (\kappa^{\plus}, j_l)\}$, and $(\pi_0 q_0) \uhr \{ (\kappa^\plus, j_l)\} = q_0 \uhr \{( \kappa^\plus, \pi_0 j_l)\}$. Since $\{(\kappa^\plus, \pi_0 j_h), \, \ldots \,, (\kappa^\plus, \pi_0 j_{k-1})\} \, \cap\, t(p_0) = \emptyset$, we can assure that at the same time, $\pi_0 p_0 \, \| \, p_0$. \\[-3mm]

For $\m{P}_1$ we proceed similarly, but in order to achieve $\pi_1 p_1 \, \| \, p_1$, we first have to extend $p$ to  condition $\ol{p} = (p_0, \ol{p}_1) \leq (p_0, p_1)$ with $\ol{p} \in G \uhr (\lambda + 1)$ such that the following holds: \\[-3mm]

\textit{For any $\mu^\plus \in \Succ^\prime$ with $\{l \in [\ol{h}, \ol{k})\ | \ \mu_l = \mu^\plus\} = \{l_0, \ldots, l_{z-1}\}$ for some $1 \leq z < \omega$ (i.e.,
 $\mu_{l_0} = \cdots = \mu_{l_{z-1}} = \mu^\plus$, so $l_0, \, \ldots \,, l_{z-1} \in [\ol{h}, \ol{k})$ implies that $\ol{\jmath}_{l_0}, \ldots, \ol{\jmath}_{l_{z-1}} \geq \beta)$,
it follows that $\mu^\plus \in \supp \ol{p}_1$ with $\ol{\jmath}_{l_0}, \, \ldots \,, \ol{\jmath}_{l_{z-1}}\ \in \dom_y \ol{p}_1 (\mu^\plus)$, and there are $\pi_1 \ol{\jmath}_{l_0}, \, \ldots \,, \, $ $\pi_1 \ol{\jmath}_{l_{z-1}} \in (\wt{\beta}, \beta) \setminus \{ \ol{\jmath}_0, \ldots, \ol{\jmath}_{\ol{h}-1}\}$ with $\ol{p}_1 \uhr \{(\mu^\plus, \ol{\jmath}_{l_0})\} = \ol{p}_1 \uhr \{(\mu^\plus, \pi_1 \ol{\jmath}_{l_0})\}, \, \ldots \,, $ $\ol{p}_1 \uhr \{(\mu^\plus, \ol{\jmath}_{l_{z-1}})\} = \ol{p}_1 \uhr \{(\mu^\plus, \pi_1 \ol{\jmath}_{l_{z-1}})\}$. } \\[-3mm]

Since $\beta = \wt{\beta} + \kappa^\plus$, and $\dom_y p_1 (\mu^\plus)$ has cardinality $\leq \mu < \kappa$, this is possible by a density argument. \\[-3mm]

Now, it is possible to construct a $\m{P}_1$-automorphism that exchanges for every $l \in [\ol{h}, \ol{k})$ the $(\m{P}_1)_{(\mu_l, \ol{\jmath}_l)}$-coordinate with the according $(\m{P}_1)_{(\mu_l, \pi_1 \ol{\jmath}_l)}$-coordinate; so for any $q_1 \in D_\pi$, we will have $(\pi_1 q_1)_{(\mu_l, \pi_1 \ol{\jmath}_l)} = (q_1)_{(\mu_l, \ol{\jmath}_l)}$, $(\pi_1 q_1)_{(\mu_l, \ol{\jmath}_l)} = (q_1)_{(\mu_l, \pi_1 \ol{\jmath}_l)}$. By our preparations about $\ol{p}$, we can also assure $\pi_1 \ol{p}_1 \, \| \, \ol{p}_1$. \\[-2mm]

Moreover, we will have $\pi \ol{\dot{S}}^{D_\pi} = \ol{\dot{S}}^{D_\pi}$: Recall that $\wt{\beta}$ was \textit{large enough for} $(A_{\dot{S}})$, and for both $\pi_0$ and $\pi_1$ we do not disturb indices below $\wt{\beta}$; so $\pi \in Fix_0 (\kappa_0, i_0) \, \cap\, \cdots\, \cap \,  Small_0 (\lambda_0, [0, \alpha_0))\, \cap\, \cdots\, \cap\, Fix_1 (\ol{\kappa}_0, \ol{\imath}_0)\, \cap\, \cdots\, \cap\, Small_1 (\ol{\lambda}_0, [0, \ol{\alpha}_0))\, \cap\, \cdots$\,. \\ For a condition $\ol{q} \leq \ol{p}, \pi \ol{p}$ and $H$ a $V$-generic filter on $\m{P}$ with $\ol{q} \in H$, it follows that $\alpha \notin \rg (\dot{S}^\beta)^H$, but at the same time \[\big(\, \big(\pi {\wt{\dot{X}}}^{\,D_\pi}\big)^H, \alpha\, \big) \in \big(\pi {\ol{\dot{S}}}^{D_\pi}\big)^H = \dot{S}^H.\] We will see that this is a contradiction, since $\big(\pi {\wt{\dot{X}}}^{\, D_\pi}\big)^H$ will be equal to some $\big({{\wt{\ddot{X}}}}^{\,D_\pi}\big)^H$, where $\ddot{X}$ is a name for the forcing \[\m{P}_0 \uhr t(\pi_0 s) \, \times\, \m{P}_1 \uhr \{ (\mu_0, j_0), \, \ldots \,, (\mu_{\ol{h} - 1}, \ol{\jmath}_{\ol{h}-1}), (\mu_{\ol{h}}, \pi_1 \ol{\jmath}_{\ol{h}}), \, \ldots \,, (\mu_{\ol{k}-1}, \pi_1 \ol{\jmath}_{\ol{k}-1})\}\, \times\, \m{P}_1 (\kappa^{\plus}), \] where $\pi_0 s \geq \pi_0 \ol{p} \geq q$ has maximal points $(\kappa^{\plus}, j_0), \dots, (\kappa^{\plus}, j_{h-1}), (\kappa^{\plus}, \pi_0 j_h), \, \ldots \,,$ $(\kappa^{\plus}, \pi_0 j_k)$; thus, $(\pi_0 s, (\mu_0, \ol{\jmath}_0), \, \ldots \,, (\mu_{\ol{h}-1}, \ol{\jmath}_{\ol{h}-1}), (\mu_{\ol{h}}, \pi_1 \ol{\jmath}_{\ol{h}}), \, \ldots \,, (\mu_{\ol{k}}, \pi_1 \ol{\jmath}_{\ol{k}})) \in M_\beta$. This will imply  
$\big(\, \big({\wt{\ddot{X}}}^{\,D_\pi} \big)^H, \alpha\, \big) \in (\dot{S}^\beta)^H$, contradicting that $\alpha \notin \rg (\dot{S}^\beta)^H$. \\[-2mm]

We start with defining $\pi_0$. Let $ht\, \pi_0 := \eta (\ol{p})$. For any $\alpha < \kappa^{\plus}$, $\pi_0 (\alpha)$ will be the identity on $\{ (\alpha, i)\ | \ i < F_{\lim} (\alpha)\}$. Regarding level $\kappa^{\plus}$, let $\supp \pi_0 (\kappa^\plus) := \{ (\kappa^\plus, j_h), \, \ldots \,, (\kappa^\plus, j_{k-1}), $ $(\kappa^\plus, \pi_0 j_h), \, \ldots \,, (\kappa^\plus, \pi_0 j_{k-1}) \}$ with $\pi_0 (\kappa^{\plus}) (\alpha, j_l) = (\alpha, \pi_0 j_l)$, $\pi_0 (\kappa^{\plus}) (\alpha, \pi_0 j_l) = (\alpha, j_l)$ for all $l \in [h, k)$. For $\kappa^{\plus \plus} \leq \alpha \leq \HT \pi_0$, the map $\pi_0 (\alpha)$ is constructed as follows: Let $\{(\alpha, \delta_0), \, \ldots \,, (\alpha, \delta_{n (\alpha)-1})\}$ denote the collection of all $(\alpha, \delta) \in t (p_0)$ which have a $t(p_0)$-predecessor in $\{(\kappa^{\plus}, j_h), \, \ldots \,, (\kappa^{\plus}, j_{k-1})\}$. Pick $\wt{\delta}_0, \, \ldots \,, \wt{\delta}_{n (\alpha)-1}  < F_{\lim} (\alpha)$ pairwise distinct with $\{ (\alpha, \wt{\delta}_0), \, \ldots \,, (\alpha, \wt{\delta}_{n (\alpha)-1})\} \, \cap\, t(p_0) = \emptyset$ such that for all $i < n (\alpha)$, there is a limit ordinal $\gamma$ with $\delta_i, \wt{\delta}_i \in [\gamma, \gamma + \omega)$. Let $\supp \pi_0 (\alpha) := \{(\alpha, \delta_0), \, \ldots \,, (\alpha, \delta_{n (\alpha)-1}), (\alpha, \wt{\delta}_0), \, \ldots \,, (\alpha, \wt{\delta}_{n (\alpha)-1})\}$ with $\pi_0 (\alpha) (\alpha, \delta_l) = (\alpha, \wt{\delta}_l)$, $\pi_0 (\alpha) (\alpha, \wt{\delta}_l) = (\alpha, \delta_l)$ for all $l < n (\alpha)$. \\
This defines $\pi_0$. \\[-3mm]   

First, we have to check whether $\pi_0 \in Fix_0 (\kappa_0, i_0)\, \cap\, \cdots\, \cap\, Fix_0 (\kappa_{n-1}, i_{n-1})$. Consider $l < n$. Then $\pi_0 \in Fix_0 (\kappa_l, i_l)$ is clear in the case that $\kappa_l < \kappa^{\plus}$. If $\kappa_l = \kappa^{\plus}$, then $(\kappa_l, i_l) \notin \supp \pi_0 (\kappa_l)$ follows from $\wt{\beta} > i_l^\prime = i_l$. In the case that $\kappa_l > \kappa^{\plus}$, let $\supp \pi_0 (\kappa_l) = \{(\kappa_l, \delta_0), \, \ldots \,, (\kappa_l, \delta_{n (\kappa_l)-1}), (\kappa_l, \wt{\delta}_0), \, \ldots \,, (\kappa_l, \wt{\delta}_{n (\kappa_l)-1})\}$ as before. Recall that we denote by $(\kappa^{\plus}, i_l^\prime)$ the $t(\ol{r})$-predecessor of $(\kappa_l, i_l)$ on level $\kappa^{\plus}$ (which is also its $t(p_0)$-predecessor). Since $\wt{\beta}$ is \textit{large enough for} $(A_{\dot{S}})$, it follows that $i_l^\prime < \wt{\beta}$, so $(\kappa^\plus, i_l^\prime) \notin \{(\kappa^\plus, j_h), \, \ldots \,, (\kappa^\plus, j_{k-1})\}$; thus, $(\kappa_l, i_l) \notin \{(\kappa_l, \delta_0), \, \ldots \,, (\kappa_l, \delta_{n (\kappa_l)-1})\}$. \\ 
 Also, $(\kappa_l, i_l) \in t(r) \subseteq t(p_0)$ gives $(\kappa_l, i_l) \notin \{ (\kappa_l, \wt{\delta}_0), \ldots, (\kappa_l, \wt{\delta}_{n(\kappa_l)-1})\}$. Hence, $(\kappa_l, i_l) \notin \supp \pi_0 (\kappa_l)$ as desired; and it follows that $\pi_0 \in Fix_0 (\kappa_0, i_0)\, \cap\, \cdots\, \cap\, Fix_0(\kappa_{n-1}, i_{n-1})$. \\[-3mm]

Also, $\pi_0 \in Small_0 (\lambda_0, [0, \alpha_0))\, \cap\, \cdots\, \cap\, Small_0 (\lambda_{m-1}, [0, \alpha_{m-1}))$: Let $l < m$. For $\lambda_l \neq \kappa^{\plus}$, there is nothing to show. In the case that $\lambda_l = \kappa^{\plus}$, we use again that $\wt{\beta}$ is \textit{large enough for} $(A_{\dot{S}})$; so $\wt{\beta} > \alpha_0, \ldots, \alpha_{m-1}$, which implies $\supp \pi_0 (\kappa^{\plus})\, \cap \, \{ (\kappa^\plus, i)\ | \ i < \alpha_l\} = \emptyset$. \\[-3mm]

Finally, $\pi p_0\, \| \, p_0$ by construction. \\[-2mm]

Now, we turn to $\pi_1$. Let $\supp \pi_1 := \{\mu_{\ol{h}}, \ldots, \mu_{\ol{k}-1}\}$. \\ [-3mm]

Consider $\mu^\plus \in Succ^\prime$ with $\{l \in [\ol{h}, \ol{k})\ | \ \mu_l = \mu^\plus\} = \{ l_0, \ldots, l_{z-1}\}$ for some $1 \leq z < \omega$.
(Then $\mu_{l_0} = \cdots = \mu_{l_{z-1}} = \mu^\plus$, and $l_0, \, \ldots \,, l_{z-1} \in [\ol{h}, \ol{k})$ implies $\ol{\jmath}_{l_0}, \ldots, \ol{\jmath}_{l_{z-1}} \geq \beta$.) Recall that we have $\pi_1 \ol{\jmath}_{l_0}, \, \ldots \,, \pi_1 \ol{\jmath}_{l_{z-1}} \in (\wt{\beta}, \beta) \setminus \{\ol{\jmath}_0, \ldots, \ol{\jmath}_{\ol{h}-1}\}$ with $\ol{p}_1 \uhr \{(\mu_{l_0},  \ol{\jmath}_{l_0})\} = \ol{p}_1 \uhr \{(\mu_{l_0}, \pi_1 \ol{\jmath}_{l_0})\}, \, \ldots \,, \ol{p}_1 \uhr \{(\mu_{l_{z-1}}, \ol{\jmath}_{l_{z-1}})\} = \ol{p}_1 \uhr \{(\mu_{l_{z-1}}, \pi_1 \ol{\jmath}_{l_{z-1}})\}$. \\
Let $\dom \pi_1 (\mu^\plus) = \dom_x\,  \pi_1 (\mu^\plus) \, \times \, \dom_y\, \pi_1 (\mu^\plus) := \dom_x \, \ol{p}_1 (\mu^\plus)\, \times \, \dom_y\, \ol{p}_1 (\mu^\plus)$, and $\supp \pi_1 (\mu^\plus) := \{\ol{\jmath}_{l_0}, \, \ldots \,, \ol{\jmath}_{l_{z-1}}, \pi_1 \ol{\jmath}_{l_0}, \, \ldots \,, \pi_1 \ol{\jmath}_{l_{z-1}}\}$.
The map $f_{\pi_1} (\mu^\plus): \SUPP \pi_1 (\mu^\plus) \rightarrow \SUPP \pi_1 (\mu^\plus)$ will be defined as follows: $f_{\pi_1}(\mu^\plus) (\ol{\jmath}_l) = \pi_1 \ol{\jmath}_l$, $f_{\pi_1}(\mu^\plus) (\pi_1 \ol{\jmath}_l) = \ol{\jmath}_l$ for all $l \in \{l_0, \ldots, l_{z-1}\}$.

 For $\zeta \in \dom_x \pi_1 (\mu^\plus)$, we need a bijection $\pi_1 (\mu^\plus) (\zeta): 2^{\SUPP \pi_1 (\mu^\plus)} \rightarrow 2^{\SUPP \pi_1 (\mu^\plus)}$. Again, we swap any $\ol{\jmath}_l$-coordinate with the according $\pi_1 \ol{\jmath}_l$-coordinate:

$\big(\, \pi_1 (\mu^\plus) (\zeta) \big(\epsilon_i\ | \ i \in \SUPP \pi_1 (\mu^\plus)\big)\, \big)_{\, \ol{\jmath}_l} := \epsilon_{\pi_1 \ol{\jmath}_l}$, $\big(\pi_1 (\mu^\plus) (\zeta) \big(\epsilon_i\ | \ i \in \SUPP \pi_1 (\mu^\plus)\big) \, \big)_{\,\pi_1 \ol{\jmath}_l} := \epsilon_{\ol{\jmath}_l}$ for $l \in \{l_0, \ldots, l_{z-1}\}$.

Finally, for $(\zeta, i) \in [\mu, \mu^\plus) \, \times\, F (\mu^\plus)$, let $\pi_1 (\mu^\plus) (\zeta, i) = 0$. \\[-3mm]

This defines $\pi_1$, with $D_{\pi_1} = \{q \in \m{P}_1\ | \ \forall\, \mu^\plus \in \supp q\, \cap\, \{\mu_{\ol{h}}, \ldots, \mu_{\ol{k}-1}\}\ \dom q (\mu^\plus) \supseteq \dom \pi_1 (\mu^\plus)\}$. \\[-3mm]

For any such $q \in D_{\pi_1}$ and $\mu^\plus \in \supp q$ with $\mu^\plus = \mu_{l_0} = \cdots = \mu_{l_{z-1}}$ for $\l_0, \ldots, l_{z-1}$ as above, we have $\{\ol{\jmath}_{l_0}, \, \ldots \,, \ol{\jmath}_{l_{z-1}}, \pi_1 \ol{\jmath}_{l_0}, \, \ldots \, \pi_1 \ol{\jmath}_{l_{z-1}}\} \subseteq \dom_y \ol{p} (\mu^\plus) = \dom_y \pi_1 (\mu^\plus) \subseteq \dom_y q_1 (\mu^\plus)$, and $(\pi_1 q) (\mu^\plus) (\zeta, \ol{\jmath}_l) = q (\mu^\plus) (\zeta, \pi_1 \ol{\jmath}_l)$, $(\pi_1 q)(\mu^\plus) (\zeta, \pi_1 \ol{\jmath}_l) = q (\mu^\plus) (\zeta, \ol{\jmath}_l)$ for all $l \in \{l_0, \ldots, l_{z-1}\}$, $\zeta \in \dom_x q (\mu^\plus)$. Moreover,
$(\pi_1 q) (\mu^\plus) (\zeta, i) = q (\mu^\plus) (\zeta, i)$ for all the $\zeta \in \dom_x q (\mu^\plus)$, $i < F(\mu^\plus)$ remaining with $i \notin \supp \pi_1 (\mu^\plus) = \{\ol{\jmath}_{l_0}, \ldots, \ol{\jmath}_{l_{z-1}}$, $\pi_1 \ol{\jmath}_{l_0}, \ldots, \pi_1 \ol{\jmath}_{l_{z-1}}\}$. \\
Since we have arranged that $\ol{p}_1 \uhr \{(\mu_{l_0},  \ol{\jmath}_{l_0})\} = \ol{p}_1 \uhr \{(\mu_{l_0}, \pi_1 \ol{\jmath}_{l_0})\},$ $\, \ldots \,, \ol{p}_1 \uhr \{(\mu_{l_{z-1}}, \ol{\jmath}_{l_{z-1}})\} = \ol{p}_1 \uhr \{(\mu_{l_{z-1}}, \pi_1 \ol{\jmath}_{l_{z-1}})\}$, it follows that $(\pi_1 \ol{p}_1) (\mu^\plus) = \ol{p}_1 (\mu^\plus)$. \\Hence, $\pi_1 \ol{p}_1\, \| \, \ol{p}_1$. \\[-2mm]

It remains to check that $\pi_1 \in Fix_1 (\ol{\kappa}_0, \ol{\imath}_0)\, \cap\, \cdots\, \cap\, Fix_1 (\ol{\kappa}_{\ol{n}-1}, \ol{\imath}_{\ol{n}-1})\, \cap\, $ $Small_1 (\ol{\lambda}_0, [0, \alpha_0))\, $ $\cap\, \cdots\, \cap\, Small_1 (\ol{\lambda}_{\ol{m}-1}, [0, \alpha_{\ol{m}-1}))$. For $l < \ol{n}$ and $\ol{\kappa}_l < \kappa$, we have $\ol{\imath}_l < \wt{\beta}$, since $\wt{\beta}$ is \textit{large enough}, so $\ol{\imath}_l \notin \SUPP \pi_1 (\ol{\kappa}_l)$.
Hence, $\pi_1 \in Fix_1 (\ol{\kappa}_l, \ol{\imath}_l)$. In the case that $\ol{\kappa}_l \geq \kappa$, it follows that $\ol{\kappa}_l \notin \supp \pi_1$,
so again, $\pi_1 \in Fix_1 (\ol{\kappa}_l, \ol{\imath}_l)$ as desired. Similarly, $\pi_1 \in Small_1 (\ol{\lambda}_0, [0, \alpha_0))\, \cap\, \cdots\, \cap\, Small_1 (\ol{\lambda}_{\ol{m}-1}, [0, \alpha_{\ol{m}-1}))$. \\[-2mm]

Thus, we have constructed an automorphism $\pi = (\pi_0, \pi_1)$ with $\pi \ol{p} \, \| \, \ol{p}$ and $\pi \in Fix_0 (\kappa_0, i_0)\, \cap\, \cdots\, \cap\, Small_0 (\lambda_0, [0, \alpha_0))\, \cap\, \cdots\, \cap\, Fix_1 (\ol{\kappa}_0, \ol{\imath}_0)\, \cap\, \cdots\, \cap \, Small_1 (\ol{\lambda}_0, [0, \ol{\alpha}_0))\, \cap\, \cdots$. This gives $\pi \ol{\dot{S}}^{D_\pi} = \ol{\dot{S}}^{D_\pi}$. \\[-2mm]

Since $p \Vdash_{\m{P} \, \uhr \, (\lambda + 1)} OR_{\m{P} \, \uhr \, (\lambda + 1)} (\wt{\dot{X}}, \alpha) \in \dot{S}$,
it follows that 
$\pi \ol{p} \Vdash_{\m{P} \, \uhr \, (\lambda + 1)} OR_{\m{P} \, \uhr \, (\lambda + 1)} (\pi {\wt{\dot{X}}}^{D_\pi}, \alpha) \in \pi \ol{\dot{S}}^{D_\pi}$;
hence, \[\pi \ol{p} \Vdash_{\m{P} \, \uhr \, (\lambda + 1)} OR_{\m{P} \, \uhr \, (\lambda + 1)} (\pi {\wt{\dot{X}}}^{D_\pi}, \alpha) \in \dot{S}\,. \]
Take $q \in \m{P} \uhr (\lambda + 1)$ with $q \leq \ol{p}, \pi \ol{p}$. Then $q \Vdash_{\m{P} \, \uhr \, (\lambda + 1)} \alpha \notin \rg \, \dot{S}^\beta$, and $q \Vdash_{\m{P} \, \uhr \, (\lambda + 1)} OR_{\m{P} \, \uhr \, (\lambda + 1)} (\pi {\wt{\dot{X}}}^{D_\pi}, \alpha) \in \dot{S}$. We will lead this into a contradiction. \\[-2mm]

As already indicated, $\pi {\wt{\dot{X}}}^{D_\pi}$ will be equal to some ${\wt{\ddot{X}}}^{D_\pi}$, where $\ddot{X}$ is a name for the forcing \[\m{P}_0 \uhr t(\pi_0 s) \, \times\, \m{P}_1 \uhr \{ (\mu_0, j_0), \, \ldots \,, (\mu_{\ol{h} - 1}, \ol{\jmath}_{\ol{h}-1}), (\mu_{\ol{h}}, \pi_1 \ol{\jmath}_{\ol{h}}), \, \ldots \,, (\mu_{\ol{k}-1}, \pi_1 \ol{\jmath}_{\ol{k}-1})\}\, \times\, \m{P}_1 (\kappa^{\plus}), \] where $\pi_0 s \in \m{P}_0$ has maximal points $(\kappa^{\plus}, j_0), \dots, (\kappa^{\plus}, j_{h-1}), (\kappa^{\plus}, \pi_0 j_h), \, \ldots \,, (\kappa^{\plus}, \pi_0 j_{k-1})$. \\[-3mm]

More generally, for a name $\dot{x} \in \Name (\m{P}_0 \uhr t(s)\, \times\, \m{P}_1 \uhr \{ (\mu_0, \ol{\jmath}_0), \, \ldots \,, (\mu_{\ol{k}-1}, \ol{\jmath}_{\ol{k}-1})\}\, \times\, \m{P}_1 (\kappa^{\plus}))$, we cannot apply $\pi$ directly to obtain $\pi \dot{x}$, but have transform $\dot{x}$ into a $\m{P}$-name $\wt{\dot{x}}$ first, and then consider the extension ${\wt{\dot{x}}}^{D_\pi}$. \\[-3mm]

However, the map $\pi$ induces a canonical isomorphism $T_\pi: \m{P}_0 \uhr t(s)\, \times\, \m{P}_1 \uhr \{ (\mu_0, \ol{\jmath}_0), \, \ldots \,, $ $(\mu_{\ol{k}-1}, \ol{\jmath}_{\ol{k}-1})\}\, \times\, \m{P}_1 (\kappa^{\plus}) \rightarrow \m{P}_0 \uhr t(\pi_0 s)\, \times\, \m{P}_1 \uhr \{(\mu_0, \ol{\jmath}_0), \, \ldots \,, (\mu_{\ol{h}-1}, \ol{\jmath}_{\ol{h}-1}), (\mu_{\ol{h}}, \pi_1 \ol{\jmath}_{\ol{h}}), \, \ldots \,, $ $(\mu_{\ol{k}-1}, \pi_1 \ol{\jmath}_{\ol{k}-1})\}\, \times\, \m{P}_1 \uhr (\kappa^{\plus})$, which extends to the name space, such that for all $\dot{x} \in \Name (\m{P}_0 \uhr t(s)\, \times\, \m{P}_1 \uhr \{ (\mu_0, \ol{\jmath}_0), \, \ldots \,, (\mu_{\ol{k}-1}, \ol{\jmath}_{\ol{k}-1})\}\, \times\, \m{P}_1 (\kappa^{\plus}))$, we have \[{\wt{T_\pi \dot{x}}}^{D_\pi} = \pi {\wt{\dot{x}}}^{D_\pi}.\]

This transformation $T_\pi$ can be defined as follows: \\ Recall that $s$ is a condition in $\m{P}_0$ with maximal points $(\kappa^{\plus}, j_0), \, \ldots \,, (\kappa^{\plus}, j_{k-1})$, so the condition $\pi_0 s$ has maximal points $(\kappa^{\plus}, j_0), \, \ldots \,, (\kappa^{\plus}, j_{h-1})$, $(\kappa^{\plus}, \pi_0 j_h), \, \ldots \,, $ $(\kappa^{\plus}, \pi_0 j_{k-1})$ with $(\pi_0 s) \uhr \kappa^{\plus} = s \uhr \kappa^{\plus}$, and for any $l < h$, it follows that $\pi_0 s$ has the same branch below $(\kappa^\plus, j_l)$ as $s$; but for $l \in [h, k)$, the $\pi_0 s$-branch below $(\kappa^\plus, \pi_0 j_l)$ coincides with the $s$-branch below $(\kappa^\plus, j_l)$. \\ [-3mm]

For a condition \[\big(\, v_0 \uhr t(s), v_1 \uhr \{(\mu_0, \ol{\jmath}_0), \, \ldots \,, (\mu_{\ol{k}-1}, \ol{\jmath}_{\ol{k}-1})\}, v_1 (\kappa^{\plus})\, \big)\] in $\m{P}_0 \uhr t(s)\, \times\, \m{P}_1 \uhr \{ (\mu_0, \ol{\jmath}_0), \, \ldots \,, (\mu_{\ol{k}-1}, \ol{\jmath}_{\ol{k}-1})\}\, \times\, \m{P}_1 (\kappa^{\plus})$, let \[T_\pi \big(v_0 \uhr t(s), v_1 \uhr \{(\mu_0, \ol{\jmath}_0), \, \ldots \,, (\mu_{\ol{k}-1}, \ol{\jmath}_{\ol{k}-1})\}, v_1 (\kappa^{\plus})\big)\] be the condition \[\big(\, v_0^\prime \uhr t(\pi_0 s), v_1 ^\prime \uhr \{(\mu_0, \ol{\jmath}_0), \, \ldots \,, (\mu_{\ol{h}-1}, \ol{\jmath}_{\ol{h}-1}), (\mu_{\ol{h}}, \pi_1 \ol{\jmath}_{\ol{h}}), \, \ldots \,, (\mu_{\ol{k}-1}, \pi_1 \ol{\jmath}_{\ol{k}-1})\}, v_1^\prime (\kappa^{\plus})\, \big)\] with \begin{itemize} \item $v_0^\prime \uhr t(\pi_0 s) = \pi_0 (v_0 \uhr t(s))$,

\item $v_1^\prime \uhr \{(\mu_0, \ol{\jmath}_0), \, \ldots \,, (\mu_{\ol{h}-1}, \ol{\jmath}_{\ol{h}-1}), (\mu_{\ol{h}}, \pi_1 \ol{\jmath}_{\ol{h}}), \, \ldots \,, (\mu_{\ol{k}-1}, \pi_1 \ol{\jmath}_{\ol{k}-1})\}$ $\in$ $\m{P}_1 \uhr \{(\mu_0, \ol{\jmath}_0), \, \ldots \,, $ $(\mu_{\ol{h}-1}, \ol{\jmath}_{\ol{h}-1}), (\mu_{\ol{h}}, \pi_1 \ol{\jmath}_{\ol{h}}), \, \ldots \,, (\mu_{\ol{k}-1}, \pi_1 \ol{\jmath}_{\ol{k}-1})\}$ is obtained from $v_1 \uhr \{(\mu_0, \ol{\jmath}_0), \, \ldots \,, $ $(\mu_{\ol{k}-1}, \ol{\jmath}_{\ol{k}-1})\}$ by swapping any $(\mu_l, \ol{\jmath}_l)$-coordinate for $l \in [\ol{h}, \ol{k})$ with the according $(\mu_l, \pi_1 \ol{\jmath}_l)$-coordinate,

\item $v_1^\prime (\kappa^{\plus}) = v_1 (\kappa^{\plus})$.
\end{itemize} 

Then $T_\pi$ induces a canonical transformation of names $T_\pi: \Name \big(\m{P}_0 \uhr t(s)\, \times\, \m{P}_1 \uhr \{ (\mu_0, \ol{\jmath}_0), \, \ldots \,,$ $(\mu_{\ol{k}-1}, \ol{\jmath}_{\ol{k}-1})\}\, \times\, \m{P}_1 (\kappa^{\plus})\big) \rightarrow \Name \big(\m{P}_0 \uhr t(\pi_0 s)\, \times\, \m{P}_1 \uhr \{(\mu_0, j_0), \, \ldots \,, $ $(\mu_{\ol{h}-1}, j_{\ol{h}-1}), (\mu_{\ol{h}}, \pi_1 j_{\ol{h}}), \, \ldots \,, (\mu_{\ol{k}-1}, \pi_1 j_{\ol{k}-1})\}\, \times\, \m{P}_1 \uhr (\kappa^{\plus})\big)$, which will be denoted by the same letter. \\ Recursively, it is not difficult to check that indeed, $\pi {\wt{\dot{x}}}^{D_\pi} = {\wt{T_\pi \dot{x}}}^{D_\pi}$. \\[-2mm]

Thus, from \[q \Vdash_{\m{P}  \, \uhr \, (\lambda + 1)} OR_{\m{P} \, \uhr \, (\lambda + 1)} (\pi {\wt{\dot{X}}}^{D_\pi}, \alpha) \in \dot{S}\] it follows that \[q \Vdash_{\m{P} \, \uhr \, (\lambda + 1)} OR_{\m{P} \, \uhr \, (\lambda + 1)} ({\wt{T_\pi \dot{X}}}^{D_\pi}, \alpha) \in \dot{S}.\]
 Now, $T_\pi \dot{X} \in \Name \big(\m{P}_0 \uhr t(\pi_0 s)\ \times\, \m{P}_1 \uhr \{(\mu_0, \ol{\jmath}_0), \, \ldots \,, (\mu_{\ol{h}-1}, \ol{\jmath}_{\ol{h}-1}), $ $(\mu_{\ol{h}}, \pi_1 \ol{\jmath}_{\ol{h}}), \, \ldots \,, $ $(\mu_{\ol{k}-1}, \pi_1 \ol{\jmath}_{\ol{k}-1})\}\, \times\, \m{P}_1 (\kappa^{\plus})\big)$, where $\pi_0 s$ has maximal points $(\kappa^{\plus}, j_0), \, \ldots \,, $ $(\kappa^{\plus}, j_{h-1}), $ $(\kappa^{\plus}, \pi_0 j_h), \, \ldots \,, $ $(\kappa^{\plus}, \pi_0 j_{k-1})$ with $j_0 < \beta, \, \ldots \,, j_{h-1} < \beta$, and $\pi_0 j_h < \beta, \, \ldots \,, \pi_0 j_{k-1} < \beta$ by construction. Also, $\ol{\jmath}_0 < \beta, \, \ldots \,, \ol{\jmath}_{\ol{h}-1} < \beta$, and $\pi_1 \ol{\jmath}_{\ol{h}} < \beta, \, \ldots \,, \pi_1 \ol{\jmath}_{\ol{k}-1} < \beta$ by construction. Thus, $(\pi_0 s, (\mu_0, \ol{\jmath}_0), \, \ldots \,, $ $(\mu_{\ol{h}-1}, \ol{\jmath}_{\ol{h}-1}), $ $(\mu_{\ol{h}}, \pi_1 \ol{\jmath}_{\ol{h}}), \, \ldots \,, $ $(\mu_{\ol{k}-1}, \pi_1 \ol{\jmath}_{\ol{k}-1})) \in M_\beta$. \\ Since $q_0 \leq \pi_0 \ol{p}_0 \leq \pi_0 s$ and $q \Vdash_{\m{P} \, \uhr \, (\lambda + 1)} OR_{\m{P} \, \uhr \, (\lambda + 1)} (\wt{T_\pi \dot{X}}, \alpha) \in \dot{S}$, %ACHTUNG - waere $p_0 \leq s$ unproblematisch? \\      
it follows 
%? 
that $\big(\, OR_{\m{P} \, \uhr \, (\lambda + 1)} (\wt{T_\pi \dot{X}}^{D_\pi}, \alpha), q\, \big) \in \dot{S}^\beta$, 
contradicting that also $q \Vdash_{\m{P} \, \uhr \, (\lambda + 1)} \alpha \notin \rg \dot{S}^\beta$. \\[-2mm]

Hence, $S^\beta$ must be surjective, which finishes the proof. \end{proof} 

Thus, we have shown that for any $\wt{\beta} < F(\kappa)$ large enough and $\beta = \wt{\beta} + \kappa^\plus$, the restriction $S^\beta: \dom S^\beta \rightarrow F(\kappa)$ must be surjective, as well. \\ We will now lead this into a contradiction. \\[-2mm] 

For the rest of this section, we fix some limit ordinal $\wt{\beta} < F(\kappa)$ {\textit{large enough}} and let $\beta := \wt{\beta} + \kappa^\plus$. We want to capture $S^\beta$ in an intermediate model $V[G^\beta \uhr (\kappa^{\plus} + 1)]$, which will be a generic extension by a certain set forcing $\m{P}^\beta \uhr (\kappa^{\plus} + 1)$. We will show that $V[G^\beta \uhr (\kappa^{\plus} + 1)]$ also contains an injection $\iota: \dom S^\beta \hookrightarrow \beta$, while $\m{P}^\beta \uhr (\kappa^\plus + 1)$ preserves all cardinals $\geq F(\kappa)$ -- a contradiction. \\[-3mm]

Roughly speaking, this forcing $\m{P}^\beta \uhr (\kappa^{\plus} + 1)$ will be obtained from $\m{P}$, by first cutting off at height $\kappa^{\plus} + 1$, and then cutting off at width $\beta$. The latter procedure is rather clear for $\m{P}_1$: For successor cardinals $\lambda^\plus < \kappa$, $\lambda^\plus \in Succ^\prime$, we take for $(\m{P}_1)^\beta (\lambda^\plus)$ the forcing $Fn ( [\lambda, \lambda^\plus) \times \beta, 2, \lambda^\plus)$ instead of $Fn ( [\lambda, \lambda^\plus) \times F(\lambda^\plus), 2, \lambda^\plus)$ in the case that $\beta < F(\lambda^\plus)$. However, the forcing notion $(\m{P}_0)^\beta \uhr (\kappa^{\plus} + 1)$ requires a careful construction. One could try and restrict $\m{P}_0$ to all those $p \in \m{P}_0 \uhr (\kappa^\plus + 1)$ which have only maximal points $(\kappa^{\plus}, i)$ with $i < \beta$. Nevertheless, their predecessors $(\lambda, j)$ on lower levels $\lambda < \kappa^\plus$ might still have indices $j > \beta$, so our forcing would still be \tbl too big\tbr. \\
Our idea will be to drop all indices at levels below $\kappa^{\plus}$ -- then the domain $t (p)$ of the conditions $p \in (\m{P}_0)^\beta \uhr (\kappa^\plus + 1)$ will be given by their maximal points $(\kappa^{\plus}, i)$ and the structure of the tree below, i.e. for any two maximal points $(\kappa^{\plus}, i)$ and $(\kappa^{\plus}, i^\prime)$ we only need information about the level at which the branches below them meet. \\[-2mm]

We start with a \tbl preliminary version\tbr\, $(\widehat{\m{P}}_0)^\beta \uhr (\kappa^\plus + 1)$: Any condition $p \in (\widehat{\m{P}}_0)^\beta \uhr (\kappa^\plus + 1)$ will be of the form $p: t (p) \rightarrow V$ with a tree $t (p)$ given by its finitely many maximal points $(\kappa^{\plus}, \beta_0), \, \ldots \,, (\kappa^{\plus}, \beta_{k-1})$ and the tree structure below. We will now specify how this tree structure should be coded into the forcing conditions: \\[-3mm]

On the one hand, for any level $\alpha \leq \kappa^{\plus}$, the tree structure of $t(p)$ induces 
an equivalence relation $\sim_{\alpha}$ on the set $\{ \beta_0, \, \ldots \,, \beta_{k-1}\}$ by setting $\beta_i \sim_{\alpha} \beta_j$ iff $(\kappa^{\plus}, i)$ and $(\kappa^{\plus}, j)$ have a common $t(p)$-predecessor on level $\alpha$. This equivalence relation $\sim_{\alpha}$ induces a partition $A_\alpha$ on $\{\beta_0, \, \ldots \,, \beta_{k-1}\}$ such that for all $l, l^\prime < k$, there exists $z \in A_\alpha$ with $\{\beta_l, \beta_{l^\prime}\} \subseteq z$ iff the vertices $(\kappa^\plus, \beta_l)$ and $(\kappa^\plus, \beta_{l^\prime})$ have a common $t(p)$-predecessor on level $\alpha$. \\[-3mm]

Conversely, the tree structure below $(\kappa^\plus, \beta_0), \, \ldots \,, (\kappa^\plus, \beta_{k-1})$ could be described by a sequence $(A_\alpha\ | \ \alpha \leq \kappa^\plus, \alpha \in \Card)$ of partitions of the set $\{\beta_0, \ldots, \beta_{k-1}\}$ such that any $A_{\alpha^\plus}$ is finer than $A_{\alpha}$, and $A_0 = \{ \{\beta_0, \ldots, \beta_{k-1}\}\}$, $A_{\kappa^\plus} = \{ \{\beta_0\}, \ldots, \{\beta_{k-1}\}\}$. Since for $F_{\lim}$-trees we do not allow splitting at limits, we have to require that for any limit cardinal $\alpha \leq \kappa$, there exists a cardinal $\ol{\alpha} < \alpha$ such that $A_\alpha = A_\beta$ for all $\beta$ with $\ol{\alpha} \leq \beta \leq \alpha$. \\ We will give any $t(p)$-vertex on level $\alpha \leq \kappa^{\plus}$ a \tbl name\tbr\, $(\alpha, z)$, where $z \in A_\alpha$ is the collection of all $i < k$ with $(\alpha, z) \leq_{t(p)} (\kappa^{\plus}, \{\beta_i\})$. Then the vertices already determine the tree structure of $t(p)$.

\begin{definition} Let $k < \omega$ and $\beta_0, \, \ldots \,, \beta_{k-1} < F_{\lim}(\kappa^\plus) = F(\kappa)$. We say that $(t, \leq_t)$ is a {\normalfont tree
below $(\kappa^{\plus}, \beta_0), \, \ldots \,, $ $(\kappa^{\plus}, \beta_{k-1})$}, iff there is a sequence $(A_\alpha\ | \ \alpha \leq \kappa^\plus, \alpha \in \Card)$ of partitions of the set $\{ \beta_0, \, \ldots \,, \beta_{k-1}\}$ such that \begin{itemize} \item for any cardinal $\alpha < \kappa^\plus$, it follows that $A_{\alpha^\plus}$ is finer than $A_\alpha$, $A_0 = \{ \{\beta_0 ,\ldots, \beta_{k-1}\}\}$, and $A_{\kappa^\plus} =  \{\{\beta_0\}, \, \ldots \,, \{\beta_{k-1}\} \}$, \item for all limit cardinals $\alpha$, there exists $\ol{\alpha} < \alpha$ with $A_\beta = A_\alpha$ for all $\ol{\alpha} \leq \beta \leq \alpha$, \end{itemize} such that \[t\, := \, \bigcup_{\substack{\alpha \in \Card \\ \alpha \leq \kappa^\plus}} \{\alpha\} \times A_\alpha,\]
i.e., the vertices of $t$ are pairs $(\alpha, z)$ with $ z \in A_\alpha$ a subset of $\{\beta_0, \ldots, \beta_{k-1}\}$. \\[-3mm]

The order $\leq_t$ is defined as follows: $(\alpha, z ) \leq_t (\beta, z^\prime)$ iff $\alpha \leq \beta$ and $z \supseteq z^\prime$. \\[-3mm]

We call $\supp t = \{(\kappa^\plus, \beta_0), \, \ldots \,, (\kappa^\plus, \beta_{k-1})\}$ the {\normalfont support of} $(t, \leq_t)$. \\[-3mm]

For $\beta \leq F_{\lim} (\kappa^\plus)$, we denote by $T(\kappa^\plus,\beta)$ the collection of all $(t, \leq_t)$ such that $(t, \leq_t)$ is a tree below some $(\kappa^{\plus}, \beta_0), \, \ldots \,, (\kappa^{\plus}, \beta_{k-1})$ with $k < \omega$ and $\beta_0, \, \ldots \,, \beta_{k-1} < \beta$.

\end{definition}

There is a canonical partial order $\leq_{T(\kappa^{\plus}, \beta)}$ on $T(\kappa^{\plus}, \beta)$: Set $(s, \leq_s) \leq_{T(\kappa^\plus, \beta)} (t, \leq_t)$ iff $supp\, s \supseteq supp \, t$, and the tree structures of $s$ and $t$ below $supp$ $t$ agree, i.e. for any $i, j \in supp$ $t$, the $(t, \leq_t)$-branches below $(\kappa^{\plus}, i)$ and $(\kappa^{\plus}, j)$ meet at the same level as they do in $(s, \leq_s)$.

\begin{definition} Let $(t, \leq_t), (s, \leq_s) \in T(\kappa^{\plus}, \beta)$ with $\supp t = \{ \beta_0, \, \ldots \,, \beta_{k-1}\}$, $\supp s = \{ \ol{\beta}_0, \, \ldots \,, \ol{\beta}_{\ol{k}-1}\}$, and the according sequences of partitions $(A_\alpha\ | \ \alpha \in \Card, \alpha \leq \kappa^\plus)$ and $(\ol{A}_\alpha\ | \ \alpha \in \Card, \alpha \leq \kappa^\plus)$. Then $(s, \leq_s) \leq_{T(\kappa^{\plus}, \beta)} (t, \leq_t)$ iff the following hold: \begin{itemize} \item $\supp s = \{\ol{\beta}_0, \, \ldots\, , \ol{\beta}_{\ol{k}-1}\} \supseteq \{\beta_0, \, \ldots\, , \beta_{k-1}\} = \supp t$, 
\item for any $\alpha \leq \kappa^\plus$, the partition $\ol{A}_\alpha$ extends $A_\alpha$, i.e. for any $\beta_l ,\beta_{l^\prime} \in \supp t$, \[(\exists\, z \in A_\alpha\  \{ \beta_l, \beta_{l^\prime} \} \subseteq z) \ \Leftrightarrow\ (\exists\, \ol{z} \in \ol{A}_\alpha\ \{\beta_l, \beta_{l^\prime}\} \subseteq \ol{z}). \]\end{itemize} \end{definition}

One can check that $\leq_{T(\kappa^\plus, \beta)}$ is indeed a partial order. \\[-3mm]

For trees $(s, \leq_s)$ and $(t, \leq_t)$ in $T(\kappa^{\plus}, \beta)$ with $(s, \leq_s) \leq_{T(\kappa^{\plus}, \beta)} (t, \leq_t)$, we can define an embedding $\iota: (t, \leq_t) \hookrightarrow (s, \leq_s)$ as follows: $\iota (\alpha, z) := (\alpha, \ol{z})$, where $(\alpha, \ol{z}) \in s$ with $\ol{z} \supseteq z$ (then $z = \ol{z} \, \cap \, \supp t$). 
With $\leq_{\iota [t]} \, := \, \iota[\leq_t] = \{ (\iota (\alpha, z), \iota (\beta, z^\prime))\ | \ (\alpha, z) \leq_t (\beta, z^\prime)\}$, it follows that $\leq_{\iota[t]} \, = \, \leq_s \, \cap \, \iota[t]$, and $(\iota [t], \leq_{\iota [t]}) \subseteq (s, \leq_s)$ is a subtree. \\
Conversely, consider $s, t \in T(\kappa^{\plus}, \beta)$ with an embedding $\iota: (t, \leq_t) \hookrightarrow (s, \leq_s)$ such that for all $(\alpha, z) \in t$, we have $\iota (\alpha, z) = (\alpha, \ol{z})$ with $\ol{z} \supseteq z$. Then $(\iota[t], \iota[\leq_t]) \subseteq (s, \leq_s)$ is a subtree, and one can easily check that $(s, \leq_s) \leq_{T(\kappa^\plus, \beta)} (t, \leq_t)$. \\ Hence, the partial order $\leq_{T(\kappa^{\plus}, \beta)}$ can also be described via embeddings. \\[-3mm]

The maximal element of $T(\kappa^{\plus}, \beta)$ is the empty tree.\\[-3mm]

Now, we can define $(\widehat{\m{P}}_0)^\beta \uhr (\kappa^{\plus} + 1)$:

\begin{definition} \label{defforc} The forcing $(\widehat{\m{P}}_0)^\beta \uhr (\kappa^{\plus} + 1)$ consists of all $p: t (p) \rightarrow V$ such that \begin{itemize} \item $t (p) \in T(\kappa^{\plus}, \beta)$, \item $p( \alpha^\plus, z) \in Fn ( [\alpha, \alpha^\plus), 2, \alpha^\plus)$ for all $(\alpha^\plus, z) \in t (p)$ with $\alpha^\plus$ a successor cardinal, \item $p (\aleph_0, z) \in Fn( \aleph_0, 2, \aleph_0)$ for all $(\aleph_0, z) \in t (p)$, \item $p( \alpha, z) = \emptyset$ for all $(\alpha, z) \in t(p)$ with $\alpha$ a limit cardinal{\normalfont, and} \item $|p \uhr \alpha| < \alpha$ for all regular limit cardinals $\alpha$. \end{itemize} For $\ol{p}, p \in (\widehat{\m{P}}_0)^\beta \uhr (\kappa^{\plus} + 1)$, set $ \ol{p} \leq p$ iff \begin{itemize} \item $t (\ol{p}) \leq_{T(\kappa^\plus, \beta)} t (p)$, \item $\ol{p}( \alpha, \ol{z}) \supseteq p(\alpha, z)$ whenever $\ol{z} \supseteq z$.\end{itemize} The maximal element $\m{1}$ in $(\widehat{\m{P}}_0)^\beta \uhr (\kappa^\plus + 1)$ is the empty condition with $t(\m{1}) = \emptyset$. \end{definition}

Our argument for capturing $S^\beta$ inside $V[G^\beta \uhr (\kappa^{\plus} + 1)]$ will roughly be as follows: We define a function $(S^\beta)^\prime$ as the set of all $(\dot{X}^{G^\beta\, \uhr\, (\kappa^\plus + 1)}, \alpha)$ for an appropriate name $\dot{X}$, such that there exists $p \in \m{P}$ with $p \Vdash (\dot{X}, \alpha) \in \dot{S}$ and $p^\beta\, \uhr \, (\kappa^\plus + 1) \in G^\beta \uhr (\kappa^\plus + 1)$. In order to show that $(S^\beta)^\prime \subseteq S^\beta$, we use an isomorphism argument similarly as before: If there was $(\dot{X}^{G^\beta \, \uhr \ (\kappa^\plus + 1)}, \alpha) \in (S^\beta)^\prime \setminus S^\beta$, one could take $p$ and $q$ in $\m{P}$ with $p^\beta \uhr (\kappa^\plus + 1) \in G^\beta \uhr (\kappa^\plus + 1)$, $q \in G$ such that $p \Vdash (\dot{X}, \alpha) \in \dot{S}$ and $q \Vdash (\dot{X}, \alpha) \notin \dot{S}$. We construct an automorphism $\pi$ with $\pi p \, \| \, q$ with $\pi {\wt{\dot{X}}}^{D_\pi} = {\wt{\dot{X}}}^{D_\pi}$ and $\pi {\ol{\dot{S}}}^{D_\pi} = {\ol{\dot{S}}}^{D_\pi}$, and obtain a contradiction. \\[-3mm] 

Recall that prior to the proof of Proposition \ref{theta 1}, we have fixed a condition $r \in G_0$ such that the maximal points of $t(r)$ are among $\{(\kappa_0, i_0), \, \ldots \,, (\kappa_{n-1}, i_{n-1})\} \subseteq t(r)$, and $\ol{r} \in G_0$, $\ol{r} \leq r$, such that all branches of $\ol{r}$ have height $\geq \kappa^{\plus}$. For $l < n$ with $\kappa_l \geq \kappa^{\plus}$, we denote by $(\kappa^{\plus}, i_l^\prime)$ the $t(\ol{r})$-predecessor of $(\kappa_l, i_l)$ on level $\kappa^{\plus}$; in the case that $\kappa_l < \kappa^{\plus}$, we have chosen for $(\kappa^{\plus}, i_l^\prime)$ some $t(\ol{r})$-successor of $(\kappa_l, i_l)$ on level $\kappa^\plus$. \\[-3mm] 

Firstly, in order to make sure that $\pi p\, \| \, q$ is possible while at the same time $\pi \in Fix_0 (\kappa_0, i_0)\, \cap\, \cdots\, \cap \, Fix_0 (\kappa_{n-1}, i_{n-1})$, it will be necessary that from $(p_0)^\beta \uhr (\kappa^\plus + 1) \in (G_0)^\beta \uhr (\kappa^\plus + 1)$, $q \in G$, it follows that $p$ and $q$ coincide on the tree $t(\ol{r})$.
Thus, we will have to include $t(\ol{r})$ into our forcing $(\widehat{\m{P}}_0)^\beta \uhr (\kappa^\plus + 1)$: Namely, we will restrict $(\widehat{\m{P}}_0)^\beta \uhr (\kappa^\plus + 1)$ to those conditions that coincide with $t(\ol{r})$ below level $\kappa^\plus$.

Secondly, for $\pi \in Small_0 (\lambda_0, [0, \alpha_0))\, \cap\, \cdots\, \cap \, Small_0 (\lambda_{m-1}, [0, \alpha_{m-1}))$, we will have to make sure that $(p_0)^\beta \uhr (\kappa^\plus + 1) \in (G_0)^\beta \uhr (\kappa^\plus + 1)$, $q \in G$ implies that for all $l < m$, the indices $(\lambda_l, i)$ at level $\lambda_l$ agree for $p$ and $q$ for all $i < \alpha_l$. In order to achieve this, we enhance our forcing $(\widehat{\m{P}}_0)^\beta \uhr (\kappa^\plus + 1)$ and assign indices $(\lambda, i)$ with $i < \alpha_l$ to some some vertices $(\lambda_l, z)$. \\[-2mm]

We start with the second, defining a forcing $\big((\widehat{\m{P}}_0)^\beta \uhr (\kappa^{\plus} + 1)\big)_{(\lambda_0, \alpha_0), \, \ldots \,}$ that will be the collection of all $p \in (\widehat{\m{P}}_0)^\beta \uhr (\kappa^{\plus} + 1)$ equipped with an additional indexing function $N (p)$ on $\{ (\lambda_l, z) \in t (p)\ | \ l < m, \lambda_l \leq \kappa\}$ such that \begin{itemize} \item $N(p)(\lambda_l, z) \in \{(\lambda_l, i)\ | \ i < \alpha_l\} \cup \{ \ast\}$ for all $(\lambda_l, z) \in$ $\dom$ $N(p)$, \item any $(\lambda, i) \in rg \; N(p)$ has only one preimage:
\[\big(\, N (p)(\lambda_l, z) = N (p)(\lambda_l, z^\prime)\, \wedge\, z \neq z^\prime\, \big) \Rightarrow N (p)(\lambda_l, z) = N (p)(\lambda_l, z^\prime) = \ast.\] \end{itemize}

The idea about this indexing function $N(p)$ is that for a condition $p \in \big((\widehat{\m{P}}_0)^\beta \uhr (\kappa^{\plus} + 1)\big)_{(\lambda_0, \alpha_0), \, \ldots \,}$, any vertex $(\lambda_l, z) \in t(p)$ with $N(p) (\lambda_l, z) = (\lambda_l, i)$ for some $i < \alpha_l$ should correspond to the vertex $(\lambda_l, i)$ for conditions in $\m{P}_0$, while all vertices $(\lambda_l, z) \in t(p)$ with $N(p) (\lambda_l, z) = \ast$ should correspond to vertices $(\lambda_l, i)$ with $i \geq \alpha_l$. \\[-3mm]

For $\ol{p}$, $p \in \big((\widehat{\m{P}}_0)^\beta \uhr (\kappa^{\plus} + 1) \big)_{(\lambda_0, \alpha_0), \, \ldots \,}$ with indexing functions $N(p)$ and $N(\ol{p})$, we set $\ol{p} \leq p$ iff $\ol{p} \leq p$ in $(\widehat{\m{P}}_0)^\beta \uhr (\kappa^{\plus} + 1)$, and $N(\ol{p})(\lambda_l, \ol{z}) = N(p) (\lambda_l, z)$ for all $\ol{z} \supseteq z$.  \\[-2mm] 

Now, we define our forcing $\big((\widehat{\m{P}}_0)^\beta \uhr (\kappa^{\plus} + 1)\big)_{(\lambda_0, \alpha_0), \, \ldots \,}^{\ol{r}}$, which could be regarded the collection of all those
conditions $p \in \big((\widehat{\m{P}}_0)^\beta \uhr (\kappa^{\plus} + 1)\big)_{(\lambda_0, \alpha_0), \, \ldots \,}$ that coincide with $t (\ol{r})$ below $(\kappa^{\plus}, i_0^\prime), \, \ldots \,, (\kappa^{\plus}, i_{n-1}^\prime)$, where the function $N(p)$
%, assigning numbers $(\alpha, i)$ to all vertices $(\alpha, z)$ in its domain, 
is now defined on \[ \{(\lambda_l, z) \in t(p)\ | \ l < m, \lambda_l \leq \kappa\} \, \cup \, \{(\alpha, z) \in t(\ol{r})\ | \ \alpha \leq \kappa^\plus\}.\] 

First, we define $T(\kappa^\plus, \beta)^{t(\ol{r})} \subseteq T(\kappa^\plus, \beta)$ as follows: The condition $t(\ol{r})$ induces on any level $\alpha \leq \kappa^{\plus}$ an equivalence relation $\sim_\alpha^{t(\ol{r})}$ on $\{i_0^\prime, \, \ldots \,, i_{n-1}^\prime\}$ by setting $i_l^\prime \sim_\alpha^{t(\ol{r})} i_{\ol{l}}^\prime$ iff $(\kappa^{\plus}, i_l^\prime)$ and $(\kappa^{\plus}, i_{\ol{l}}^\prime)$ have a common $t(\ol{r})$-predecessor on level $\alpha$. \\ Thus, let $(t, \leq_t) \in T(\kappa^\plus, \beta)^{t(\ol{r})}$ iff $(t, \leq_t) \in T(\kappa^{\plus}, \beta)$ with partitions $(A_\alpha\ | \ \alpha \in \Card, \alpha \leq \kappa^\plus)$ as in the definition of $T(\kappa^\plus, \beta)$, such that $\{(\kappa^\plus, i_0^\prime), \ldots, (\kappa^\plus, i_{n-1}^\prime)\} \subseteq supp \, t$, and for any level $\alpha \leq \kappa^\plus$, the partition $A_\alpha$ coincides with $\sim_\alpha^{t(\ol{r})}$, i.e. for all $l, \ol{l} < n$, we have $i_l^\prime \sim_\alpha^{t(\ol{r})} i_{\ol{l}}^\prime$ iff there exists $z \in A_\alpha$ with $\{i_l^\prime, i_{\ol{l}}^\prime\} \subseteq z$. \\ In other words, we want the tree structure of $t$ below $(\kappa^\plus, i_0^\prime), \ldots, (\kappa^\plus, i_{n-1}^\prime)$ coincide with the tree structure of $t(\ol{r})$. \\ The partial order $\leq_{T(\kappa^\plus, \beta)^{t(\ol{r})}}$ on $T(\kappa^\plus, \beta)^{t(\ol{r})}$ is inherited from $T(\kappa^\plus, \beta)$.\\[-3mm]

Now, any $p \in \big((\widehat{\m{P}}_0)^\beta \uhr (\kappa^{\plus} + 1)\big)_{(\lambda_0, \alpha_0), \, \ldots \,}^{\ol{r}}$ will be of the form $p: t (p) \rightarrow V$ with $t (p) \in T(\kappa^\plus, \beta)^{t(\ol{r})}$ and the values $p(\alpha, z)$ as in Definition \ref{defforc},
equipped with an indexing function $N(p)$ defined on \[ \big\{ \,(\lambda_l, z) \in t(p)\ | \ l < m, \lambda_l \leq \kappa\, \big\} \, \cup \, \big\{\, (\alpha, z)\ | \ \exists\, l < n\, \ (\alpha, z) \leq_{t(p)} (\kappa^{\plus}, \{i_l^\prime\})\, \big\}\] with the following properties: \begin{itemize}  \item For $(\alpha, z) \leq_{t (p)} (\kappa^{\plus}, \{i_l^\prime\})$ with $N(p)(\alpha, z) = (\alpha, i)$, it follows that $(\alpha, i)$ is the $t(\ol{r})$-predecessor of $(\kappa^{\plus}, i_l^\prime)$ on level $\alpha$. \item For all the $(\lambda_l, z)$ remaining, $N(p) (\lambda_l, z) \in \{(\lambda_l, i)\ | \ i < \alpha_l\}\, \cup\, \{\ast\}$ as before with \[ \big( N(\lambda_l, z) = N(\lambda_l, z^\prime)\, \wedge\, z  \neq z^\prime\big) \Rightarrow N(\lambda, z) = N(\lambda, z^\prime) = \ast.\] \end{itemize} 
The idea about extending the domain of $N(p)$ is that any $(\alpha, z) \leq_{t(p)} (\kappa^\plus, \{i_l^\prime\})$ with $N(p) (\alpha, z) = (\alpha, i)$ should correspond to the vertex $(\alpha, i) \in t(\ol{r})$. \\[-3mm]

The partial order \tbl $\leq$\tbr\, on $\big((\widehat{\m{P}}_0)^\beta \uhr (\kappa^{\plus} + 1)\big)_{(\lambda_0, \alpha_0), \, \ldots \,}^{\ol{r}}$ is defined as follows: Set $\ol{p} \leq p$ iff $t (\ol{p}) \leq t(p)$ in $T(\kappa^{\plus}, \beta)^{t(\ol{r})}$, and for all $(\alpha, z) \in t (p)$, $(\alpha, \ol{z}) \in t (\ol{p}) $ with $z \subseteq \ol{z}$, it follows that $\ol{p} (\alpha, \ol{z}) \supseteq p (\alpha, z)$, and $N(p) (\alpha, z) = N (\ol{p}) (\alpha, \ol{z})$ in the case that $(\alpha, z) \in \dom N(p)$. \\[-3mm] 

For the maximal element $\m{1}$, we have for $t(\m{1})$ a tree below $(\kappa^\plus, i_0^\prime), \ldots, (\kappa^\plus, i_{n-1}^\prime)$ with partitions $(A_\alpha\ | \ \alpha \in \Card, \alpha \leq \kappa^\plus)$ and the values $N(\m{1}) (\alpha, z)$ given by $t(\ol{r})$, and $\m{1} (\alpha, z) = \emptyset$ for all $(\alpha, z) \in t(\m{1})$. \\[-3mm]

This defines $(\m{P}_0)^\beta \uhr (\kappa^\plus + 1) := \big((\widehat{\m{P}}_0)^\beta \uhr (\kappa^{\plus} + 1)\big)_{(\lambda_0, \alpha_0), \, \ldots \,}^{\ol{r}}$. \\[-2mm]

We will now see that there is a subforcing $(\wt{\m{P}}_0)^{\ol{r}} \subseteq \m{P}_0$ dense in $\m{P}_0$ below $\ol{r}$ with a projection of forcing posets $\rho^\beta_0: (\wt{\m{P}})^{\ol{r}} \rightarrow (\m{P}_0)^\beta \uhr (\kappa^\plus + 1)$. Hence, $G_0$ induces a $V$-generic filter $(G_0)^\beta \uhr (\kappa^\plus + 1)$ on $(\m{P}_0)^\beta \uhr (\kappa^\plus + 1)$. \\[-3mm]

Generally, for a condition $p \in \m{P}_0$ with $t (p) \leq t( \ol{r})$ such that all the branches of $t (p) $ have height $\geq \kappa^{\plus}$, we can define $\rho_0^\beta (p) \in (\m{P}_0)^\beta \uhr (\kappa^\plus + 1)$
as follows: Roughly, we take all predecessors of the points $\{(\kappa^{\plus}, i) \in t (p) \ | \ i < \beta\}$ and drop the indices below level $\kappa^\plus$.
We start with defining $t := t \Big(\rho_0^\beta (p)\Big)$. Let $\supp  t := \{(\kappa^{\plus}, \beta_l)\ | \ l < k\} := \{(\kappa^{\plus}, i) \in t (p) \ | \ i < \beta\}$.
For any level $\alpha \leq \kappa^{\plus}$, the condition $p$ induces an equivalence relation $\sim_\alpha$ on $\{\beta_0, \, \ldots \,, \beta_{k-1}\}$ by setting $\beta_l \sim_\alpha \beta_{\ol{l}}$ iff $(\kappa^{\plus}, \beta_l)$ and $(\kappa^{\plus}, \beta_{\ol{l}})$ have a common $t(p)$-predecessor on level $\alpha$. We take for $t$ the sequence $(A_\alpha\ | \ \alpha \in \Card, \alpha \leq \kappa^\plus)$
of partitions such that any $A_\alpha$ corresponds to the equivalence relation $\sim_\alpha$: For any $\beta_l, \beta_{\ol{l}}$, we have $\beta_l \sim_\alpha \beta_{\ol{l}}$ iff there exists $z \in A_\alpha$ with $\{\beta_l, \beta_{\ol{l}}\} \subseteq z$.
Together with the order relation $\leq_t$ given by $(\alpha, z) \leq_t (\beta, z^\prime)$ iff $\alpha \leq \beta$ and $z \supseteq z^\prime$, this defines $t \in T(\kappa^{\plus}, \beta)$. From $t(p) \leq t(\ol{r})$ it follows that $t \in T(\kappa^{\plus}, \beta)^{t(\ol{r})}$. \\[-2mm]

The tree $t$ can be embedded into $t(p)$: Namely, a canonical map $\iota_0^\beta (p): t \hookrightarrow t(p)$ can be defined as follows. 
For $(\alpha, z) \in t$, consider $\beta_l \in z$. Let $(\alpha, i)$ denote the $t (p)$-predecessor of $(\kappa^{\plus}, \beta_l)$ on level $\alpha$. Then $(\alpha, z) \in t$
corresponds to the vertex $(\alpha, i) \in t(p)$, and we set $\iota_0^\beta (p) (\alpha, z) := (\alpha, i)$. This map is well-defined and injective, with $(\alpha, z) \leq_t (\beta, z^\prime)$ if and only if $\iota_0^\beta (p) (\alpha, z) \leq_{t(p)} \iota_0^\beta (p) (\beta, z^\prime)$. \\Hence, $(\iota_0^\beta (p) [t], \iota_0^\beta (p) [\leq_t]) \subseteq (t(p), \leq_{t(p)})$ is a subtree. \\[-3mm]

For $(\alpha, z) \in t = t \big( \rho_0^\beta(p) \big)$, we set $\big( \rho_0^\beta (p) \big) (\alpha, z) := p \big( \iota_0^\beta (p) (\alpha, z)\big)$. \\[-3mm]

It remains to define the indexing function $N := N\Big(\rho_0^\beta (p) \Big)$: For $(\alpha, z) \in t$ with $(\alpha, z) \leq_t (\kappa^{\plus}, \{i_l^\prime\})$ for some $l < n$, let $N(\alpha, z) := (\alpha, i) := \iota_0^\beta (p) (\alpha, z)$. For all $(\lambda_l, z) \in t$, $\l < m$, with  $\iota_0^\beta (p) (\lambda_l, z) = (\lambda_l, i)$, let $N (\lambda_l, z) := \iota_0^\beta (p) (\lambda_l, z) = (\lambda_l, i)$ in the case that $i < \alpha_l$, and $N (\lambda_l, z) := \ast$, else. \\[-3mm]

This defines the projection $\rho_0^\beta (p)$. \\[-3mm]

Whenever $(\alpha, z) \in t$ with $(\alpha, z) \leq_{t} (\kappa^\plus, \{i_l^\prime\})$ for some $l < n$, then $N(\alpha, z) = (\alpha, i)$ is the $t(p)$-predecessor of $(\kappa^\plus, i_l^\prime)$ on level $\alpha$. Since $t(p) \leq t(\ol{r})$ it follows that $(\alpha, i)$ is also the $t(\ol{r})$-predecessor of $(\kappa^\plus, i_l^\prime)$ on level $\alpha$. Hence, $\rho_0^\beta (p)$ is indeed a condition in $\big((\widehat{\m{P}}_0)^\beta \uhr (\kappa^\plus + 1) \big)^{\ol{r}}_{(\lambda_0, \alpha_0), \ldots} = (\m{P}_0)^\beta \uhr (\kappa^\plus + 1)$.\\[-3mm]

Let now $(\wt{\m{P}}_0)^{\ol{r}}$ 
denote the collection of all $p \in \m{P}_0$ with $t(p) \leq t(\ol{r})$ such that all branches of $p$ have height at least $\kappa^{\plus}$, and the following additional property holds:

\begin{itemize} \item[(1)] For $l < m$, every $(\lambda_l, k) \in t(p)$ with $k < \alpha_l$ has a $t(p)$-successor $(\kappa^{\plus}, i)$ with $i < \beta$. \end{itemize}

Then $\big((\wt{\m{P}}_0)^{\ol{r}}, (\wt{\leq}_0)^{\ol{r}}\big)$ is a forcing with the partial order $(\wt{\leq}_0)^{\ol{r}}$ induced by $\leq_0$ and maximal element $\m{1}: t(\ol{r}) \rightarrow V$ with $\m{1} (\alpha, i) = \emptyset$ for all $(\alpha, i) \in t(\ol{r})$. \\[-3mm]

Since $(\wt{\m{P}}_0)^{\ol{r}}$ is dense in $\m{P}_0$ below $\ol{r}$,
it follows that $(\wt{G}_0)^{\ol{r}} := \{p \in (\wt{\m{P}}_0)^{\ol{r}}\ | \ p \in G_0\}$ is a $V$-generic filter on $(\wt{\m{P}}_0)^{\ol{r}}$.

\begin{prop} The map $\rho_0^\beta: (\wt{\m{P}}_0)^{\ol{r}} \rightarrow (\m{P}_0)^\beta \uhr (\kappa^\plus + 1)$, $p \mapsto \rho_0^\beta (p)$ 
is a projection of forcing posets.
In particular, \[(G_0)^\beta \uhr (\kappa^\plus + 1) := \rho_0^\beta \big[(\wt{G}_0)^{\ol{r}}\big] = \big\{\rho_0^\beta (p)\ | \ p \in (\wt{\m{P}}_0)^{\ol{r}} \, \cap \, G_0 \big\}\] 

is a $V$-generic filter on $(\m{P}_0)^\beta \uhr (\kappa^\plus + 1)$. \end{prop}

The latter will be important, since we want to work with models of the form $V[(G_0)^\beta \uhr (\kappa^\plus + 1)]$ as intermediate generic extensions to capture parts of the map $S^\beta$.

\begin{proof} It is not difficult to see that $\rho_0^\beta$ is order-preserving and surjective with $\rho_0^\beta (\m{1}) = \m{1}$. 

In order to show that  $\rho_0^\beta$ is a projection of forcing posets, it remains to verify the following property: {\textit{For any $p \in (\wt{\m{P}}_0)^{\ol{r}}$ and $q \in (\m{P}_0)^\beta \uhr (\kappa^\plus + 1)$ with $q \leq \rho_0^\beta (p)$, there exists $s \in (\wt{\m{P}}_0)^{\ol{r}}$, $s \leq p$ with $\rho_0^\beta (s) \leq q$.}} \\ Then it follows that  
$(G_0)^\beta \uhr (\kappa^\plus + 1)$ hits any open dense set $D \subseteq (\m{P}_0)^\beta \uhr (\kappa^\plus + 1)$.\\[-3mm]

Let $p \in (\wt{\m{P}}_0)^{\ol{r}}$ and $q \leq \rho_0^\beta (p)$ as above. First, we construct a condition $\ol{q} \in (\wt{\m{P}}_0)^{\ol{r}}$ compatible with $p$ such that $\rho_0^\beta (\ol{q}) = q$. We do not change the tree structure of $q$, 
but give any vertex $(\alpha, z) \in t(q)$ an index $\ol{N} (q) (\alpha, z) = (\alpha, i)$, where $\ol{N} (q)$ should extend the following indexing functions $N_{\kappa^\plus} (q)$, $N^\prime (q)$ and $N_p (q)$:

\begin{itemize} \item $N_{\kappa^\plus} (q)$ maps any $(\kappa^\plus, \{i\}) \in t(q)$ to the number $(\kappa^\plus, i)$,
\item  
$N^{\prime} (q)$ is the restriction of $N(q)$ to the set of all $(\lambda_l, z) \in t(q)$, $\lambda_l \leq \kappa$, with $N(q) (\lambda_l, z) \neq \ast$,  
\item $N_p (q)$ maps any $(\alpha, \ol{z}) \in t(q)$ which corresponds to a vertex $(\alpha, z)\in t = t \big( \rho_0^\beta (p) \big)$ to the number $(\alpha, i)$ that $(\alpha, z)$ inherits from $t(p)$. \\ More precisely:
Since $q \leq \rho_0^\beta (p)$, 
there is an embedding $\iota: (t, \leq_t) \hookrightarrow $ $(t(q), \\ \leq_{t(q)})$ such that for all $(\alpha, z) \in t$, it follows that $\iota (\alpha, z) = (\alpha, \ol{z})$ for some $\ol{z} \supseteq z$. For any $(\alpha, \ol{z}) \in \im\, \iota$ with $(\alpha, \ol{z}) = \iota (\alpha, z)$, let $N_p (q) (\alpha, \ol{z})$ be the number $(\alpha, i)$ of the $t(p)$-vertex corresponding to $(\alpha, z)$:
With our canonical map $\iota_0^\beta (p): t \hookrightarrow t(p)$ with $\iota^\beta (p) (\alpha, z) = (\alpha, i)$, set $N_p (q) (\alpha, \ol{z}) := (\alpha, i) = \iota^\beta (p) (\iota^{-1} (\alpha, \ol{z})\big)$.

\end{itemize}

It is not difficult to see that $N_{\kappa^{\plus}} (q)\, \cup \, N^{\prime} (q)\, \cup \, N_p (q)$ is well-defined and injective. \\ Since $t(p) \leq t(\ol{r})$, it follows that for any $(\alpha, \ol{z}) \in t(q)$ with $(\alpha, \ol{z}) \leq_{t(q)} (\kappa^\plus, \{i_l^\prime\})$ for some $l < n$, we have $N_p (q) (\alpha, \ol{z}) = (\alpha, i)$, where $(\alpha, i)$ is the $t(\ol{r})$-predecessor of $(\kappa^\plus, i_l^\prime)$ on level $\alpha$. \\[-3mm]

It remains to define $\ol{N} (q) (\alpha, z)$ for those $(\alpha, z) \in t(q)$ remaining with $(\alpha, z) \notin \dom \, \big(N_{\kappa^\plus} (q)\, \cup \, N^{\prime} (q)\, \cup \, N_p (q) \big)$. \\[-3mm]

For $\alpha < \kappa^{\plus}$, $\alpha \notin \{\lambda_l\ | \ l < m\}$, let \[Z_\alpha := \big \{ \, (\alpha, i)\ | \ i < F_{\lim} (\alpha)\, , \, (\alpha, i) \notin t(p) \; \cup \; \im \, \big( N_{\kappa^{\plus}} (q)\, \cup\, N^{\prime} (q)\, \cup\, N_p (q) \big)\, \big \}.\] For $l < m$ with $\lambda_l \leq \kappa$, let \[Z_{\lambda_l} := \big\{ \, (\lambda_l, i)\ | \ i \in [\alpha_l, F_{\lim} (\lambda_l))\, , \, (\lambda_l, i) \notin t(p) \; \cup \; \im\, \big(N_{\kappa^{\plus}} (q)\, \cup\, N^{\prime} (q) \, \cup\, N_p (q) \big)\, \big \}. \] 

We take for $\ol{N} (q): t(q) \rightarrow V$ an injective function with $\ol{N} (q) \supseteq N_{\kappa^{\plus}} (q)\, \cup \, N^{\prime} (q)\, \cup \, N_p (q)$ such that $\ol{N} (q) (\alpha, z) \in Z_{\alpha}$ for all $(\alpha, z) \in t(q) \, \setminus\, \dom \big(N_{\kappa^{\plus}} (q)\, \cup \, N^{\prime} (q)\, \cup \, N_p (q)\big)$. \\ The condition $\ol{q} \in (\wt{\m{P}}_0)^{\ol{r}}$ is defined as follows: $t (\ol{q}) := \{ \ol{N} (q) (\alpha, z)\ | \ (\alpha, z) \in t (q)\}$, with $\leq_{t(\ol{q})} \, := \, \{\big(\ol{N} (q) (\alpha, z), \ol{N} (q) (\beta, z^\prime)\big)\ | \ (\alpha, z) \leq_{t(q)} (\beta, z^\prime)\}$.\\ For any $(\alpha, i) = \ol{N} (q) (\alpha, z) \in t(\ol{q})$, let $\ol{q} (\alpha, i) := q (\alpha, z)$. \\ This finishes the construction of $\ol{q}$.\\[-3mm]

By construction, it follows that $\rho_0^\beta (\ol{q}) = q$. Also, $\ol{q} \, \| \, p$: Firstly, for any $(\kappa^{\plus}, j) \in t(p)$ with $j < \beta$, it follows by construction of $N_p(q)$ that the $t(p)$-branch below $(\kappa^{\plus}, j)$ coincides with the $t(\ol{q})$-branch below $(\kappa^{\plus}, j)$. \\
On the other hand, the set of all $(\alpha, i) \in t(p)$ which have no successor $(\kappa^\plus, j)$ with $j < \beta$ is disjoint from $t(\ol{q})$: The sets $Z_\alpha$ and $Z_{\lambda_l}$ are disjoint from $t(p)$ by construction, so $\ol{N} (q) (\alpha, \ol{z}) = (\alpha, i) \in t(p)$ would imply $(\alpha, i) \in \im \big( N_{\kappa^\plus} (q)\, \cup \, N^\prime (q)\, \cup\, N_p (q) \big)$. But any $(\alpha, i) \in \im N_{\kappa^\plus} (q)\, \cup\, \im N_p (q)$ clearly \textit{has} a $t(p)$-successor $(\kappa^\plus, j)$ with 
$j < \beta$,
so the only possibility remaining is that $(\alpha, i) = (\lambda_l, i) = N^\prime (q) (\lambda_l, z) = N(q) (\lambda_l, z)$ for some $l < m$ with $i < \alpha_l$.
But then it follows from property $(1)$ for $(\wt{\m{P}}_0)^{\ol{r}}$ that again, $(\lambda_l, i)$ has a $t(p)$-successor $(\kappa^\plus, j)$ with $j < \beta$ -- contradiction. \\ 
For any $(\alpha, i) = \ol{N} (q) (\alpha, \ol{z}) \in t(\ol{q})\, \cap \, t(p)$, we have $(\alpha, i) = N_p (q) (\alpha, \ol{z})$, and with the embedding $\iota: (t, \leq_t) \hookrightarrow (t(q), \leq_{t(q)})$ 
as in the definition of $N_p (q)$ with $\iota (\alpha, z) = (\alpha, \ol{z})$, it follows from $q \leq \rho_0^\beta (p)$ that $\ol{q} (\alpha, i) = q (\alpha, \ol{z}) \supseteq \rho_0^\beta (p) (\alpha, z) = p (\alpha, i)$. \\ Hence, $\ol{q}\, \| \, p$. \\[-3mm]

Setting $s := p\, \cup\, \ol{q}$, it follows that $s \leq p$ with $s \in (\wt{\m{P}}_0)^{\ol{r}}$ and $\rho_0^\beta (s) \leq \rho_0^\beta (\ol{q}) = q$. \\ Hence, the condition $s$ has all the desired properties, and it follows that $\rho_0^\beta$ is indeed a projection of forcing posets. 

\end{proof}

For capturing $S^\beta$, we will consider the product forcing \[(\m{P}_0)^\beta \uhr (\kappa^\plus + 1)\, \times \, \big(\m{P}_0 \uhr t(\ol{r}) \big) \uhr [\kappa^{\plus}, \infty).\]  Then also the map $\ol{\rho}_0^\beta: (\wt{\m{P}}_0 )^{\ol{r}} \rightarrow (\m{P}_0)^\beta \uhr (\kappa^\plus + 1)\, \times \, \big(\m{P}_0 \uhr t(\ol{r}) \big) \uhr [\kappa^{\plus}, \infty)$, which maps a condition $p \in (\wt{\m{P}}_0 )^{\ol{r}}$ to
$\big( \, \rho_0^\beta (p), (p \uhr t(\ol{r}) ) \uhr [\kappa^{\plus}, \infty)\, \big)$ is a projection of forcing posets; hence, 
$(G_0)^\beta \uhr (\kappa^\plus + 1)\, \times\, \big(G_0 \uhr t(\ol{r}) \big)\, \uhr [\kappa^{\plus}, \infty)$ is a $V$-generic filter on $(\m{P}_0)^\beta \uhr (\kappa^\plus + 1)\, \times \, \big(\m{P}_0 \uhr t(\ol{r}) \big) \uhr [\kappa^{\plus}, \infty)$.  \\[-2mm]

Now, we turn to $(\m{P}_1)^\beta \uhr (\kappa + 1)$. As already mentioned, we take for any $\lambda^\plus \in Succ^\prime\, \cap\, \kappa$ at stage $\lambda^\plus$ the forcing $Fn \big( [\lambda, \lambda^\plus) \times \min\{\beta, F(\lambda^\plus)\}, 2, \lambda^\plus\big)$ instead of $Fn \big( [\lambda, \lambda^\plus) \times F(\lambda^\plus), 2, \lambda^\plus\big)$.

More precisely, $(\m{P}_1)^\beta \uhr (\kappa + 1)$ consists of all conditions $p: Succ^\prime \, \cap\, (\kappa + 1) \rightarrow V$ with $\supp p := \{\lambda^\plus < \kappa\ | \ p(\lambda^\plus) \neq \emptyset\}$ finite such that for all $\lambda^\plus \in \supp p$, \[p(\lambda^\plus) \in Fn \big( [\lambda, \lambda^\plus) \times \min \{F(\beta, \lambda^\plus)\}, 2, \lambda^\plus \big)\] with $\dom p$ rectangular, i.e. \[\dom p(\lambda^\plus) = \dom_x\,  p(\lambda^\plus) \times \dom_y\, p (\lambda^\plus)\] for some $dom_x \, p(\lambda^\plus) \subseteq [\lambda, \lambda^\plus)$ and $dom_y \, p(\lambda^\plus) \subseteq \min \{\beta, F(\lambda^\plus)\}$. The partial order \tbl$\leq$\tbr\, is reverse inclusion, and the maximal element $\m{1}$ is the empty condition. \\[-3mm]

For $p \in \m{P}_1$, we can define a projection $\rho^\beta_1 (p)$ as follows: $\supp \rho_1^\beta (p) := \supp p \, \cap \, (\kappa + 1)$, and for any $\lambda^\plus < \kappa$ with $\lambda^{\plus} \in \supp p$, \[\dom \big(\rho_1^\beta (p) \big) (\lambda^{\plus}) := \dom_x p (\lambda^{\plus})\, \times\, (\dom_y p (\lambda^\plus)\, \cap\, \beta),\] with $\big(\rho_1^\beta (p) \big) (\lambda^{\plus}) (\zeta, i) = p (\lambda^{\plus}) (\zeta, i)$ for all $(\zeta, i) \in \dom \big(\rho_1^\beta (p) \big) (\lambda^{\plus})$. \\[-3mm]

It is not difficult to check that $\rho_1^\beta$ is indeed a projection from $\m{P}_1$ onto $(\m{P}_1)^\beta \uhr (\kappa + 1)$. Hence, \[(G_1)^\beta \uhr (\kappa + 1) := \{ \rho_1^\beta (p) \ | \ p \in G_1\}\] is a $V$-generic filter on $(\m{P}_1)^\beta \uhr (\kappa + 1)$. \\[-3mm]

For capturing $S^\beta$, we will work with the forcing \[\big((\m{P}_1)^\beta \uhr (\kappa + 1)\big) \, \times \, \m{P}_1 (\kappa^{\plus}) \, \times \, \m{P}_1 \uhr \{(\ol{\kappa}_l, \ol{\imath}_l)\ | \ l < \ol{n},\,  \ol{\kappa}_l > \kappa^{\plus} \}. \] The map $\ol{\rho}_1^\beta: \m{P}_1 \rightarrow \big((\m{P}_1)^\beta \uhr (\kappa + 1)\big) \, \times \, \m{P}_1 (\kappa^{\plus}) \, \times \, \m{P}_1 \uhr \{(\ol{\kappa}_l, \ol{\imath}_l)\ | \ l < \ol{n},\,  \ol{\kappa}_l > \kappa^{\plus} \}$ that maps a condition $p \in \m{P}_1$ to $\big( \rho_1^\beta (p), p_1 (\kappa^\plus), p_1 \uhr \{ (\ol{\kappa}_l, \ol{\imath}_l)\ | \ l < \ol{n}, \ol{\kappa}_l > \kappa^\plus\}\big)$ is a projection of forcing posets, as well.
Hence, it follows that \[(G_1)^\beta \uhr (\kappa + 1) \times G_1 (\kappa^{\plus}) \times G_1 \uhr \{(\ol{\kappa}_l, \ol{\imath}_l)\ | \ l < \ol{n}, \, \ol{\kappa}_l > \kappa^{\plus}\}\] is a $V[G_0]$-generic filter on $\big((\m{P}_1)^\beta \uhr (\kappa + 1)\big) \times \m{P}_1 (\kappa^{\plus}) \times \m{P}_1 \uhr \{(\ol{\kappa}_l, \ol{\imath}_l)\ | \ l < \ol{n}, \, \ol{\kappa}_l > \kappa^{\plus}\}$. \\
In particular, \[V \big[ (G_0)^\beta \uhr (\kappa^\plus + 1)\, \times\, \big(G_0 \uhr t(\ol{r}) \big) \uhr [\kappa^{\plus},\infty)\, \times \, (G_1)^\beta \uhr (\kappa + 1)\, \times \, G_1 (\kappa^{\plus}) \, \times\] \[ \times\,  G_1 \uhr \{(\ol{\kappa}_l, \ol{\imath}_l)\ | \ l < \ol{n},\, \ol{\kappa}_l > \kappa^{\plus}\} \big]\] is a well-defined generic extension by the forcing \[(\m{P}_0)^\beta \uhr (\kappa^\plus + 1)\, \times\, \big(\m{P}_0 \uhr t(\ol{r}) \big) \uhr [\kappa^{\plus}, \infty)\, \times \, (\m{P}_1)^\beta \uhr (\kappa + 1) \, \times \, \m{P}_1 (\kappa^{\plus}) \, \times \, \]\[ \times\, \m{P}_1 \uhr \{(\ol{\kappa}_l, \ol{\imath}_l)\ | \ l < \ol{n}, \, \ol{\kappa}_l > \kappa^{\plus}\}.\]

\begin{lem} \label{prescardF(k)} \[(\m{P}_0)^\beta \uhr (\kappa^\plus + 1)\, \times\, \big(\m{P}_0 \uhr t(\ol{r}) \big) \uhr [\kappa^{\plus}, \infty)\, \times \, (\m{P}_1)^\beta \uhr (\kappa + 1) \, \times \, \m{P}_1 (\kappa^{\plus}) \, \times \, \]\[ \times \,\m{P}_1 \uhr \{(\ol{\kappa}_l, \ol{\imath}_l)\ | \ l < \ol{n}, \, \ol{\kappa}_l > \kappa^{\plus}\}\] preserves cardinals $\geq F(\kappa)$.

\end{lem} 
\begin{proof} First, it is not difficult to see that the forcing $(\m{P}_0)^\beta \uhr (\kappa^\plus + 1)$ has cardinality $\leq |\beta| < F(\kappa)$ (one has to use that $\beta$ is {\textit{large enough}}, which implies that $\beta > \alpha_l$ for all $l < m$ with $\lambda_l \leq \kappa^\plus$). \\[-3mm]

Concerning $(\m{P}_1)^\beta \uhr (\kappa + 1)$, we have several cases to distinguish: If $|\beta|^\plus < F(\kappa)$, then $|(\m{P}_1)^\beta \uhr (\kappa + 1)| \leq |\beta|^\plus < F(\kappa)$. For the rest of the proof, assume $|\beta|^\plus = F(\kappa)$. \begin{itemize} \item If the class $Succ^\prime$ has no maximal element below $\kappa$, it follows that $F(\lambda^\plus) < |\beta|$ for all $\lambda^\plus < \kappa$ with $\lambda^\plus \in Succ^\prime$, since
$F(\lambda^\plus) < F(\mu^\plus)$ for all $\lambda^\plus, \mu^\plus \in Succ^\prime$ with $\lambda^\plus < \mu^\plus$. Hence, all the blocks $Fn \big( [\lambda, \lambda^\plus) \times F(\lambda^\plus), 2, \lambda^\plus\big)$ in $(\m{P}_1)^\beta \uhr (\kappa + 1)$ have cardinality $\leq |\beta|$; so $|(\m{P}_1)^\beta \uhr (\kappa + 1)| < F(\kappa)$. \end{itemize} It remains to consider the case that $Succ^\prime$ has a maximal element $\mu^\plus$ below $\kappa$. \\ 
Now, we have to treat the block $(\m{P}_1)^\beta (\mu^\plus) = Fn \big( [\mu, \mu^\plus) \times \min\{ F(\mu^\plus), \beta\}, 2, \mu^\plus\big)$ separately and consider the forcing $(\m{P}_1)^\beta \uhr (\mu + 1)$. \begin{itemize} \item In the case that $F(\mu^\plus) \leq |\beta|$ or \tbl \textit{$F(\mu^\plus) = F(\kappa) = |\beta|^\plus$ and the class $Succ^\prime$ has no maximal element below $\mu^{\plus}$}\tbr, it follows that $|(\m{P}_1)^\beta \uhr (\mu + 1)| < F(\kappa)$ similarly as before.

\item Finally, if $F(\mu^\plus) = F(\kappa) = |\beta|^\plus$ and $Succ^\prime$ {\textit{has}} a maximal element $\nu^\plus$ below $\mu^{\plus}$, we have to treat the product $(\m{P}_1)^\beta (\nu ^\plus) \times (\m{P}_1)^\beta (\mu^\plus)$ separately. Since $F(\nu^\plus) \leq |\beta|$, it follows that $F(\lambda^\plus) < |\beta|$ for all $\lambda^\plus \in Succ^\prime$ with $\lambda^\plus < \nu^\plus$; hence, $| (\m{P}_1)^\beta \uhr (\nu + 1)| \leq |\beta| < F(\kappa)$. \end{itemize} 
For the rest of the proof, we restrict to the latter case with $(\m{P}_1)^\beta \uhr (\kappa + 1)\cong ((\m{P}_1)^\beta \uhr (\nu + 1)) \times (\m{P}_1)^\beta (\nu^\plus) \times (\m{P}_1)^\beta (\mu^\plus)$ and $|(\m{P}_1)^\beta \uhr (\nu + 1)| < F(\kappa)$ -- the other cases can be treated similarly. \\[-2mm]

Consider the product forcing \[(\m{P}_0)^\beta \uhr (\kappa^\plus + 1)\, \times\, \big(\m{P}_0 \uhr t(\ol{r}) \big) \uhr [\kappa^{\plus}, \infty)\, \times\, (\m{P}_1)^\beta \uhr (\kappa + 1)\, \times \, \m{P}_1 (\kappa^{\plus}) \, \times \, \]\[ \times\, \m{P}_1 \uhr \{(\ol{\kappa}_l, \ol{\imath}_l)\ | \ l < \ol{n}, \, \ol{\kappa}_l > \kappa^{\plus}\}.\] Similarly as in Proposition \ref{prescard}, it follows that the \tbl upper part\tbr\, \[\big( \m{P}_0 \uhr t(\ol{r}) \big) \uhr [\kappa^{\plus}, \infty)\, \times \m{P}_1 \uhr \{(\ol{\kappa}_l, \ol{\imath}_l)\ | \ l < \ol{n}, \, \ol{\kappa}_l > \kappa^{\plus} \}\] preserves cardinals. Since this forcing is also $\leq \kappa^{\plus}$-closed, it follows that the \tbl lower part\tbr, namely, \[(\m{P}_0)^\beta \uhr (\kappa^\plus + 1)\, \times (\m{P}_1)^\beta \uhr (\kappa + 1)\, \times \, \m{P}_1 (\kappa^{\plus}),\] is the same forcing in a $\big(\m{P}_0 \uhr t(\ol{r}) \big) \uhr [\kappa^{\plus} ,\infty)\, \times \, \m{P}_1 \uhr \{(\ol{\kappa}_l, \ol{\imath}_l)\ | \ l < \ol{n}, \, \ol{\kappa}_l > \kappa^{\plus}\}$-generic extension as it is in $V$. \\[-3mm] 

Thus, it suffices to show that $(\m{P}_0)^\beta \uhr (\kappa^\plus + 1) \times (\m{P}_1)^\beta \uhr (\kappa + 1)\times \m{P}_1 (\kappa^{\plus})$ preserves cardinals $\geq F(\kappa)$. We factor \[(\m{P}_0)^\beta \uhr (\kappa^\plus + 1) \times (\m{P}_1)^\beta \uhr (\kappa + 1)\times \m{P}_1 (\kappa^{\plus})\ \cong \]\[ \cong \Big(\; (\m{P}_0)^\beta \uhr (\kappa^\plus + 1)\, \times \, (\m{P}_1)^\beta \uhr (\nu + 1) \Big)\; \times \; \Big( (\m{P}_1)^\beta (\nu^\plus) \times (\m{P}_1)^\beta (\mu^\plus) \times \m{P}_1 (\kappa^{\plus}) \Big). \] 

The product $(\m{P}_1)^\beta (\nu^\plus) \times (\m{P}_1)^\beta (\mu^\plus) \times \m{P}_1 (\kappa^{\plus})$ preserves all cardinals. Secondly, as we have argued before, the forcing $(\m{P}_0)^\beta \uhr (\kappa^\plus + 1) \times (\m{P}_1)^\beta \uhr (\nu + 1)$ has cardinality $< F(\kappa)$ (in $V$ and hence, also in any $(\m{P}_1)^\beta (\nu^\plus) \times (\m{P}_1)^\beta (\mu^\plus) \times \m{P}_1 (\kappa^{\plus})$-generic extension). Hence, the product forcing $(\m{P}_0)^\beta \uhr (\kappa^\plus + 1) \times (\m{P}_1)^\beta \uhr (\kappa + 1) \times \m{P}_1 (\kappa^{\plus})$ preserves cardinals $\geq F(\kappa)$, which finishes the proof.
\end{proof}

We want to show by an isomorphism argument that our surjection $S^\beta: \dom S^\beta \rightarrow F(\kappa)$ is contained in \[V[ (G_0)^\beta \uhr (\kappa^\plus + 1)\, \times\, \big(G_0 \uhr t(\ol{r}) \big) \uhr [\kappa^{\plus}, \infty)\, \times \, (G_1)^\beta \uhr (\kappa + 1)\, \times \, G_1 (\kappa^{\plus}) \, \times \, \]\[ \times \, G_1 \uhr \{(\ol{\kappa}_l, \ol{\imath}_l)\ | \ l < \ol{n}, \, \ol{\kappa}_l > \kappa^{\plus}\}].\] 

Also, we will see that in $V[ (G_0)^\beta \uhr (\kappa^\plus + 1)\, \times\, \big( G_0 \uhr t(\ol{r}) \big) \uhr [\kappa^{\plus}, \infty)\, \times \, (G_1)^\beta \uhr (\kappa + 1)\,  \times \, G_1 (\kappa^{\plus}) \, \times \, G_1 \uhr \{(\ol{\kappa}_l, \ol{\imath}_l)\ | \ l < \ol{n}, \, \ol{\kappa}_l > \kappa^{\plus}\}]$, there is also an injection $\iota: \dom S^\beta \hookrightarrow \beta$. Together with Lemma \ref{prescardF(k)} this gives the desired contradiction. \\

Recall that any $X$ in the domain of $S^\beta$ is of the form \[X = \dot{X}^{G_0  \, \uhr \, t(s) \times G_1 \, \uhr \, \{(\mu_0, \ol{\jmath}_0), \, \ldots \,, (\mu_{\ol{k}-1}, \ol{\jmath}_{\ol{k}-1})\} \times G_1 (\kappa^{\plus})}\] where $s$ is a condition in $G_0 \uhr (\kappa^\plus + 1)$ and $(s, (\ol{\mu}_0, \ol{\jmath}_0), \ldots, (\ol{\mu}_{\ol{k}-1}, \ol{\jmath}_{\ol{k}-1})) \in M_\beta$, i.e. $s$ has finitely many maximal points $(\kappa^{\plus}, j_0), \, \ldots \,, (\kappa^{\plus}, j_{k-1})$ with $j_0 < \beta, \, \ldots \,, j_{k-1} < \beta$, and $\ol{k} < \omega$, $\mu_0, \, \ldots \,, \mu_{\ol{k}-1} \in \kappa\, \cap\, Succ^\prime$, $\ol{\jmath}_0 < \min \{F(\mu_0), \beta\}, \, \ldots \,, \ol{\jmath}_{\ol{k}-1} < \min \{F(\mu_{\ol{k}-1}), \beta\}$. For any such $s$, let $\ol{s} = s\, \cup\, \ol{r}$. \\[-3mm]

Since do not want to use $G_0 \uhr (\kappa^\plus + 1)$ for capturing $S^\beta$, but only $(G_0)^\beta \uhr (\kappa^\plus + 1)$, 
we would like to replace the filter \[G_0 \uhr t(s) = \{p \uhr t(s)\ | \ p \in G_0, t(p) \leq t(s) \},\] by something like \[\mbox{\tbl} \, \big((G_0)^\beta \uhr (\kappa^\plus + 1)\big) \uhr t(s) :=  \big\{\, p \uhr t(s)\ \big| \ \rho_0^\beta (p) \in (G_0)^\beta \uhr (\kappa^\plus + 1), \, t \big(\rho_0^\beta (p) \big) \leq t \big(\rho_0^\beta (\ol{s})\big)\, \big\}\tbr\] but we have to specify what we mean by $p \uhr t(s)$ if not necessarily $t(p) \leq t(s)$, but we only know that $t \big(\rho_0^\beta (p) \big) \leq t \big(\rho_0^\beta (\ol{s})\big)$, i.e. merely the tree structures of $t(p)$ and $t(s)$ agree below the vertices $(\kappa^{\plus}, j) \in t(s)$. \\[-3mm] 

We will have $t \big(p \uhr t(s)\big) := t(s)$. For a vertex $(\alpha, m) \in t(s)$ with $t(s)$-successor $(\kappa^{\plus}, j)$, let $(\alpha, m^\prime)$ denote the $t(p)$-predecessor of $(\kappa^{\plus}, j)$ on level $\alpha$. We will set $(p \uhr t(s)) (\alpha, m) := p (\alpha, m^\prime)$.
From $t \big( \rho_0^\beta (p) \big) \leq t \big(\rho_0^\beta (\ol{s}) \big)$ it follows that this is well-defined: If $(\kappa^{\plus}, j), (\kappa^\plus, j^\prime)$ are both $t(s)$-successors of $(\alpha, m)$, then also in the tree $t(p)$, 
the vertices $(\kappa^{\plus}, j)$ and $(\kappa^{\plus}, j^\prime)$ have a common predecessor $(\alpha, m^\prime)$ on level $\alpha$. \\
In other words: The condition $p \uhr t(s)$ is  constructed from $p \uhr \{(\kappa^\plus, j)\ | \ (\kappa^\plus, j) \in t(s)\}$ by exchanging any index $(\alpha, m^\prime)$ such that $(\alpha, m^\prime) \leq_{t(p)} (\kappa^{\plus}, j)$, with $(\alpha, m)$ such that $(\alpha, m) \leq_{t(s)} (\kappa^\plus, j)$.

\begin{deflem} \label{Deflem} Let $q$ denote a condition in $\m{P}_0 \uhr (\kappa^\plus + 1)$ with maximal points $(\kappa^{\plus}, j_0), \, \ldots \,, (\kappa^{\plus}, j_{\ol{k}-1})$ such that $j_0, \, \ldots \,, j_{\ol{k}-1} < \beta$, and $q \, \| \, \ol{r}$. With $\ol{q} = q\, \cup\, \ol{r}$, assume $\rho_0^\beta (\ol{q}) \in (G_0)^\beta \uhr (\kappa^\plus + 1)$. \\ We define \[\big((G_0)^\beta \uhr (\kappa^\plus + 1)\big) \uhr t(q)\] as the set of all $p \uhr t(q)$ with $p \in \m{P}_0 \uhr (\kappa^\plus + 1)$ such that \[\rho_0^\beta (p) \in (G_0)^\beta \uhr (\kappa^\plus + 1)\] and \[t \big(\rho_0^\beta (p) \big) \leq t \big(\rho_0^\beta (\ol{q}) \big),\] with $p \uhr t(q)$ as defined before.

This is a $V$-generic filter on $\m{P}_0 \uhr t(q)$, with $\big((G_0)^\beta \uhr (\kappa^\plus + 1)\big) \uhr t(q) = G_0 \uhr t(q)$ in the case that $q \in G$. \end{deflem}

\begin{proof} More generally, for conditions $q_0, q_1 \in \m{P}_0$ with maximal points in $\{(\kappa^{\plus}, j)\ | \ j < \beta\}$ and $q_0 \, \| \, \ol{r}$, $q_1 \, \| \, \ol{r}$, let $\ol{q}_0 = q_0 \, \cup\, \ol{r}$ and $\ol{q}_1 = q_1\, \cup\, \ol{r}$ as before. If \[t \big(\rho_0^\beta (\ol{q}_0) \big) = t \big(\rho_0^\beta (\ol{q}_1) \big),\]there is the following canonical isomorphism $T (q_0, q_1): \m{P}_0 \uhr t(q_0) \rightarrow \m{P}_0 \uhr t(q_1)$: For a condition $p \in \m{P}_0 \uhr t(q_0)$ and some vertex $(\alpha, m) \in t(q_1)$, consider a $t(q_1)$-successor $(\kappa^{\plus}, j)$. Let $(\alpha, m^\prime)$ denote the according $t(q_0)$-predecessor of $(\kappa^{\plus}, j)$ on level $\alpha$. Set $\big(T(q_0, q_1) (p)\big) (\alpha, m) := p (\alpha, m^\prime)$. As argued before, it follows from $t \big(\rho_0^\beta (\ol{q}_0) \big) = t \big(\rho_0^\beta (\ol{q}_1) \big)$ that this is well-defined.

This isomorphism $T(q_0, q_1)$ extends to an isomorphism $\ol{T} (q_0, q_1): \Name (\m{P}_0 \uhr t(q_0)) \rightarrow \Name (\m{P}_0 \uhr t(q_1))$ on the name space as usual: For $\dot{Y} \in \Name (\m{P}_0 \uhr t(q_0))$, define recursively: \[\ol{T} (q_0, q_1) (\dot{Y}) := \big \{ \, \big(\, \ol{T} (q_0, q_1) (\dot{Z}), T(q_0, q_1) (p)\, \big)\ | \ (\dot{Z}, p) \in \dot{Y}\, \big\}.\] 

In the case that $t \big(\rho_0^\beta (\ol{q}_0) \big) = t \big(\rho_0^\beta (\ol{q}_1) \big)$ agrees with the generic filter $(G_0)^\beta \uhr (\kappa^\plus + 1)$, it is not difficult to check that \[\dot{Y}^{(G_0)^\beta \, \uhr \, (\kappa^\plus + 1))\;  \uhr \; t(q_0)} = \big(\, \ol{T} (q_0, q_1) \dot{Y} \, \big)^{ ((G_0)^\beta \, \uhr \, (\kappa^\plus + 1) )\; \uhr \; t(q_1)}.\] 
Hence, using canonical names for the generic filter, it follows that \[\big((G_0)^\beta \uhr (\kappa^\plus + 1) \big) \uhr t(q_1) = T (q_0, q_1) \big [\, \big((G_0)^\beta \uhr (\kappa^\plus + 1)\big) \uhr t(q_0)\, \big].\]

Now, let $q \in \m{P}_0 \uhr (\kappa^\plus + 1)$ as in the statement of this lemma, with maximal points $(\kappa^\plus, j_0), \ldots, (\kappa^\plus, j_{\ol{k}-1})$ with $j_0, \ldots, j_{\ol{k}-1} < \beta$ such that $q\, \| \, r$, and $\rho_0^\beta (\ol{q}) \in (G_0)^\beta \uhr (\kappa^\plus + 1)$ for $\ol{q} := q\, \cup \, r$. \\ Let $s \in G_0$ with the same maximal points $(\kappa^\plus, j_0), \ldots, (\kappa^\plus, j_{\ol{k}-1})$
and $\rho_0^\beta (\ol{s}) = \rho_0^\beta (\ol{q})$, where $\ol{s} := s\, \cup\, \ol{r}$ as before. Since $\big((G_0)^\beta \uhr (\kappa^\plus + 1)\big) \uhr t(s) = G_0 \uhr t(s)$ is a $V$-generic filter on $\m{P}_0 \uhr t(s)$ and $T(s, q): \m{P}_0 \uhr t(s) \rightarrow \m{P}_0 \uhr t(q)$ is an isomorphism of forcings, it follows from
\[\big((G_0)^\beta \uhr (\kappa^\plus + 1)\big) \uhr t(q) =  T (s, q) [G_0 \uhr t(s)]\] 
that $\big((G_0)^\beta \uhr (\kappa^\plus + 1)\big) \uhr t(q)$ is a $V$-generic filter on $\m{P}_0 \uhr t(q)$ as desired. \end{proof}

Now, we turn to $\m{P}_1$: For finitely many $(\mu_0, \ol{\jmath}_0), \ldots, (\mu_{\ol{k}-1}, \ol{\jmath}_{\ol{k}-1})$ with $\mu_0, \, \ldots \,, \mu_{\ol{k}-1} < \kappa$, $\ol{\jmath}_0 < \min\{ F(\mu_0), \beta\}, $ $\, \ldots \, $, $\ol{\jmath}_{\ol{k}-1} < \min \{F(\mu_{\ol{k}-1}), \beta\}$, let 
\[\big((G_1)^\beta \uhr (\kappa + 1) \big) \uhr \{(\mu_0, \ol{\jmath}_0), \, \ldots \,, (\mu_{\ol{k}-1}, \ol{\jmath}_{\ol{k}-1})\} \] denote the collection of all $p_1 \uhr \{(\mu_0, \ol{\jmath}_0), \, \ldots \,, (\mu_{\ol{k}-1},\ol{\jmath}_{\ol{k}-1})\}$ with $p_1 \in \m{P}_1 \uhr (\kappa + 1)$ such that \[(p_1)^\beta \uhr (\kappa + 1) \in (G_1)^\beta \uhr (\kappa + 1).\]

Then $\big((G_1)^\beta \uhr (\kappa + 1) \big) \uhr \{(\mu_0, \ol{\jmath}_0), \, \ldots \,, (\mu_{\ol{k}-1}, \ol{\jmath}_{\ol{k}-1})\} = G_1 \uhr \{(\mu_0, \ol{\jmath}_0), \, \ldots \,, (\mu_{\ol{k}-1}, \ol{\jmath}_{\ol{k}-1})\}$.

Thus, for any $X \in \dom S^\beta$, $X = \dot{X}^{G_0 \, \uhr \, t(s) \, \times \, G_1 \, \uhr \, \{(\mu_0, \ol{\jmath}_0), \, \ldots \,, (\mu_{\ol{k}-1}, \ol{\jmath}_{\ol{k}-1})\} \, \times \, G_1 (\kappa^{\plus})}$, it follows that \[X = \dot{X}^{\big((G_0)^\beta \, \uhr \, (\kappa^\plus + 1) \big) \, \uhr \; t(s)\,  \times \, \big((G_1)^\beta \, \uhr \, (\kappa + 1) \big)\, \uhr \, \{ (\mu_0, \ol{\jmath}_0), \, \ldots \,, (\mu_{\ol{k}-1}, \ol{\jmath}_{\ol{k}-1})\} \, \times \, G_1 (\kappa^{\plus})}.\]

This will help us prove the following proposition: 

\begin{prop} \label{sbetacontained} The restriction $S^\beta$ is contained in \[V\big[\, (G_0)^\beta \uhr (\kappa^\plus + 1)\, \times\, ( G_0 \uhr t(\ol{r}) ) \, \uhr \, [\kappa^{\plus}, \infty)\, \times \, (G_1)^\beta \, \uhr \, (\kappa + 1)\, \times \, G_1 (\kappa^\plus) \, \times \, \]\[ \times \, G_1 \uhr \{(\ol{\kappa}_l, \ol{\imath}_l)\ | \ l < \ol{n}, \ol{\kappa}_l > \kappa^\plus\}\, \big]. \] \end{prop}

\begin{proof} 

As in the proof of Proposition \ref{theta 1}, fix a cardinal $\lambda$
with $\lambda > \max\{ \kappa^\plus, \kappa_0, \, \ldots \,, \kappa_{n-1}, \lambda_0, $ $\, \ldots \,, $ $\lambda_{m-1}, \ol{\kappa}_0, \, \ldots \,, \ol{\kappa}_{\ol{n}-1}, \ol{\lambda}_0, $ $\, \ldots \,, $ $\ol{\lambda}_{\ol{m}-1}\}$ such that $\dot{S} \in \Name (\m{P} \uhr (\lambda + 1))$. Then also $\dot{S}^\beta \in \Name (\m{P} \uhr (\lambda + 1))$.

Let $(S^\beta)^\prime$ denote the collection of all \[\big(\, \dot{X}^{\big((G_0)^\beta \, \uhr \, (\kappa^\plus + 1) \big) \, \uhr \; t(q)\, \times\, \big( (G_1)^\beta \, \uhr \, (\kappa + 1)\big) \, \uhr \, \{ (\mu_0, \ol{\jmath}_0), \, \ldots \,, (\mu_{\ol{k}-1}, \ol{\jmath}_{\ol{k}-1})\}\, \times\, G_1 (\kappa^\plus)}, \alpha \, \big)\] such that \begin{itemize} \item[(i)] $q$ is a condition in $\m{P}_0 \uhr (\kappa^\plus + 1)$ where $t(q)$ has maximal points $(\kappa^\plus, j_0), \, \ldots \,, (\kappa^\plus, j_{k-1})$ with $j_0, \ldots, j_{k-1} < \beta$; moreover, $q \, \| \, \ol{r}$, and for $\ol{q} := q \, \cup\, \ol{r}$, it follows that $\rho_0^\beta (\ol{q}) \in (G_0)^\beta \uhr (\kappa^\plus + 1)$,

\item[(ii)] $\ol{k} < \omega$, $\mu_0, \, \ldots \,, \mu_{\ol{k}-1} \in \kappa\; \cap\; Succ^\prime$ and $\ol{\jmath}_0 < \min \{F(\mu_0), \beta\}, \, \ldots \,, \ol{\jmath}_{\ol{k}-1} < \min \{ F(\mu_{\ol{k}-1}), \beta\}$,  \item[(iii)] $\dot{X}$ is a name for the forcing $\m{P}_0 \uhr t(q) \, \times \, \m{P}_1 \uhr \{(\mu_0, \ol{\jmath}_0), \, \ldots \, ,\,(\mu_{\ol{k}-1}, \ol{\jmath}_{\ol{k}-1})\} \, \times \, \m{P}_1 (\kappa^\plus)$, 
\item[(iv)] there is a condition $p \in \m{P} \uhr (\lambda + 1)$ with $p_0 \in (\wt{\m{P}}_0)^{\ol{r}}$, $p_0 \leq \ol{q}$ and \begin{itemize} \item $\rho_0^\beta (p) \in (G_0)^\beta \uhr (\kappa^\plus + 1)$ \item $(p_0 \uhr t(\ol{r})) \uhr [\kappa^{\plus}, \infty) \in (G_0 \uhr t(\ol{r})) \uhr [\kappa^{\plus}, \infty)$ \item  $(p_1)^\beta \uhr (\kappa + 1) \in (G_1)^\beta \uhr (\kappa + 1)$ \item $p_1 (\kappa^\plus) \in G_1 (\kappa^\plus)$ \item $p_1 \uhr \{(\ol{\kappa}_l, \ol{\imath}_l)\ | \ l < \ol{n}, \ol{\kappa}_l > \kappa^\plus\} \in G_1 \uhr \{ (\ol{\kappa}_l, \ol{\imath}_l)\ | \ l < \ol{n}, \ol{\kappa}_l > \kappa^\plus\}$ \end{itemize} such that $p \Vdash_{\m{P} \, \uhr \, (\lambda + 1)} (\wt{\dot{X}}, \alpha) \in \dot{S}$. \end{itemize}

It suffices to show that $S^\beta = (S^\beta)^\prime$. \begin{itemize} \item[\tbl $\supseteq$\tbr:] For $(X, \alpha) \in S^\beta$, we have  $X = \dot{X}^{G_0 \, \uhr \, t(s) \times G_1 \, \uhr \, \{(\mu_0, \ol{\jmath}_0), \, \ldots \,, (\mu_{\ol{k}-1}, \ol{\jmath}_{\ol{k}-1})\} \times G_1 (\kappa^\plus)}$ for some $(s, (\mu_0, \ol{\jmath}_0), \ldots, (\mu_{\ol{k}-1}, \ol{\jmath}_{\ol{k}-1})) \in M_\beta$ with $s \in G_0 \uhr (\kappa^\plus + 1)$, where $\dot{X}$ is a name for the forcing $\m{P}_0 \uhr t(s)\, \times\, \m{P}_1 \uhr \{(\mu_0, \ol{\jmath}_0), \, \ldots \,\} \, \times\, \m{P}_1 (\kappa^\plus)$. \\[-4mm]

Then $(X, \alpha) = ({\wt{\dot{X}}}^{\;G \, \uhr \, (\lambda + 1)}, \alpha) \in \dot{S}^{G \, \uhr \, (\lambda + 1)}$, so there must be $p \in G \uhr (\lambda + 1)$, $p_0 \leq \ol{s} := s\, \cup\, r$, with $p \Vdash_{\m{P} \, \uhr \, (\lambda + 1)} (\wt{\dot{X}}, \alpha) \in \dot{S}$. 

Setting $q := s$, it follows that \[ (X, \alpha) = \big(\, \dot{X}^{\big((G_0)^\beta \, \uhr \, (\kappa^\plus + 1) \big) \, \uhr \, t(q)\, \times\, \big((G_1)^\beta \, \uhr \, (\kappa + 1)\big) \, \uhr \, \{ (\mu_0, \ol{\jmath}_0), \, \ldots \,, (\mu_{\ol{k}-1}, \ol{\jmath}_{\ol{k}-1})\}\ \times\, G_1 (\kappa^\plus)}, \alpha\, \big)\] is contained in $(S^\beta)^\prime$ as desired. 

\item[\tbl $\subseteq$\tbr\,:] Assume towards a contradiction, there was $(X, \alpha) \in (S^\beta)^\prime \setminus S^\beta$. Let \[X = \dot{X}^{\big((G_0)^\beta \, \uhr \, (\kappa^\plus + 1)\big) \, \uhr \, t(q)\, \times\, \big( (G_1)^\beta \, \uhr \, (\kappa + 1)\big) \, \uhr \, \{ (\mu_0, \ol{\jmath}_0), \, \ldots \,, (\mu_{\ol{k}-1}, \ol{\jmath}_{\ol{k}-1})\}\ \times\, G_1 (\kappa^\plus)}\] as in the definition of $(S^\beta)^\prime$ with $p \in \m{P} \uhr (\lambda + 1)$ as in (iv)
such that \[p \Vdash_{\m{P} \, \uhr \, (\lambda + 1)} (\wt{\dot{X}}, \alpha) \in \dot{S} \hspace*{2cm}  (\times).\]

Since $\rho_0^\beta (\ol{q}) \in (G_0)^\beta \uhr (\kappa^\plus + 1)$, we can take a condition $\ol{s} \in G_0 \uhr (\lambda + 1)$, $\ol{s} \in (\wt{\m{P}}_0)^{\ol{r}}$ with $\ol{s} \leq \ol{r}$ and $\rho_0^\beta (\ol{s}) = \rho_0^\beta (\ol{q})$. W.l.o.g. we can assume that $\ol{s} = s \cup \ol{r}$ for some $s \in G_0 \uhr (\kappa^\plus + 1)$ which has the same maximal points $(\kappa^\plus, j_0), \ldots, (\kappa^\plus, j_{k-1})$ as $q$. The isomorphism $T(q, s): \m{P}_0 \uhr q \rightarrow \m{P}_0 \uhr s$ from the proof of Definition / Lemma \ref{Deflem} can be extended to an isomorphism from $\m{P}_0 \uhr t(q) \, \times \, \m{P}_1 \uhr \{(\mu_0, \ol{\jmath}_0), \, \ldots \,, \ (\mu_{\ol{k}-1}, \ol{\jmath}_{\ol{k}-1})\} \, \times \, \m{P}_1 (\kappa^\plus)$ onto $\m{P}_0 \uhr t(s)\,  \times \,  \m{P}_1 \uhr \{(\mu_0, \ol{\jmath}_0), \, \ldots \,, \ (\mu_{\ol{k}-1}, \ol{\jmath}_{\ol{k}-1})\} \, \times \, \m{P}_1 (\kappa^\plus)$ that is the identity on the second and third coordinate. We will denote this extension by $T(q, s)$ as well, and consider the according isomorphism on the name space $\ol{T} (q, s): \Name \big(\m{P}_0 \uhr t(q) \times \m{P}_1 \uhr \{(\mu_0, \ol{\jmath}_0), \, \ldots \,, \ (\mu_{\ol{k}-1}, \ol{\jmath}_{\ol{k}-1})\} \times \m{P}_1 (\kappa^\plus)) \rightarrow \Name \big(\m{P}_0 \uhr t(s) \times \m{P}_1 \uhr \{(\mu_0, \ol{\jmath}_0), \, \ldots \,, \ (\mu_{\ol{k}-1}, \ol{\jmath}_{\ol{k}-1})\} \times \m{P}_1 (\kappa^\plus)\big)$. \\[-3mm]

Let $\ddot{X} := \ol{T}(q, s) \dot{X}$. Then \[X = \dot{X}^{\big((G_0)^\beta \, \uhr \, (\kappa^\plus + 1)\big) \, \uhr \; t(q)\, \times\, \big((G_1)^\beta \, \uhr \, (\kappa + 1)\big) \, \uhr \, \{ (\mu_0, \ol{\jmath}_0), \, \ldots \,, (\mu_{\ol{k}-1}, \ol{\jmath}_{\ol{k}-1})\}\, \times\, G_1 (\kappa^\plus)} = \] \[ = \ddot{X}^{\big((G_0)^\beta \, \uhr \, (\kappa^\plus + 1) \big) \, \uhr \; t(s)\, \times\, \big( (G_1)^\beta \, \uhr \, (\kappa + 1)\big) \, \uhr \, \{ (\mu_0, \ol{\jmath}_0), \, \ldots \,, (\mu_{\ol{k}-1}, \ol{\jmath}_{\ol{k}-1})\}\, \times\, G_1 (\kappa^\plus)} =\] \[ = \ddot{X}^{G_0 \, \uhr \, t(s) \, \times \, G_1 \ \uhr \, \{(\mu_0, \ol{\jmath}_0), \, \ldots \,, (\mu_{\ol{k}-1}, \ol{\jmath}_{\ol{k}-1})\, \times\, G_1 (\kappa^\plus)} = {\wt{\ddot{X}}}^{\; G},\] where as before, $\wt{\ddot{X}}$ denotes the canonical extension of $\ddot{X}$ to a $\m{P}$-name. \\[-3mm]

Since $(X, \alpha) \notin S$, there exists $p^\prime \in G \uhr (\lambda + 1)$, $p_0^\prime \in (\wt{\m{P}}_0)^{\ol{r}}$, with \[p^\prime \Vdash_{\m{P} \, \uhr \, (\lambda + 1)} (\wt{\ddot{X}}, \alpha) \notin \dot{S} \hspace*{2cm} (\times \times).\]

W.l.o.g. we can take $p^\prime_0 \leq \ol{s}$, and assure by a density argument, 
that $\rho_0^\beta (p_0^\prime) \leq \rho_0^\beta (p_0)$. \\[-3mm]

We want to construct an isomorphism $\pi: \m{P} \rightarrow \m{P}$ with the following properties:

\begin{itemize} \item[--] $\pi p \, \| \, p^\prime$ \item[--] ${\pi \wt{\dot{X}}}^{\;D_\pi} = {\wt{\ddot{X}}}^{\;D_\pi}$ 
\item[--] $\pi {\ol{\dot{S}}}^{D_\pi} = {\ol{\dot{S}}}^{D_\pi}$. \end{itemize} Together with $(\times)$ and $(\times \times)$, this gives the desired contradiction. \\[-3mm]

The third condition is satisfied if we make sure that $\pi$ is contained in the intersection $Fix_0 (\kappa_0, i_0)\, \cap\, \cdots\, \cap\, Small_0 (\lambda_0, [0, \alpha_0))\, \cap\, \cdots\,\cap\, Fix_1 (\ol{\kappa}_0, \ol{\imath}_0)\, \cap\, \cdots\, \cap\, Small_1 (\ol{\lambda}_0, [0, \ol{\alpha}_0)\, \cap\, \cdots$. 

We start with the construction of $\pi_0$. From $\rho_0^\beta (p_0^\prime) \leq \rho_0^\beta (p_0)$, it follows that the tree structures of $t(p_0)$ and $t(p_0^\prime)$ coincide below the vertices $(\kappa^\plus, i) \in t(p_0)$ with $i < \beta$. Hence, we can achieve $\pi_0 p_0\, \| \, p_0^\prime$ by changing any index $(\alpha, m)$ with $(\alpha, m) \leq_{t(p_0)} (\kappa^\plus, i)$ for some $i < \beta$, to $(\alpha, m^\prime)$, where $(\alpha, m^\prime) \leq_{t(p_0^\prime)} (\kappa^\plus, i)$, i.e. $(\alpha, m^\prime)$ is the corresponding index in the tree structure of $t(p_0^\prime)$; and outside the branches below $\{(\kappa^\plus, i) \in t(p_0)\ | \ i < \beta\}$, we make $t(\pi_0 p_0)$ and $t(p_0^\prime)$ disjoint.\\[-3mm]

Let $\HT \pi_0 := \lambda + 1$. 
For a cardinal $\alpha < \HT \pi_0$ with $\alpha \notin \{\lambda_0, \ldots, \lambda_{m-1}\}$, take for $\pi_0 (\alpha)$ a bijection on $\{ (\alpha, j) \ | \ j < F_{\lim} (\alpha)\}$ with finite support  such that the following hold:\begin{itemize} \item[--] If $(\alpha, j) \in t( \ol{r})$, 
then $\pi_0(\alpha) (\alpha, j) := (\alpha, j)$. \item[--] If $(\alpha, j)$ has a $t(p_0)$-successor $(\kappa^\plus, i)$ with $i < \beta$, it follows from $\rho_0^\beta (p_0^\prime) \leq \rho_0^\beta (p_0)$ that also $(\kappa^\plus, i) \in t(p_0^\prime)$. 
Let $\pi_0 (\alpha) (\alpha, j) := (\alpha, j^\prime)$ be the $t(p_0^\prime)$-predecessor of $(\kappa^\plus, i)$ on level $\alpha$. \item[--] 
For all the $(\alpha, j) \in t(p_0)$ remaining, $j \in [\gamma (j), \gamma (j) + \omega)$ for $\gamma (j)$ a limit ordinal, let $\pi_0 (\alpha)(\alpha, j) = (\alpha, j^\prime)$ for some $j^\prime \in [\gamma (j), \gamma (j) + \omega)$ with $(\alpha, j^\prime) \notin t(p_0)\, \cup \, t(p_0^\prime)$.

\end{itemize}

This is well-defined: If $(\alpha, j)$ has two $t(p_0)$-successors $(\kappa^\plus, i)$ and $(\kappa^\plus, i^\prime)$ with $i, i^\prime < \beta$, then it follows from $\rho_0^\beta (p_0^\prime) \leq \rho_0^\beta (p_0)$ that $(\kappa^\plus, i)$ and $(\kappa^\plus, i^\prime)$ also have the same $t(p_0^\prime)$-predecessor on level $\alpha$.
Also, if $(\alpha, j) \in t(\ol{r})$
has a $t(p_0)$-successor $(\kappa^\plus, i)$ with $i < \beta$, it follows that in $t(p_0^\prime)$, the vertex $(\kappa^\plus, i)$ has predecessor $(\alpha, j)$ as well, since $t(p_0)$ and $t(p_0^\prime)$ both extend $t(\ol{r})$. Thus, $\pi_0 (\alpha) (\alpha, j) = (\alpha, j)$. \\[-3mm]

In the case that $\alpha = \lambda_l$ for some $l < m$, we have to be careful, since we want $\pi \in Small_0 (\lambda_l, [0, \alpha_l))$. Thus, for any interval $[\gamma, \gamma + \omega) \subseteq \alpha_l$ with $\gamma$ a limit ordinal and $j \in [\gamma, \gamma + \omega)$, we have to make sure that $\pi_0 (\lambda_l) (\lambda_l, j) = (\lambda_l, j^\prime)$ such that also $j^\prime \in [\gamma, \gamma + \omega)$: 

Consider $(\lambda_l, j) \in t(p_0)$ with $t(p_0)$-successor $(\kappa^\plus, i)$ for some $i < \beta$. Let $(\lambda_l, z) \in t \big(\rho_0^\beta (p_0)\big)$ with $i \in z$, and $(\lambda_l, \ol{z}) \in t\big( \rho_0^\beta (p_0^\prime) \big)$ with $i \in \ol{z}$. \\ Since $\rho_0^\beta (p_0^\prime) \leq \rho_0^\beta (p_0)$, it follows that $\ol{z} \supseteq z$, and in the case that $j < \alpha_l$, we have $N\big( \rho_0^\beta (p_0^\prime) \big) (\lambda_l, \ol{z}) = N \big(\rho_0^\beta (p_0) \big) (\lambda_l, z) = (\lambda_l, j)$. Hence, $(\lambda_l, j)$ is also the $t(p_0^\prime)$-predecessor of $(\kappa^\plus, i)$ on level $\lambda_l$, which gives $\pi_0 (\lambda_l) (\lambda_l, j) = (\lambda_l, j)$. \\
In the case that $j \geq \alpha_l$, it follows from \[N\big( \rho_0^\beta (p^\prime) \big) \, (\lambda_l, \ol{z}) = N\big( \rho_0^\beta (p)\big) (\lambda_l, z) = \ast\] that for $(\lambda_l, j^\prime)$ denoting the $t(p_0^\prime)$-predecessor of $(\kappa^\plus, i)$ on level $\lambda_l$, i.e.
$\pi_0 (\lambda_l) (\lambda_l, j) = (\lambda_l, j^\prime)$, we have $j^\prime \geq \alpha_l$, as well. \\[-3mm]

Thus, we can make sure that for any $l < m$, the following additional property holds for $\pi_0 (\lambda_l)$:     

\begin{itemize}

\item[--] For any $(\lambda_l, j)$ with $\gamma$ a limit ordinal such that $j \in [\gamma (j), \gamma (j) + \omega) \subseteq \alpha_l$, we have $\pi_0 (\lambda_l) (\lambda_l, j) = (\lambda_l, j^\prime)$ such that $j^\prime$ is contained in the interval $[\gamma (j), \gamma (j) + \omega)$, as well.
\end{itemize}

Then $\pi_0 \in Small_0 (\lambda_0, [0, \alpha_0))\, \cap\, \cdots\, \cap\, Small_0(\lambda_{m-1}, [0, \alpha_{m-1}))$, and $\pi_0 \in Fix_0 (\kappa_0, i_0)\, \cap\, \cdots\, \cap\, Fix_0 (\kappa_{n-1}, i_{n-1})$, since $(\kappa_l, i_l) \in t(\ol{r})$ for all $l < n$. \\[-3mm]

We now have to verify that $\pi_0 p_0 \, \| \, p^\prime_0$. Firstly, on the tree $t(\ol{r})$, the conditions $p_0$ and $p^\prime_0$ coincide, and $\pi_0$ is the identity. Secondly, from $\rho_0^\beta (p_0^\prime) \leq \rho_0^\beta (p_0)$ and by construction of the map $\pi_0$, it follows that $\pi_0 p_0$ and $p_0^\prime$ agree on the branches below $\{(\kappa^\plus, i) \in t(\pi_0 p_0)\ | \ i < \beta\}$.
All the remaining $t(\pi_0 p_0)$- and $t(p_0^\prime)$-branches are disjoint, 
i.e. whenever $(\alpha, j)  \in t(p_0^\prime)\, \setminus \, t(\ol{r})$, and $(\alpha, j)$ has no $t(p_0^\prime)$-successor $(\kappa^\plus, i)$ with $i < \beta$, then $(\alpha, j) \notin t(\pi_0 p_0)$.
Hence, $\pi_0 p_0\, \| \, p_0^\prime$.

The map $\pi_1$ with $\pi_1 p_1 \, \|  \, p_1^\prime$ can be constructed as in Proposition \ref{separation}, and since $p_1^\prime \in G_1 \uhr (\lambda + 1)$ and $p$ satisfies (iv), it follows that $\pi_1 \in Fix_1 (\ol{\kappa}_0, \ol{\imath}_0)\, \cap\, \cdots\, \cap\, Fix_1 (\ol{\kappa}_{\ol{n}-1}, \ol{\imath}_{\ol{n}-1})\, \cap\, Small_1 (\ol{\lambda_0}, [0, \ol{\alpha}_0))\, \cap\, \cdots\, \cap\, Small_1 (\ol{\lambda}_{\ol{m}-1}, [0, \ol{\alpha}_{\ol{m}-1}))$ as desired. 

It remains to check that ${\pi \wt{\dot{X}}}^{\, D_\pi} = {\wt{\ddot{X}}^{\, D_\pi}}$, where $\ddot{X} := \ol{T} (q, s) \dot{X}$. \\ Firstly, $\pi_1$ is the identity on $\m{P}_1 \uhr \{(\mu_0, \ol{\jmath}_0), \ldots, (\mu_{\ol{k}-1}, \ol{\jmath}_{\ol{k}-1})\}$, since $\mu_l < \kappa$, $\ol{\jmath}_l < \beta$ for all $l < k$; so from $(p_1)^\beta \uhr (\kappa + 1) \in (G_1)^\beta \uhr (\kappa^{\plus} + 1)$, $p_1^\prime \in G_1$, it follows that $p_1$ and $p_1^\prime$ coincide on $\m{P}_1 \uhr \{ (\mu_0, \ol{\jmath}_0), \ldots, (\mu_{\ol{k}-1}, \ol{\jmath}_{\ol{k}-1})\}$. Similarly, $\pi_1$ is the identity on $\m{P}_1 (\kappa^\plus)$. \\ Now, consider $\pi_0$. Recall that any $(\alpha, j) \in t(p_0)$ with $(\alpha, j) \leq_{t(p_0)} (\kappa^\plus, i)$ for some $i < \beta$ is mapped to $(\alpha, j^\prime)$ such that $(\alpha, j^\prime)$ is the $t(p^\prime_0)$-predecessor of $(\kappa^\plus, i)$ on level $\alpha$. Since $p_0 \leq \ol{q} = q\, \cup\, \ol{r}$, $p_0^\prime \leq \ol{s} = s\, \cup\, \ol{r}$ with $\rho_0^\beta (\ol{s}) = \rho_0^\beta (\ol{q})$, it follows that any $(\alpha, j) \in t(q)$ with $(\alpha, j) \leq_{t(q)} (\kappa^\plus, i)$ for some $i < \beta$ is mapped to the corresponding $t(s)$-predecessor of $(\kappa^\plus, i)$ on level $\alpha$: $\pi_0 (\alpha) (\alpha, j) = (\alpha, j^\prime)$ with $(\alpha, j^\prime) \leq_{t(s)} (\kappa^\plus, i)$. Hence, it follows for any condition $\wt{q} \in \m{P}_0 \uhr t(q)$ that $\pi_0 \wt{q} = T(q, s) (\wt{q}) \in \m{P}_0 \uhr t(s)$. \\
Inductively, this implies ${\pi \wt{\dot{x}}}^{D_\pi} = {\wt{\ddot{x}}}^{D_\pi}$ whenever $\dot{x}$ is a name for $\m{P}_0 \uhr t(q) \times \m{P}_1 \uhr \{(\mu_0, \ol{\jmath}_0), \, \ldots \,, (\mu_{\ol{k}-1}, \ol{\jmath}_{\ol{k}-1})\} \times \m{P}_1 (\kappa^\plus)$ and $\ddot{x} := \ol{T}(q, s) \dot{x}$.

In particular, ${\pi \wt{\dot{X}}}^{\;D_\pi} = {\wt{\ddot{X}}}^{\;D_\pi}$, which finishes the proof. 

\end{itemize}
\end{proof}

Thus, we have shown that the surjection $S^\beta: \dom S^\beta \rightarrow F(\kappa)$ is contained in $V[ (G_0)^\beta \uhr (\kappa^\plus + 1)\, \times\, (G_0 \uhr t(\ol{r}) ) \uhr [\kappa^\plus, \infty) \, \times \, (G_1)^\beta \uhr (\kappa + 1) \times G_1 (\kappa^\plus) \times G_1 \uhr \{(\ol{\kappa}_l, \ol{\imath}_l)\ | \ l < \ol{n}, \ol{\kappa}_l > \kappa^\plus\}]$. We will now see that in this model, there is also an injection $\iota^\beta: \dom S^\beta \hookrightarrow \beta$. Together with Lemma \ref{prescardF(k)} this gives the desired contradiction.  

\begin{prop} \label{iotabeta} In $V[ (G_0)^\beta \uhr (\kappa^\plus + 1)\, \times\, (G_0 \uhr t(\ol{r})) \uhr [\kappa^{\plus}, \infty)\, \times \, (G_1)^\beta \uhr (\kappa + 1) \, \times \, G_1 (\kappa^\plus) \, \times \, G_1 \uhr \{(\ol{\kappa}_l, \ol{\imath}_l)\ | \ l < \ol{n}, \ol{\kappa}_l > \kappa^\plus\}]$, there is an injection $\iota^\beta: \dom S^\beta \rightarrow \beta$. \end{prop}

\begin{proof} We work inside 
$V[ (G_0)^\beta \uhr (\kappa^\plus + 1)\, \times\, (G_0 \, \uhr \, t(\ol{r}) ) \, \uhr \, [\kappa^{\plus}, \infty)\, \times \, (G_1)^\beta \uhr (\kappa + 1) \, \times \, G_1 (\kappa^\plus) \, \times \, G_1 \uhr \{(\ol{\kappa}_l, \ol{\imath}_l)\ | \ l < \ol{n}, \ol{\kappa}_l > \kappa^\plus\}] \vDash ZFC$. \\[-3mm]

Let $\wt{M}_\beta$ denote the collection of all tuples $(q, (\mu_0, j_0), \, \ldots \,, (\mu_{\ol{k}-1}, j_{\ol{k}-1})) \in M_\beta$ with the property that $q \, \| \, \ol{r}$, and for $\ol{q} = q \, \cup \,  \ol{r}$ as before, $\rho_0^\beta (\ol{q}) \in (G_0)^\beta \uhr (\kappa^\plus + 1)$. \\[-3mm]

Fix some $(q, (\mu_0, \ol{\jmath}_0), \, \ldots \,, (\mu_{\ol{k}-1}, \ol{\jmath}_{\ol{k}-1})) \in \wt{M}_\beta$. Then $\big((G_0)^\beta \uhr (\kappa^\plus + 1)\big) \uhr t(q)\, \times\, \big((G_1)^\beta \uhr (\kappa + 1)\big) \uhr \{(\mu_0, \ol{\jmath}_0), \, \ldots \,, (\mu_{\ol{k}-1}, \ol{\jmath}_{\ol{k}-1})\}\, \times\, G_1 (\kappa^\plus)$ is a $V$-generic filter on $\m{P}_0 \uhr t(q)\, \times\, \m{P}_1 \uhr \{(\mu_0, \ol{\jmath}_0), \, \ldots \,, \,  (\mu_{\ol{k}-1}, \ol{\jmath}_{\ol{k}-1})\}\, \times\, \m{P}_1 (\kappa^\plus)$.\\
From Proposition \ref{prescard} we know that the forcing $\m{P}_0 \uhr t (q)\, \times\, \m{P}_1 \uhr \{ (\mu_0, \ol{\jmath}_0), \, \ldots \,\}$ preserves cardinals and the $GCH$. By the same proof, one can show that $\m{P}_0 \uhr t (q)\, \times\, \m{P}_1 \uhr \{ (\mu_0, \ol{\jmath}_0), \, \ldots \,\} \, \times\, \m{P}_1 (\kappa^\plus)$ preserves cardinals and the $GCH$ below $\kappa^\plus$ (since $\m{P}_1 (\kappa^\plus)$ is $\leq \kappa$-closed): \\ For every $\alpha \leq \kappa$, \[ \big(2^\alpha\big)^{ V \boldsymbol{[( }(G_0)^\beta \, \uhr \, (\kappa^\plus + 1) \boldsymbol{)} \, \uhr \, t(q) \, \times\, \boldsymbol{(}(G_1)^\beta \, \uhr \, (\kappa + 1)\boldsymbol{)} \, \uhr \, \{(\mu_0, \ol{\jmath}_0), \, \ldots \,\}\, \times\, G_1 (\kappa^\plus)\boldsymbol{]}} = (\alpha^\plus)^V.\] Hence, in $V[ \big((G_0)^\beta \uhr (\kappa^\plus + 1) \big) \uhr t(q) \, \times\, \big( (G_1)^\beta \uhr (\kappa + 1)\big) \uhr \{ (\mu_0, \ol{\jmath}_0), \, \ldots \,\}\, \times\, G_1 (\kappa^\plus)]$, there is an injection $\iota: \powerset(\kappa) \hookrightarrow (\kappa^\plus)^V$. \\[-3mm]

Now, 
we can use $AC$ (in $V[ (G_0)^\beta \uhr (\kappa^\plus + 1)\, \times\, (G_0 \, \uhr \, t(\ol{r}) ) \, \uhr \, [\kappa^{\plus}, \infty)\, \times \, (G_1)^\beta \uhr (\kappa + 1) \, \times \, G_1 (\kappa^\plus) \, \times \, G_1 \uhr \{(\ol{\kappa}_l, \ol{\imath}_l)\ | \ l < \ol{n}, \ol{\kappa}_l > \kappa^\plus\}]$\big) to obtain a collection of injections $\iota_{\boldsymbol{(}q, (\mu_0, \ol{\jmath}_0), \ldots, (\mu_{k-1}, \ol{\jmath}_{\ol{k}-1}) \boldsymbol{)}}\: : \: \powerset (\kappa)\, \cap\, V[  \big((G_0)^\beta \uhr (\kappa^\plus + 1) \big) \uhr t(q) \, \times\,\big( (G_1)^\beta \uhr (\kappa + 1)\big) \uhr \{ (\mu_0, \ol{\jmath}_0), \, \ldots \,\}\, \times\, G_1 (\kappa^\plus)] \hookrightarrow (\kappa^\plus)^V$
for $(q, (\mu_0, \ol{\jmath}_0)\, , \, \ldots \, , \, (\mu_{k-1}, \ol{\jmath}_{k-1})) \in \wt{M}_\beta$. \\[-3mm]

Let $\wt{\wt{M}}_\beta$ denote the set of all tuples $((\kappa^\plus, j_0)\, , \, \ldots\, , \, (\kappa^\plus, j_{k-1}), (\mu_0, \ol{\jmath}_0)\, , \, \ldots\, , \, $ $(\mu_{\ol{k}-1}, \ol{\jmath}_{\ol{k}-1}))$ with $k, \ol{k} < \omega$ and $j_0\, , \, \ldots\, , \, j_{k-1} < \beta$, $\mu_0\, , \, \ldots\, , \, \mu_{\ol{k}-1} \in \kappa\, \cap\, Succ^\prime$, $\ol{\jmath}_0 < \min \{F(\mu_0), \beta\}\, , \, \ldots\, , \, $ $\ol{\jmath}_{\ol{k}-1} < \min \{F (\mu_{\ol{k}-1}), \beta\}$. \\[-3mm]

Let $\tau$ denote an injection that maps any tuple $((\kappa^\plus, j_0)\, , \, \ldots \, , \, (\kappa^\plus, j_{k-1}))$ with $j_0, \ldots, j_{k-1} < \beta$ as above to some condition $q \in \m{P}_0$ such that $t(q)$ has maximal points $(\kappa^\plus, j_0)\, , \, \ldots \, , \, (\kappa^\plus, j_{k-1})$, $q \, \| \, \ol{r}$, and for $\ol{q} := q \, \cup \, \ol{r}$ as before, $\rho_0^\beta (\ol{q}) \in (G_0)^\beta \uhr (\kappa^\plus + 1)$. \\[-3mm]

For any $((\kappa^\plus, j_0)\, , \, \ldots\, , \, (\kappa^\plus, j_{k-1}), (\mu_0, \ol{\jmath}_0)\, , \, \ldots\, , \, (\mu_{\ol{k}-1}, \ol{\jmath}_{\ol{k}-1})) \in \wt{\wt{M}}_\beta$, let \[\iota_{\boldsymbol{(}(\kappa^\plus, j_0)\, , \, \ldots\, , \, (\mu_0, \ol{\jmath}_0)\, , \, \ldots\boldsymbol{)}} := \iota_{ \boldsymbol{(} q, (\mu_0, \ol{\jmath}_0)\, , \, \ldots \boldsymbol{)}}, \] where $q := \tau ((\kappa^\plus, j_0), \ldots, (\kappa^\plus, j_{k-1}))$. \\[-3mm]

Any $X \in \dom S^\beta$ is of the form \[X = \dot{X}^{\big((G_0)^\beta \, \uhr \, (\kappa^\plus + 1) \big) \, \uhr \, t(q) \, \times\, \big((G_1)^\beta \, \uhr \, (\kappa + 1)\big) \, \uhr \, \{(\mu_0, \ol{\jmath}_0), \, \ldots \,\}\, \times\, G_1 (\kappa^\plus)}\] for some $\dot{X} \in \Name \big(\m{P}_0 \uhr t(q)\, \times\, \m{P}_1 \uhr \{(\mu_0, \ol{\jmath}_0), \ldots, (\mu_{\ol{k}-1}, \ol{\jmath}_{\ol{k}-1})\}\, \times\, \m{P} (\kappa^\plus) \big)$ with $(q, (\mu_0, \ol{\jmath}_0), \ldots, (\mu_{\ol{k}-1}, \ol{\jmath}_{\ol{k}-1})) \in \wt{M}_\beta$. Denote by $(\kappa^\plus, j_0), \ldots, (\kappa^\plus, j_{k-1})$ the maximal points of $t(q)$ with $\tau ( (\kappa^\plus, j_0), \ldots, (\kappa^\plus, j_{k-1})) =: q^\prime$. Then $\rho_0^\beta (\ol{q}), \rho_0^\beta (\ol{q}^\prime) \in G_0^\beta \uhr (\kappa^\plus + 1)$ with the same maximal points; hence, $\rho_0^\beta (\ol{q}) = \rho_0^\beta (\ol{q}^\prime)$. With the isomorphism $T (q, q^\prime): \m{P}_0 \uhr t(q) \rightarrow \m{P}_0 \uhr t(q^\prime)$ from Definition / Lemma \ref{Deflem} and its extension $\ol{T} (q, q^\prime): \Name (\m{P}_0 \uhr t(q) \, \times\, \m{P}_1 \uhr \{ (\mu_0, \ol{\jmath}_0), \ldots, (\mu_{\ol{k}-1}, \ol{\jmath}_{\ol{k}-1})\}\, \times\, \m{P}_1 (\kappa^\plus)) \rightarrow \Name (\m{P}_0 \uhr t(q^\prime) \, \times\, \m{P}_1 \uhr \{ (\mu_0, \ol{\jmath}_0), \ldots, (\mu_{\ol{k}-1}, \ol{\jmath}_{\ol{k}-1})\}\, \times\, \m{P}_1 (\kappa^\plus))$, 
it follows that \[X = \big(\ol{T} (q, q^\prime) \dot{X}\big)^{\big((G_0)^\beta \, \uhr \, (\kappa^\plus + 1) \big) \, \uhr \, t(q^\prime) \, \times\, \big((G_1)^\beta \, \uhr \, (\kappa + 1)\big) \, \uhr \, \{(\mu_0, \ol{\jmath}_0), \, \ldots \,\}\, \times\, G_1 (\kappa^\plus)}\] where $ \big( \ol{T} (q, q^\prime) \dot{X} \big) \in \Name (\m{P}_0 \uhr t(q^\prime) \, \times\, \m{P}_1 \uhr \{ (\mu_0, \ol{\jmath}_0), \ldots, (\mu_{\ol{k}-1}, \ol{\jmath}_{\ol{k}-1})\}\, \times\, \m{P}_1 (\kappa^\plus))$. \\[-3mm]

Hence, \[X \in \dom \iota_{ \boldsymbol{(}q^\prime, (\mu_0, \ol{\jmath}_0), \ldots\boldsymbol{)}} = \dom \iota_{\boldsymbol{(}(\kappa^\plus, j_0), \ldots, (\mu_0, \ol{\jmath}_0), \ldots\boldsymbol{)}} .\]

\vspace*{1mm}

There is a canonical bijection $b: \wt{\wt{M}}_\beta \rightarrow \beta$. Hence, 
the injections $\iota_{\boldsymbol{(}(\kappa^\plus, j_0)\, , \, \ldots \, , \, (\mu_0, \ol{\jmath}_0)\, , \, \ldots\, \boldsymbol{)}}$ for $((\kappa^\plus, j_0), \ldots, (\mu_0, \ol{\jmath}_0), \ldots) \in \wt{\wt{M}}_\beta$ can be \tbl glued together\tbr\,to an injection $\ol{\iota}: \dom S^\beta \rightarrow (\kappa^\plus)^V\, \times\, \beta$ as follows: For $X \in \dom S^\beta$, take $((\kappa^\plus, j_0), \ldots, (\mu_0, \ol{\jmath}_0), \ldots) \in \wt{\wt{M}}_\beta$ with $\delta := b ((\kappa^\plus, j_0), \ldots, (\mu_0, \ol{\jmath}_0), \ldots) < \beta$ least such that $X \in \dom \iota_{\boldsymbol{(}(\kappa^\plus, j_0), \ldots, (\mu_0, \ol{\jmath}_0), \ldots\boldsymbol{)}}$ and set \[\ol{\iota} (X) := \big(  \,\iota_{\boldsymbol{(}(\kappa^\plus, j_0)\, , \, \ldots\, , \, (\mu_0, \ol{\jmath}_0)\, , \, \ldots\,\boldsymbol{)}} (X), \delta \, \big).\]

This gives an injection $\iota: \dom S^\beta \rightarrow \beta$ in $V[ (G_0)^\beta \uhr (\kappa^\plus + 1)\, \times\, (G_0 \, \uhr \, t(\ol{r}) ) \, \uhr \, [\kappa^{\plus}, \infty)\, \times \, (G_1)^\beta \uhr (\kappa + 1) \, \times \, G_1 (\kappa^\plus) \, \times \, G_1 \uhr \{(\ol{\kappa}_l, \ol{\imath}_l)\ | \ l < \ol{n}, \ol{\kappa}_l > \kappa^\plus\}]$ as desired. 

\end{proof}

Thus, we have shown that our assumption of a surjective function $S: \powerset (\kappa) \rightarrow F(\kappa)$ in $N$ leads to a contradiction.

Hence, $\theta^N (\kappa) \leq F(\kappa)$ for any limit cardinal $\kappa$. \\[-2mm]

It remains to show that $\theta^N (\kappa^\plus) \leq F(\kappa^\plus)$ for all successor cardinals $\kappa^\plus$; which can be done by the same argument: \\[-3mm]

Like before, we assume towards a contradiction there was a surjective function $S: \powerset (\kappa^\plus) \rightarrow F(\kappa^\plus)$ in $N$, $S = \dot{S}^G$ with $\pi \ol{\dot{S}}^{D_\pi} = \ol{\dot{S}}^{D_\pi}$ for all $\pi$ that are contained in an intersection like $(A_{\dot{S}})$. Again, fix a condition $r \in G_0$ such that $\{(\kappa_0, i_0), \, \ldots \,, (\kappa_{n-1}, i_{n-1})\} \subseteq t(r)$ contains all maximal points of $t(r)$, and an extension $\ol{r} \leq_0 r$, $\ol{r} \in G_0$ such that all $t(\ol{r})$-branches have height $\geq \kappa^{\plus}$. \\
From Corollary \ref{corapprox}, it follows that any $X \in N$, $X \subseteq \kappa^\plus$, is contained in a model of the form \[V[G_0 \uhr \{ (\kappa^\plus, j_0), \ldots, (\kappa^\plus, j_{k-1})\}\, \times \, G_1 \uhr \{(\mu_0, \ol{\jmath}_0), \ldots, (\mu_{\ol{k} - 1}, \ol{\jmath}_{\ol{k} - 1})\} ]\] 
where $j_0, \ldots, j_{k-1} < F_{\lim} (\kappa^{\plus}) = F(\kappa)$ and $\mu_0, \ldots, \mu_{\ol{k}-1} \in Succ^\prime\, \cap\, (\kappa^\plus + 1)$ with $\ol{\jmath}_0 < F(\mu_0),\,  \ldots \,, \ol{\jmath}_{\ol{k}-1} < F(\mu_{\ol{k} - 1})$. \\[-3mm]

For a limit ordnal $\wt{\beta} < F(\kappa^\plus)$, our definition of \textit{large enough for $(A_{\dot{S}})$} has to be slightly modified: This time, we require that $\wt{\beta} > \ol{\imath}_l$ for all $l < \ol{n}$ with $\ol{\kappa}_l \leq \kappa^\plus$ (instead of just $\ol{\kappa}_l < \kappa$), and $\wt{\beta} > \ol{\alpha}_l$ for all $l < \ol{m}$ with $\ol{\lambda}_l \leq \kappa^\plus$ (instead of just $\ol{\lambda}_l < \kappa$). \\[-3mm]

Fix $\wt{\beta} < F(\kappa^\plus)$ \textit{large enough for} $(A_{\dot{S}})$ and $\beta := \wt{\beta} + \kappa^\plus$ (addition of ordinals). We define the restriction $S^\beta$ similarly as before:
Let $M^\prime$ denote the collection of all tuples $(s, (\mu_0, \ol{\jmath}_0)\, , \, \ldots\, , \, (\mu_{\ol{k}-1}, \ol{\jmath}_{\ol{k}-1}))$ with $\ol{k} < \omega$, $\mu_0\, , \, \ldots \, , \, \mu_{\ol{k}-1} \leq \kappa^\plus$, $\ol{\jmath}_0 < F(\mu_0)\, , \, \ldots \, , \, \ol{\jmath}_{\ol{k}-1} < F(\mu_{\ol{k}-1})$, and $s$ a condition in $\m{P}_0$ with maximal points $(\kappa^\plus, j_0)\, , \, $ $ \ldots \, , $ $\, (\kappa^\plus, j_{k-1})$ where $j_0 < F_{\lim} (\kappa^{\plus}), \, \ldots \,, j_{k-1} < F_{\lim} (\kappa^{\plus})$. Moreover, we denote by $M_\beta^\prime$ the collection of all tuples $(s, (\mu_0, \ol{\jmath}_0)\, , \, \ldots \, , \, (\mu_{\ol{k}-1}, \ol{\jmath}_{\ol{k}-1})) \in M^\prime$ with the additional property that $\ol{\jmath}_0 < \beta\, , \, \ldots\, , \, \ol{\jmath}_{\ol{k}-1} < \beta$, and $s$ has maximal points $(\kappa^\plus, j_0)\, , $ $ \, \ldots\, , $ $\, (\kappa^\plus, j_{k-1})$ with $j_0 < \beta \, , \,\ldots\, , \,j_{k-1} < \beta$.

Let \[ S^\beta := S \uhr \{ \,X \subseteq \kappa\ | \ \exists\, (s, (\mu_0, \ol{\jmath}_0), \ldots, (\mu_{\ol{k}-1}, \ol{\jmath}_{\ol{k}-1})) \in M^\prime_\beta: \ s \in G_0 \uhr (\kappa^\plus + 1), \]\[X \in V[G_0 \uhr t(s)\, \times\, G_1 \uhr \{ (\mu_0, \ol{\jmath}_0)\, , \, \ldots\, , \, (\mu_{\ol{k}-1}, \ol{\jmath}_{\ol{k}-1})\}] \, \}.\] 
The same proof as for Proposition \ref{theta 1} shows that the surjectivity of $S$ implies that $S^\beta$ must be surjective, as  well. \\[-3mm]

Now, with the same construction as before, one can capture $S^\beta$ in an intermediate model $V[ (G_0)^\beta \uhr (\kappa^\plus + 1)\, \times \,  (G_0 \uhr t(\ol{r})) \, \uhr \, [\kappa^{\plus}, \infty)\, \times \, (G_1)^\beta \, \uhr \, (\kappa^\plus + 1)\, \times \, G_1 \uhr \{ (\ol{\kappa}_l, \ol{\imath}_l)\ | \ l < \ol{n}, \ol{\kappa}_l > \kappa^\plus\}]$, and like in Lemma \ref{prescardF(k)}, one can show that the according forcing $(\m{P}_0)^\beta \uhr (\kappa^\plus + 1) \, \times \, ( \m{P}_0 \, \uhr \, t(\ol{r})) \, \uhr \, [\kappa^{\plus}, \infty)\, \times \, (\m{P}_1)^\beta\, \uhr \, (\kappa^\plus + 1) \, \times \, \m{P}_1 \uhr \{ (\ol{\kappa}_l, \ol{\imath}_l)\ | \ l < \ol{n},\, \ol{\kappa}_l > \kappa^\plus \}$ preserves cardinals $\geq F(\kappa^\plus)$. \\ Finally, one can show like in Proposition \ref{iotabeta} that in this model $V[ (G_0)^\beta \uhr (\kappa^\plus + 1)\, \times \,  (G_0 \uhr t(\ol{r})) \, \uhr \, [\kappa^{\plus}, \infty)\, \times \, (G_1)^\beta \, \uhr \, (\kappa^\plus + 1)\, \times \, G_1 \uhr \{ (\ol{\kappa}_l, \ol{\imath}_l)\ | \ l < \ol{n}, \, \ol{\kappa}_l > \kappa^\plus\}]$, there is also an injection $\iota^\beta: \dom S^\beta \hookrightarrow \beta$. This gives the desired contradiction. \\[-3mm]

Hence, it follows that $\theta^N (\kappa^\plus) \leq F(\kappa^\plus)$ for all successor cardinals $\kappa^\plus$. \\[-2mm]

Thus, our model $N$ has all the desired properties.

\section {Discussion and Remarks} One could ask whether it is possible to do a similar construction and obtain a $ZF$-model $N$ where additionally, $DC$ holds. Note that the Axiom of Dependent Choice imposes the following restrictions on the $\theta$-function: Firstly, $DC$ implies $\cf \theta (\kappa) > \omega$ for all $\kappa$. Secondly, in this particular setting where we obtain our choiceless model $N$ as a symmetric extension of $V \vDash ZFC\, \plus \, GCH$,
it follows from $\theta^N (\kappa) = \delta^\plus$ for some cardinal $\delta$
that $\cf \, \delta > \omega$. \\[-3mm]

Generally, whenever $\m{P}$ is a $<\eta$ - closed forcing in $V$ with a group $A$ of $\m{P}$-automorphisms and $\mathcal{F}$ a $<\eta $ - closed normal filter on $A$, then the according $\m{P}$-generic symmetric extension satisfies $DC_{< \eta}$ (for instance, see \cite[Lemma 1]{Karagila}).\\[-3mm]

A straightforward generalization of $\m{P}_0$ would be a forcing with trees $(t, \leq_t)$ where countably many maximal points are allowed, instead of just finitely many. \\[-4mm]

However, this gives rise to the following appearance that we call an \textit{open branch}: There might be a $\leq_t$-increasing chain of vertices $((\alpha, i_\alpha)\ | \ \alpha < \lambda)$ for some cardinal $\lambda$ of countable cofinality such that there exists no $(\lambda, i) \in t$ with $(\alpha, i_\alpha) \leq_t (\lambda, i)$ for all $\alpha < \lambda$. The number of open branches might be $2^{\aleph_0} = \aleph_1$, so we can not always \tbl close\tbr\,all of them and retain a condition in the forcing. \\[-3mm]

Let us shortly discuss the following technical problem that comes along with these open branches: If conditions $p$ and $q$ in $\m{P}_0$ agree on a subtree $t(r) \geq t(p)$, $t(q)$, it
might not be possible to 
achieve $\pi p \, \| \, q$ by a small $\m{P}_0$-automorphism $\pi$ that is the identity on $t(r)$: Consider the case that the tree $t(r)$ has an open branch $((\alpha, i_\alpha)\ | \ \alpha < \lambda)$ such that in $t(p)$, there is a vertex $(\lambda, i)$ with $(\lambda, i) \geq_{t(p)} (\alpha, i_\alpha)$ for all $\alpha < \lambda$, but in $t(q)$, there is a different vertex $(\lambda, i^\prime)$ with $i^\prime \neq i$ and $(\lambda, i^\prime) \geq_{t(q)} (\alpha, i_\alpha)$ for all $\alpha < \lambda$.
An automorphism $\pi$ with $\pi p\, \| \, q$ such that $\pi$ is the identity on this branch $((\alpha, i_\alpha)\ | \ \alpha < \lambda)$, 
has to satisfy $\pi (\lambda) (\lambda, i) = (\lambda, i^\prime)$, since the tree
$t(\pi p) \, \cup \, t(q)$ must not have a \tbl splitting\tbr\,at level $\lambda$. But then, there is no way to guarantee that $\pi$ is small, since in general, $i$ and $i^\prime$ will not be close to each other. \\[-3mm]

Hence, generalizing $\m{P}_0$ to trees with countably many maximal points makes us lose an essential homogeneity property, so several crucial arguments in the original proof do not work any more. \\[-3mm]

However, one could allow trees with $< \eta$ - many maximal points, where $\eta$ is an inaccessible cardinal.
Then our conditions in the forcing have $< \eta$ - many open branches, so we can now \tbl close\tbr\,all of them and still remain inside $\m{P}_0$. 

In this setting, we call a $\m{P}_0$-automorphism \textit{small}, if for any level $\kappa$ and $\pi (\kappa) (\kappa, i) = (\kappa, i^\prime)$, it follows that there is a ordinal $\gamma$ divisible by $\eta$ with $i, i^\prime \in [\gamma, \gamma + \eta)$. \\
Concerning $\m{P}_1$, we can use $\eta$ - support instead of finite support, and then take intersections of $< \eta$-many $Fix (\kappa, i)$ - and $Small(\lambda, [0, \alpha))$ - subgroups for generating $N$. 
Then $N \vDash DC_{< \eta}$.

But $DC_{< \eta}$ imposes further restrictions on the $\theta$-function, and one cannot use this modified forcing for setting $\theta$-values $\theta^N (\kappa)$ for cardinals $\kappa < \eta$. \\[-2mm]

In this setting with $N \vDash DC$, one could now furthermore ask whether certain large cardinal properties are preserved between $V$ and the symmetric extension $N$.

\bibliographystyle{alpha}
\nocite{*}
\bibliography{bib}

\vspace*{3mm}

{\scshape{ \footnotesize Anne Fernengel, Mathematisches Institut, Rheinische Friedrich-Wilhelms-Universit\"at, \\[-1mm] Bonn, Germany }}\\[-0,5mm]
{\rmfamily \itshape \footnotesize E-Mail address: }{ \ttfamily \footnotesize anne@math.uni-bonn.de} \\[1mm]

{\scshape{ \footnotesize Peter Koepke, Mathematisches Institut, Rheinische Friedrich-Wilhelms-Universit\"at, \\[-1mm] Bonn, Germany }} \\[-0,5mm]
{\rmfamily \itshape \footnotesize E-Mail address:} { \ttfamily \footnotesize koepke@math.uni-bonn.de}

\end{document}